\documentclass[11pt,a4paper]{article}
\usepackage[margin=3cm]{geometry}
\usepackage{graphicx}
\usepackage{amsmath, amsthm, amssymb, amsfonts}
\usepackage{bbm}
\usepackage{enumerate}
\usepackage{mathtools}
\usepackage{xcolor}
\usepackage{hyperref}
\usepackage{centernot}

\newtheorem{theorem}{Theorem}[section]
\newtheorem{remark}[theorem]{Remark}
\newtheorem{lemma}[theorem]{Lemma}
\newtheorem{proposition}[theorem]{Proposition}
\newtheorem{corollary}[theorem]{Corollary}

\newtheorem{definition}[theorem]{Definition}

\numberwithin{equation}{section}

\usepackage{authblk}

\title{Regularity of Gibbs measures for unbounded spin systems on general graphs}

\author[1]{Christoforos Panagiotis}
\author[2]{William Veitch}
\affil[1,2]{{Department of Mathematical Sciences}\\
        {University of Bath}\\
        {UK}}

\begin{document}
\date{}
\maketitle

\begin{abstract}
We consider a general class of spin systems with potentially unbounded real-valued spins, defined via a single-site potential with super-Gaussian tails on general graphs, allowing for both short- and long-range interactions. This class includes all $P(\varphi)$ models, in particular the well-studied $\varphi^4$ model. We construct an infinite-volume extremal measure called the plus measure as the limit of finite-volume Gibbs measures with weakly growing boundary conditions and show that it is regular, in the sense that it admits a bounded Radon-Nikodym derivative with respect to a product measure of single-site distributions with super-Gaussian tails. 
Moreover, we provide an alternative construction of the plus measure as the limit of finite-volume Gibbs measures that are regular up to the boundary.

As a key intermediate step, we establish regularity and tightness of finite-volume Gibbs measures for a large class of growing boundary conditions $\xi$. Our regularity estimates are encoded in terms of a function $A(\xi)$, which provides precise control on the change of measure induced by boundary perturbations, and can thus be viewed as an analogue of the Cameron–Martin theorem for non-Gaussian fields. In the nearest-neighbour case, this class includes boundary conditions that grow at most double-exponentially in the distance to the boundary when the single-site measure has tails of the form $e^{-a|u|^n}$ for some $n>2$. In contrast, when the single-site measure has Gaussian tails, the allowed growth is at most exponential. Our results apply to arbitrary graphs and improve upon earlier results of Lebowitz and Presutti \cite{Lebowitz_Presutti}, and Ruelle \cite{Ruelle1970, Ruelle_estimates}, which apply in the context of $\mathbb{Z}^d$ and allow only logarithmically growing boundary conditions, as well as subsequent extensions to vertex-transitive graphs of polynomial growth \cite{random_tangled}.
\end{abstract}

\section{Introduction}

We consider a general family of spin systems with pair interactions where each spin takes values in $\mathbb{R}$. On a finite graph $(\Lambda, E)$ with ferromagnetic nearest-neighbour interactions and free boundary conditions, this consists of probability measures on spin configurations $\varphi \in \mathbb{R}^{\Lambda}$ defined so that the expectation of a bounded measurable function $f: \mathbb{R}^{\Lambda} \rightarrow \mathbb{R}$ is given by
\begin{equation*}
    \langle f \rangle = \frac{1}{Z} \int_{\mathbb{R}^{\Lambda}} f(\varphi) \exp\left(\sum_{\{x,y\} \in E} \beta \varphi_x \varphi_y\right) \prod_{x \in \Lambda} \mathrm{d} \rho(\varphi_x),
\end{equation*}
where $\beta \geq 0$ is the \textit{inverse temperature}, $\rho$ is the single-site measure (chosen so that the above integral is finite), and $Z$ is the appropriate normalisation constant, called the \textit{partition function}. 
This framework includes several well-studied models in statistical physics:
\begin{enumerate}
    \item[$(i)$] The Ising and $\varphi^4$ models, by choosing, respectively
    \begin{equation*}
        \rho=\delta_{-1}+\delta_1, \qquad \mathrm{d}\rho(\varphi_x)=\exp(-g\varphi_x^4-a\varphi_x^2)\mathrm{d}\varphi_x,
    \end{equation*}
    where for $t\in \mathbb R$, $\delta_t$ is the Dirac measure at $t$, and where $g>0$ and $a\in \mathbb R$.
    \item[$(ii)$] The Gaussian free field, by choosing
    \begin{equation*}
        \mathrm{d}\rho(\varphi_x)=\exp(-a\varphi_x^2)\mathrm{d}\varphi_x
    \end{equation*}
    for a large enough constant $a$ that depends on $\beta$.
    \item[$(iii)$] General $P(\varphi)$ models, by choosing
    \begin{equation}\label{eq:equation p(phi)}
    \mathrm{d}\rho(\varphi_x)=\exp(-P(\varphi_x))\mathrm{d}\varphi_x,
    \end{equation}
    where $P$ is an even polynomial of degree at least $4$ and of positive leading coefficient.
\end{enumerate}
For a comprehensive introduction to these models, the interested reader can consult \cite{DC_lecturenotes,friedli_velenik_2017,Glimm-Jaffe,WP21}.

The definition of the model can be extended to include boundary conditions and long-range interactions; see Section \ref{section:def of model} for the more general definition. We will write $\nu_{\Lambda, \beta, \rho, J}^{\xi}$ for the finite-volume measure on $\Lambda$ with inverse temperature $\beta$, single-site measure $\rho$, interactions $J$ and boundary conditions $\xi$.

Infinite-volume Gibbs measures, defined via the \textit{Dobrushin--Lanford--Ruelle (DLR) equation} (see Definition \ref{Def:Gibbs} below), are objects of central interest in statistical physics, arising naturally as limits of finite-volume measures as $(\Lambda,E)$ tends to an infinite graph such as the lattice $\mathbb{Z}^d$. Since spins are, in general, unbounded in our setting, this raises the following question: for which boundary conditions is the sequence of measures tight? 
In the case of the massless Gaussian free field, the Cameron--Martin formula implies that the configuration on $\Lambda$ under boundary conditions $\xi$ has the same distribution as the configuration with free boundary conditions on $\Lambda$ shifted by the harmonic extension of $\xi$. 
In particular, if the boundary spins grow to infinity as $|\Lambda| \rightarrow \infty$, then the sequence is not tight. In contrast, for the massive Gaussian free field, the faster decay of the density of the single-site measure leads to exponential decay of correlations, which in turn allows for tightness of finite-volume measures as long as the boundary conditions grow weakly enough.

Answering this question is more challenging in the non-Gaussian case, for example, when the single-site measure $\rho$ is such that
\begin{equation}
    \label{eq: single-site n=2}
    \forall  a > 0 \quad 0 < \int_{\mathbb{R}} e^{a |u|^{2}} \mathrm{d} \rho(u) < \infty.
\end{equation}
The problem was studied on the lattice $\mathbb{Z}^{d}$ by Lebowitz and Presutti \cite{Lebowitz_Presutti} who proved tightness for boundary conditions that grow like $\sqrt{\log(\|x\|_{\infty})}$. Their approach utilises a regularity estimate developed by Ruelle \cite{Ruelle1970, Ruelle_estimates}, which bounds the density at a spin configuration $\varphi$ in terms of the density at $\varphi$ of a non-interacting system (i.e.\ a system with $\beta = 0$). See also \cite{COPP78, Bellissard1982, AKRT00} and references therein for related results on infinite-volume Gibbs measures supported on configurations of tempered growth.

The methods of \cite{Lebowitz_Presutti} were applied in \cite{random_tangled} to the $\varphi^4$ model on vertex-transitive graphs of polynomial growth. 
The main result of \cite{random_tangled} in this context is that any translation invariant Gibbs measure is a convex combination of two extremal measures, and regularity is needed to construct these extremal measures as limits of finite-volume measures. It was observed in \cite{random_tangled} that this result should extend to the setting of vertex-transitive amenable graphs, as previously established for the Ising model \cite{Ising_Gibbs_states}, but it is not clear how to generalise the regularity estimates of \cite{Lebowitz_Presutti} to this case. This provides the motivation for us to develop alternative arguments for proving regularity.

The main result of the present article is a regularity estimate that applies to both nearest-neighbour and long-range interactions on an arbitrary graph. The theorem bounds the Radon--Nikodym derivative of the system in a finite domain $\Lambda$ at parameter $\beta>0$, with boundary conditions $\xi$ that are allowed to grow to infinity as $|\Lambda|\to\infty$, with respect to a product measure associated with a non-interacting system with a modified single-site measure. The bound is expressed in terms of a function $A(x,\Lambda,\xi,C)$, where $x \in \Lambda$ is a vertex and $C \geq 1$ is a parameter. This function is related to the mean of each spin $\varphi_x$ and compared to the massless Gaussian free field, it can be interpreted as an analogue of the harmonic extension of $\xi$. More generally, it plays the role of a non-Gaussian analogue of the Cameron–Martin formula, quantifying the change of measure induced by the boundary conditions. The function $A(x,\Lambda,\xi,C)$ may take large values when $x$ is close to the boundary, but it decreases as $x$ moves further into the bulk, and provided the boundary conditions $\xi$ do not grow too rapidly, $A(x,\Lambda,\xi,C)$ remains bounded in the bulk of $\Lambda$. The formal definition can be found in Section \ref{section: bc notation}. For now, let us mention for concreteness that in the case of nearest-neighbour interactions, 
\[
A(x, \Lambda, \xi,C) \approx \max \left\{ 1, \max_{z \in \partial \Lambda} \left(\frac{|\xi_z|}{C}\right)^{(n-1)^{-d(x, z)}} \right\}.
\]

We observe a qualitative change in the behaviour of $A(x,\Lambda,\xi,C)$ depending on the tails of the single-site measure $\rho$. In order to make this point clearer, let us state the result in the case where $\rho$ satisfies the stronger assumption
\begin{equation}
\label{eq: single-site assumption}
0 < \int_{\mathbb{R}} e^{a |u|^{n}} \mathrm{d} \rho(u) < \infty,
\end{equation}
for some $a>0$ and $n>2$, a condition satisfied by all the $P(\varphi)$ models. 
Here we let $V$ be a countable set, and we consider interactions $(J_{x, y})_{x, y \in V}$ on $V$. We call the interactions \textit{admissible} if they satisfy: 
\begin{enumerate}
\renewcommand{\labelenumi}{\textbf{(C\arabic{enumi})}}
\renewcommand{\theenumi}{C\arabic{enumi}}
    \item \label{C1}\textbf{(Symmetry)} $J_{x, y} = J_{y, x}$ for all $x, y \in V$;
    \item \label{C2}\textbf{(Integrability)} There exists $f:\mathbb{R} \rightarrow [1, \infty)$ and constants $\delta_f, M_f > 0$ such that $f(t) \geq \log(|t|^{-1})^{1/n}$ for all $t \in (-\delta_f, \delta_f)$, and $\sum_{y \in V} |J_{x, y}| f(J_{x, y}) \leq M_f$ for all $x \in V$.
    \end{enumerate}    
We now state our regularity result, which does \textit{not} require the interactions $J$ to be ferromagnetic. Below $\varphi|_{\Lambda'}$ denotes the restriction of the field to $\Lambda'$, and $\mathrm{d} \nu_{\Lambda, \beta, \rho, J}^{\xi}[\varphi|_{\Lambda'} = \psi]$ the density of the restriction.

\begin{theorem}
\label{main theorem}
    Let $V$ be a countable set and $(J_{x, y})_{x, y \in V}$ be admissible interactions on $V$. Let $\beta \geq 0$, $a > 0$, $n > 2$ and let $\rho$ be a single-site measure satisfying \eqref{eq: single-site assumption}.
    There exist constants $C \geq 1, \tilde{C} > 0$ such that for any $\Lambda \subset V$ finite, $\Lambda' \subset \Lambda$, $\psi \in \mathbb{R}^{\Lambda'}$, and any boundary conditions $\xi \in \mathbb{R}^{V}$ with $\sum_{y \in V} |J_{x, y} \xi_y| < \infty$ for all $x \in \Lambda$,
    \begin{equation*}
        \mathrm{d} \nu_{\Lambda, \beta, \rho, J}^{\xi}[\varphi|_{\Lambda'} = \psi]
        \leq \left( \prod_{x \in \Lambda'} e^{\tilde{C} A(x, \Lambda, \xi, C)^{n}} \right) \mathrm{d} \nu_{\Lambda', 0, \rho_{\frac{a}{2}}, 0}^{0} [\psi],
    \end{equation*}
    where $\rho_{\frac{a}{2}}$ is defined by $\mathrm{d}\rho_{\frac{a}{2}}(u) = e^{\frac{a}{2} |u|^{n}} \mathrm{d} \rho(u)$.
    Moreover, one can take $C = C_1 \beta^{\frac{1}{n-2}} + C_2$ and $\tilde{C} = \tilde{C}_1 \beta^{\frac{n}{n-2}} + \tilde{C}_2$, where $C_1, C_2, \tilde{C}_1, \tilde{C}_2$ depend only on $\delta_f,M_f$ and $\rho$.
\end{theorem}

Theorem~\ref{main theorem} can be generalised to allow for single-site measures that depend on the vertex --- see Remark~\ref{remark:changed single-site} --- and to more general Hamiltonians or even conditional measures --- see Theorem~\ref{changed theorem n=2} and Remark~\ref{rem:conditional measure}. We also expect our arguments to be robust enough to apply to models with $k$-body interactions, provided \eqref{eq: single-site assumption} is satisfied with $n > k$, but we do not pursue this here. Furthermore, Theorem~\ref{main theorem} applies to arbitrary graphs, thus extending the regularity results of \cite{Lebowitz_Presutti} and \cite{random_tangled}. In particular, it opens the way for characterising the translation invariant Gibbs measures for the ferromagnetic $\varphi^4$ model on any vertex-transitive amenable graph, but verifying this is beyond the scope of the current paper. 

We now outline the proof of Theorem~\ref{main theorem}, which we believe offers a clearer probabilistic intuition than earlier approaches to regularity. The proof is based on an exploration argument in which we consider the cluster $\mathcal{C}$ of $\Lambda'$ consisting of vertices $x$ for which the spin $\varphi_x$ takes a large value. To accommodate potentially large boundary conditions, the threshold for including a vertex $x$ in the cluster is allowed to gradually grow as $x$ moves further away from $\Lambda'$, so that it matches the boundary condition $\xi_z$ at a vertex $z$ on the boundary of $\Lambda$. This is where the function $A(x,\Lambda,\xi,C)$ comes into play; its precise definition enables fine control of the behaviour of $\mathcal{C}$. In particular, $A(x,\Lambda,\xi,C)$ is defined so that certain technical conditions relating parents and children in the exploration process are satisfied (see Lemmas \ref{edge removal general 1} and \ref{edge removal general 2}). These conditions allow us to isolate each vertex of $\mathcal{C}$ from its neighbours at a finite cost, thereby yielding a non-interacting system. Finally, we control the size of $\mathcal{C}$ by comparing it to the total progeny of a subcritical branching process; to carry this out, the parameter $C$ in the definition of $A(x,\Lambda,\xi,C)$ must be chosen appropriately. In contrast to \cite{Lebowitz_Presutti} and \cite{random_tangled}, our arguments do not require any assumption of polynomial growth or amenability of the underlying graph. 


As an immediate consequence of Theorem \ref{main theorem}, we obtain tightness for any boundary conditions such that $A(x, \Lambda, \xi, C)$ remains bounded for each $x\in V$ as $\Lambda \nearrow V$. In the case of ferromagnetic nearest-neighbour interactions, this includes boundary conditions growing like a double exponential of the form $K^{(n-1)^{d(o, x)}}$, where $o$ is a fixed origin. This improves on the results of Lebowitz and Presutti \cite{Lebowitz_Presutti}, which allow only logarithmically growing boundary conditions. In Proposition \ref{Anti-tightness prop}, we show that this result is optimal in the case of non-negative boundary conditions for $\rho$ defined as in \eqref{eq:equation p(phi)}, in the sense that tightness does not occur for non-negative boundary conditions that grow even faster than a constant to the power $(n-1)^{d(o, x)}$. In the case of long-range interactions, the rate of decay of the interactions $J$ comes into play, as any vertex $x$ can affect the value of $\varphi_o$ through the edge $ox$. This makes characterising the boundary conditions that lead to tightness challenging. Nevertheless, in Proposition \ref{long range bc example}, we prove that if the interactions $J$ are well behaved, in the sense that $|J_{x, y}| \leq d(x, y)^{-r}$ for some $r > 0$, then we have tightness for boundary conditions growing like $f(d(o, x)^{-r})$, where $f$ is a function that encodes the rate of decay of $J$, i.e.\ it satisfies \eqref{C2} and some additional assumptions.

Coming back to the dependence of the behaviour of $A(x,\Lambda,\xi,C)$ on the tails of the single-site measure $\rho$, let us mention that a similar statement (see Theorem \ref{Theorem: n=2}) to that of Theorem \ref{main theorem} holds if we relax the assumptions on the single-site measure $\rho$ to allow for any $\rho$ that satisfies \eqref{eq: single-site n=2}. In this case, we observe a qualitative change in the behaviour of $A(x, \Lambda, \xi, C)$. For example, in the nearest-neighbour case, we obtain tightness for any boundary conditions growing at most exponentially in the distance, so we observe a jump in the threshold for tightness from exponential to double-exponential at $n=2$.

Theorems \ref{main theorem} and \ref{Theorem: n=2} can be used to obtain regularity for infinite-volume measures, which we define below.
\begin{definition}
    \label{Def:Gibbs}
    Let $a>0, \beta \geq 0$, and assume $J$ satisfies \eqref{C1}, \eqref{C2} and $\rho$ satisfies \eqref{eq: single-site n=2}.
    We say that a probability measure $\nu$ on $\mathbb{R}^{V}$ with the $\sigma-$algebra generated by Borel events depending on finitely many vertices is
    \begin{itemize}
    \item
    $a$-regular if there exists a constant $B \in [0, \infty)$ 
    such that for every $\Lambda \subset V$ finite and $\psi \in \mathbb{R}^{\Lambda}$, 
    \begin{equation*}
        \mathrm{d}\nu[\varphi|_{\Lambda} = \psi] \leq e^{B|\Lambda|} \mathrm{d} \nu_{\Lambda, 0, \rho_a, 0}^{0}[\psi],
    \end{equation*}
    where $\rho_{a}$ is defined by $\mathrm{d}\rho_{a}(u) = e^{a |u|^{2}} \mathrm{d} \rho(u)$. 

    \item
    A Gibbs measure if for every finite $\Lambda \subset V$ and any bounded measurable function $g: \mathbb{R}^{\Lambda} \rightarrow \mathbb{R}$, the DLR equation
    \begin{equation*}
        \nu[g] = \int_{\xi \in \mathbb{R}^{V}} \langle g\rangle _{\Lambda, \beta, \rho, J}^{\xi} \mathrm{d} \nu(\xi)
    \end{equation*}
    holds. In particular, we assume that $\nu$ is almost surely supported on configurations $\xi$ such that   $\langle \cdot \rangle_{\Lambda, \beta, \rho, J}^{\xi}$ is well-defined.
    \end{itemize}
\end{definition}
In Section \ref{Section: infinite-volume}, we give conditions on $\xi$ that ensure that the limiting measure (if it exists) as $\Lambda \nearrow V$ is an $a$-regular Gibbs measure. For ferromagnetic interactions, we construct the plus measure $\nu^{+}$ as the limit of finite-volume measures and show that it is $a$-regular for some $a>0$. We also show that $\nu^{+}$ is maximal, hence extremal, in the sense that if $\nu$ is an $a'$-regular Gibbs measure for some $a'>0$, then $\nu$ is stochastically dominated by $\nu^{+}$. Furthermore, in Section~\ref{sec: alternative construction}, in the nearest-neighbour case, we introduce a family of finite-volume measures with random boundary conditions that converge to $\nu^{+}$ and are regular up to the boundary, in contrast to the constructions in \cite{Lebowitz_Presutti} and \cite{random_tangled}, which rely on logarithmically growing boundary conditions. These finite-volume measures are also stochastically decreasing in the volume, similarly to the case of the Ising model. As a key result for establishing these properties, we show that our spin measures are stochastically dominated by a product measure.
We expect that this alternative construction of the plus measure may help avoid challenges arising from the absence of maximal boundary conditions at finite volume, leading to simplifications of the arguments in \cite{random_tangled} and \cite{well_behaved}, as well as to applications in future works.

\subsection{Paper organisation}
In Section \ref{section: preliminaries}, we define the notation that will be used throughout the rest of the paper.
Section \ref{section: main proof} is dedicated to the proof of Theorem \ref{main theorem}.
The methods developed here can be applied to a range of other similar models, some examples of which are given in Section \ref{section: related models}.
In Section \ref{section: corollaries}, we examine for which boundary conditions tightness can be obtained and construct the infinite-volume plus measure as a limit of finite-volume measures. In Section~\ref{sec: stochastic domination}, we prove stochastic domination of our spin measures by a product measure, and we use this result to give an alternative construction of the plus measure.

\paragraph{Acknowledgements} We thank Trishen Gunaratnam, Dmitrii Krachun, Romain Panis and Franco Severo for useful discussions. CP was supported by an EPSRC New Investigator Award (UKRI1019).

\section{Definitions and preliminaries}
\label{section: preliminaries}
In this section, we define the model in full generality, as well as introduce some additional notation and results that will be used in the proofs.

\subsection{Definition of the model}
\label{section:def of model}
Let $V$ be a countably infinite set of vertices. Let $\beta \geq 0$ be the inverse temperature and let $\rho$ be a single-site measure satisfying \eqref{eq: single-site n=2}. At some points, including in Theorem \ref{main theorem}, we assume $\rho$ satisfies the stronger condition \eqref{eq: single-site assumption} with respect to some constants $a >0$ and $n > 2$. We say $\rho$ is \emph{even} if for every Borel measurable set $S \subset \mathbb{R}$, $\rho(S) = \rho(-S)$.
For $b \in \mathbb{R}$, we will write $\rho_b$ for the measure with density $e^{b|u|^{n}}$ with respect to $\rho$, where we implicitly assume $n=2$ when we do not require $\rho$ to satisfy \eqref{eq: single-site assumption}. 

Consider interactions $(J_{x, y})_{x, y \in V}$ on $V$ that satisfy conditions \eqref{C1} and \eqref{C2}.
We will sometimes assume also that the interactions are ferromagnetic, meaning that $J_{x, y} \geq 0$ for all $x, y \in V$, but this is not required in our regularity results.
Given a finite subset $\Lambda \subset V$, we denote by $\overline{E}(\Lambda,J)$ the set of unordered pairs of vertices $x, y \in V$ with at least one vertex in $\Lambda$ such that $J_{x, y} \neq 0$, and write elements of $\overline{E}(\Lambda,J)$ in the form $xy$. We define the model on $\Lambda$ with boundary conditions $\xi \in \mathbb{R}^{V}$, which we assume satisfy $\sum_{y \in V} |J_{x, y} \xi_y| < \infty$ for all $x \in \Lambda$.


\begin{definition}
    The finite-volume spin model on $\Lambda$ is the measure $\nu_{\Lambda, \beta, \rho, J}^{\xi}$ on $\mathbb{R}^{\Lambda}$ given by
\begin{equation}
    \label{long range definition}
    \mathrm{d} \nu_{\Lambda, \beta, \rho, J}^{\xi}[\varphi] = \frac{1}{Z_{\Lambda, \beta, \rho, J}^{\xi}}\exp(- \beta H_{\Lambda, J}^{\xi}(\varphi)) \prod_{x \in \Lambda} \mathrm{d} \rho(\varphi_x),
\end{equation}
where the partition function $Z_{\Lambda, \beta, \rho, J}^{\xi}$ is the normalising constant that makes $\nu_{\Lambda, \beta, \rho, J}^{\xi}$ a probability measure and $H_{\Lambda, J}^{\xi}(\varphi)$ is the Hamiltonian, given by
\begin{equation*}
    H_{\Lambda, J}^{\xi}(\varphi) = - \sum_{\substack{xy \in \overline{E}(\Lambda, J)\\x, y \in \Lambda}} J_{x, y} \varphi_x \varphi_y - \sum_{\substack{xy \in \overline{E}(\Lambda, J) \\ x \in \Lambda, \, y\in V \setminus \Lambda}} J_{x,y} \varphi_x \xi_y.
\end{equation*}
\end{definition}

Let us mention that assumption \eqref{eq: single-site n=2} is necessary for the model to be well-defined for all $\beta\geq 0$ due to the quadratic nature of the interactions:
\[
\varphi_x\varphi_y=-\frac{(\varphi_x-\varphi_y)^2}{2}+\frac{\varphi_x^2+\varphi_y^2}{2}.
\]
We write $\langle \cdot \rangle_{\Lambda, \beta, \rho, J }^{\xi}$ for the expectation with respect to the measure $\nu_{\Lambda, \beta, \rho, J}^{\xi}$, and for $\Lambda' \subset \Lambda$ write $\nu_{(\Lambda|\Lambda'), \beta, \rho, J}^{\xi}$ for the restriction of $\nu_{\Lambda, \beta, \rho, J}^{\xi}$ to events that only depend on spins in $\Lambda'$.
For a sequence $(\Lambda_i)_{i \geq 1}$ of finite subsets of $V$, we say $\Lambda_i \nearrow V$ if $\Lambda_i \subset \Lambda_{i+1}$ for all $i$ and $\bigcup_{i =1}^{\infty}\Lambda_i = V$.
We write $\Lambda \Subset V$ to denote that $\Lambda$ is a finite subset of $V$ and say that the family of measures $(\nu_{\Lambda, \beta, \rho, J}^{\xi})_{\Lambda \Subset V}$ is tight if for any $\Lambda' \Subset V$, the measures $\nu_{(\Lambda|\Lambda'), \beta, \rho, J}^{\xi}$ for $\Lambda' \subset \Lambda \Subset V$ are tight in the usual sense.

The measure $\nu_{\Lambda, \beta, \rho, J}^{\xi}$ satisfies the domain Markov property, which states that for any $\Lambda' \subset \Lambda$, $\psi \in \mathbb{R}^{\Lambda'}, \eta \in \mathbb{R}^{\Lambda \setminus \Lambda'},$
\begin{equation*}
    \mathrm{d} \nu_{\Lambda, \beta, \rho, J}^{\xi}[\varphi|_{\Lambda'} = \psi \mid \varphi|_{\Lambda \setminus \Lambda'} = \eta] = \mathrm{d} \nu_{\Lambda', \beta, \rho, J}^{\eta \cup \xi}[\psi],
\end{equation*}
where $\eta \cup \xi \in \mathbb{R}^{V}$ is the configuration which is equal to $\eta$ on $\Lambda \setminus \Lambda'$ and is equal to $\xi$ elsewhere.

Another useful property of the model is monotonicity in boundary conditions. Before stating this, we must first introduce the notion of an increasing function.

\begin{definition}
    We say that $g: \mathbb{R}^{\Lambda} \rightarrow \mathbb{R}$ is an increasing function if for any $\varphi, \varphi' \in \mathbb{R}^{\Lambda}$ with $\varphi_x \leq \varphi'_x$ for all $x \in \Lambda$, then $g(\varphi) \leq g(\varphi')$. We say an event $E$ is an increasing event if $\mathbbm{1}_{E}$ is an increasing function. For measures $\nu, \nu'$ on $\mathbb{R}^{\Lambda}$, we say that $\nu$ is stochastically dominated by $\nu'$ and write $\nu \preceq \nu'$ if $\nu[g] \leq \nu'[g]$ for any increasing function $g: \mathbb{R}^{\Lambda} \rightarrow \mathbb{R}$.
\end{definition}

\begin{proposition}
    \label{b.c monotonicity}
    Suppose $J$ is ferromagnetic and let $\xi, \xi' \in \mathbb{R}^{V}$.
    \begin{enumerate}
        \item[(i)] 
        If $\xi_y \leq \xi'_y$ for all $y \in V$, then 
        \[
        \nu_{\Lambda, \beta, \rho, J}^{\xi} \preceq \nu_{\Lambda, \beta, \rho, J}^{\xi'}.
        \]
        \item[(ii)]
        If $\rho$ is even and $|\xi_y| \leq \xi'_y$ for all $y \in V$, then the law of the absolute value field satisfies
        \[
        \nu_{\Lambda, \beta, \rho, J}^{\xi} ( |\cdot|) \preceq \nu_{\Lambda, \beta, \rho, J}^{\xi'} (| \cdot|).
        \]
    \end{enumerate}

\end{proposition}
\begin{proof}
    Let $g: \mathbb{R}^{\Lambda} \rightarrow \mathbb{R}$ be an increasing function. To prove the first statement, note that
    \begin{equation*}
        F(\varphi) \coloneq \exp(\beta H_{\Lambda, J}^{\xi}(\varphi) - \beta H_{\Lambda, J}^{\xi'}(\varphi)) = \exp\Bigg( \sum_{\substack{xy \in \overline{E}(\Lambda, J) \\ x \in \Lambda, \, y\in V \setminus \Lambda}} \beta J_{x,y} \varphi_x (\xi'_y -\xi_y) \Bigg)
    \end{equation*}
    is an increasing function. Using the FKG inequality \cite[Theorem 4.4.1]{Glimm-Jaffe}, we obtain
    \begin{equation*}
        \langle g \rangle_{\Lambda, \beta, \rho, J }^{\xi'} = \frac{\langle F g \rangle_{\Lambda, \beta, \rho, J }^{\xi}}{\langle F \rangle_{\Lambda, \beta, \rho, J }^{\xi}} \geq \langle g \rangle_{\Lambda, \beta, \rho, J }^{\xi}.
    \end{equation*}
    The second statement follows from the absolute value FKG inequality \cite[Corollary 6.4]{LammersOtt2021} and a small modification to the proof of \cite[Lemma 2.13]{random_tangled}.
\end{proof}

\subsection{Definition of $A(x,\Lambda)$}
\label{section: bc notation}
In this section, we introduce the functions $A(x, \Lambda)$ and $\tilde{A}(x, \Lambda)$. Before stating the formal definitions, we give the following rough description. Consider a walk from $x$ to some vertex $z \in V$ and assign a value to each vertex along the walk, with the value at $z$ proportional to $|\xi_z|$. The value at each vertex moving away from $z$ drops by an amount depending on the interaction strength between the vertices, and we set $\tilde{A}(x, \Lambda)$ to be the maximum over all walks of the value at $x$. The function $A(x, \Lambda)$ is similar, but here we view the boundary conditions $\xi$ as an external field by combining the contributions from all the vertices in $V \setminus \Lambda$ that interact with a given vertex $x \in \Lambda$ into a single point $h_{x,\Lambda}$, defined as
\begin{equation*}
    h_{x, \Lambda} = \sum_{y \in V \setminus \Lambda} J_{x, y} \xi_y.
\end{equation*}
Doing this will ensure that $A(x, \Lambda)$ is finite for any finite subset $\Lambda \subset V$, since our assumptions on the boundary conditions $\xi$ imply that $|h_{x,\Lambda}| < \infty$ for all $x \in \Lambda$.

We now give the definitions of $A(x, \Lambda)$ and $\tilde{A}(x, \Lambda)$, first stating what we mean by a walk.
    Let $R, S, T \subset V$.
    \begin{itemize}
    \item
    We say that a sequence of (not necessarily distinct) vertices $x_0, x_1, \ldots, x_m \in V$ is a walk from $S$ to $T$ in $\overline{(R, J)}$ if $x_0 \in S, \, x_m \in T$, $x_1, \ldots, x_{m-1} \in R$, and $J_{x_{i-1}, x_i} \neq 0$ for all $i \in \{1, \ldots, m\}$.
    \item
    We say that $R$ is $J$-connected if for any $x, y \in R$, there exists a walk from $x$ to $y$ in $\overline{(R, J)}$.
    \end{itemize}
\noindent When discussing walks to or from a singleton $\{x\}$, we may write it as $x$ instead.

We first treat the $n>2$ case. We give concrete examples in Sections~\ref{section:nearest-neighbour} and \ref{Section: long-range}.
For $n>2$, $R \subset V$, $x \in R$ and $C \geq 1$, define $\tilde{A}(x,R)=\tilde{A}(x, R, \xi, C, J, f, n)$, to be the smallest $A \geq 1$ such that for any $z \in V$ and any walk $x_0, \ldots, x_m$ from $x$ to $z$ in $\overline{(R, J)}$,
\begin{equation*}
    |\xi_z| \leq C A^{(n-1)^{m}} \prod_{i =1}^{m} f(J_{x_{i-1}, x_{i}})^{(n-1)^{m-i}}.
\end{equation*}
if such $A$ exists, and otherwise we say that $\tilde{A}(x, R) = \infty$. Note that $\tilde{A}(x, R)$ is increasing in $R$. 
We also define $A(x,R)=A(x, R, \xi, C, J, f, n)$ to be the smallest $A \geq 1$ such that for any $y \in R$ and any walk $x_0, x_1,\ldots, x_{m}$ from $x$ to $y$ in $\overline{(R, J)}$,
\begin{equation*}
    |h_{y, R}| \leq \left( \sum_{z \in V \setminus R} |J_{y, z}| f(J_{y,z}) \right) C A^{(n-1)^{m+1}} \prod_{i =1}^{m} f(J_{x_{i-1}, x_{i}})^{(n-1)^{m+1-i}},
\end{equation*}
and say that $A(x, R) = \infty$ if no such $A$ exists.

We now give the definition in the $n=2$ case. For $R \subset V$, $x \in R$, $C\geq 1$ and $\lambda \geq 1$, we define $\tilde{A}(x, R) = \tilde{A}(x, R, \lambda, \xi, C, J, f)$ to be the smallest $A\geq 1$ such that for any $z \in V$ and any walk $x_0, x_1, \ldots, x_m$ from $x$ to $z$ in $\overline{(R, J)}$,
\begin{equation*}
    |\xi_z| \leq C A \lambda^{m} \prod_{i=1}^{m} f(J_{x_{i-1}, x_i}).
\end{equation*}
We also define
$A(x, R) = A(x, R, \lambda, \xi, C, J, f)$ to be the smallest $A\geq 1$ such that for any walk $x_0, x_1, \ldots, x_m$ from $x$ to $R$ in $\overline{(R, J)}$,
\begin{equation*}
    |h_{x_m, R}| \leq \left(\sum_{z \in V \setminus R} |J_{x_m, z}| f(J_{x_m, z}) \right) C A \lambda^{m+1} \prod_{i=1}^{m} f(J_{x_{i-1}, x_i}).
\end{equation*}

We will frequently drop the parameters $\xi, C, J, f, n, \lambda$ from the notation when they are clear from the context.
In both cases $n>2$ and $n=2$, it follows (from Theorem \ref{main theorem} and Theorem \ref{Theorem: n=2} respectively) that we have tightness if $A(x, \Lambda)$ is bounded above by a function of $x$ that does not depend on $\Lambda$. When considering whether this is the case for a particular choice of boundary conditions, it may be more convenient to work with the function $\tilde{A}$ and use the fact that for any $\Lambda \Subset V$ and any $x \in \Lambda$ we have $A(x, \Lambda) \leq \tilde{A}(x, \Lambda) \leq \tilde{A}(x, V)$. To see why this is the case, let $y \in \Lambda$ and let $x_0, \ldots, x_m$ be a walk from $x$ to $y$ in $\overline{(\Lambda, J)}$. Suppose $n>2$ (the $n=2$ case is similar). For any $z \in V$ with $J_{y,z} \neq 0$, the definition of $\tilde{A}(x, \Lambda)$ applied to the walk $x_0, \ldots, x_m, z$ implies that
\begin{equation*}
    |\xi_z| \leq C \tilde{A}(x, \Lambda)^{(n-1)^{m+1}} f(J_{y,z}) \prod_{i =1}^{m} f(J_{x_{i-1}, x_{i}})^{(n-1)^{m+1-i}},
\end{equation*}
hence
\begin{align*}
    |h_{y, \Lambda}| &\leq \sum_{z \in V \setminus \Lambda} |J_{y, z}| |\xi_z|\\
    &\leq \left( \sum_{z \in V \setminus \Lambda} |J_{y, z}| f(J_{y,z}) \right) C \tilde{A}(x, \Lambda)^{(n-1)^{m+1}} \prod_{i =1}^{m} f(J_{x_{i-1}, x_{i}})^{(n-1)^{m+1-i}},
\end{align*}
which implies $A(x, \Lambda) \leq \tilde{A}(x, \Lambda)$. Since $\tilde{A}(x, R)$ is increasing in $R$, we obtain tightness if $\tilde{A}(x, V)$ is finite for all $x \in V$. For $n>2$, define $\Xi = \Xi(V, J, f, n)$ to be the set of boundary conditions for which this is the case. Similarly, for $n=2$, define $\Xi(\lambda)$ to be the set of boundary conditions for which $\tilde{A}(x, V,\lambda)$ is finite for all $x \in V$.
See Sections \ref{section:nearest-neighbour} and \ref{Section: long-range} for examples of boundary conditions that are in $\Xi$ for different choices of interactions $J$.

We now state a lemma that allows us to compare the values of $A(x, R)$ and $A(y, R)$ or $\tilde{A}(x, R)$ and $\tilde{A}(y, R)$. In the case when $V$ is $J$-connected, this means that to determine whether given boundary conditions are in $\Xi$, it suffices to check whether $\tilde{A}(x, V)$ is finite for one vertex $x$.
\begin{lemma}
    \label{A comparison}
    Let $R \subset V$, $C \geq 1$, and $\xi \in \mathbb{R}^{V}$.
    For any walk $x_0, x_1, \ldots, x_k$ in $\overline{(R, J)}$ with $x_0, x_k \in R$, we have
    \begin{align*}
            &\text{(i) If } n > 2, \quad \tilde{A}(x_k, R) \leq \tilde{A}(x_0, R)^{(n-1)^{k}} \prod_{i=1}^{k} f(J_{x_{i-1}, x_i})^{(n-1)^{k-i}},\\
            &\text{(ii) If } n > 2, \quad A(x_k, R) \leq A(x_0, R)^{(n-1)^{k}} \prod_{i=1}^{k} f(J_{x_{i-1}, x_i})^{(n-1)^{k-i}},\\
            &\text{(iii) If $n=2$, \quad $\tilde{A}(x_k, R) \leq \tilde{A}(x_0, R) \lambda^{k} \prod_{i=1}^{k} f(J_{x_{i-1}, x_i})$},\\
            &\text{(iv) If $n=2$, \quad $A(x_k, R) \leq A(x_0, R) \lambda^{k} \prod_{i=1}^{k} f(J_{x_{i-1}, x_i})$}.
    \end{align*}
\end{lemma}

\begin{proof}
    The proofs of the first two and last two statements are very similar, so we only prove (i) and (iv) here.
    For (i), let $z \in V$ and let $y_0, y_1, \ldots, y_j$ be a walk from $x_k$ to $z$ in $\overline{(R, J)}$. Considering the walk $x_0, \ldots, x_k, y_1, \ldots, y_j$, then by definition of $\tilde{A}(x_{0}, R)$, we have
    \begin{equation*}
        |\xi_{z}| \leq C \tilde{A}(x_0, R)^{(n-1)^{k+j}}
        \left(\prod_{i = 1}^{k} f(J_{x_{i-1}, x_i})^{(n-1)^{k+j-i}} \right)
        \left(\prod_{i = 1}^{j} f(J_{y_{i-1}, y_i})^{(n-1)^{j-i}} \right).
    \end{equation*}
    Hence $\tilde{A}(x_0, R)^{(n-1)^{k+j}}
        \prod_{i = 1}^{k} f(J_{x_{i-1}, x_i})^{(n-1)^{k+j-i}}$ satisfies the requirements for $\tilde{A}(x_k, R)^{(n-1)^{j}}$, so
    \begin{align*}
        \tilde{A}(x_k, R)^{(n-1)^{j}} &\leq \tilde{A}(x_0, R)^{(n-1)^{k+j}}
        \left(\prod_{i = 1}^{k} f(J_{x_{i-1}, x_i})^{(n-1)^{k+j-i}} \right)\\
        &= \left( \tilde{A}(x_0, R)^{(n-1)^{k}}
        \prod_{i = 1}^{k} f(J_{x_{i-1}, x_i})^{(n-1)^{k-i}} \right)^{(n-1)^j}.
    \end{align*}
    Taking both sides to the power $1/(n-1)^j$ yields (i).
    
    To prove (iv), let $y \in R$ and let $y_0, y_1, \ldots, y_j$ be a walk from $x_k$ to $y$ in $\overline{(R, J)}$. Considering the walk $x_0, \ldots, x_k, y_1, \ldots, y_j$, then by definition of $A(x_{0}, R)$, we have
    \begin{equation*}
    |h_{y, R}| \leq \left(\sum_{z \in V \setminus R} |J_{y, z}| f(J_{y, z}) \right) C A(x_0,R) \lambda^{k+j+1} \prod_{i=1}^{k} f(J_{x_{i-1}, x_i})\prod_{i=1}^j f(J_{y_{i-1}, y_i}).
\end{equation*}
    This means that $A(x_0,R)\lambda^{k} \prod_{i=1}^{k} f(J_{x_{i-1}, x_i})$ satisfies the requirements for $A(x_k,R)$, which yields (iv).
\end{proof}

\subsection{Branching processes}
In the proof of our main regularity theorem, we will use a standard result on branching processes from \cite{Branching} to bound the size of the cluster where the spins take large values. Below we give the definition of a branching process and then state this result.

For a random variable $X$ taking values in the non-negative integers, a branching process with offspring distribution $X$ and initial population $k \in \mathbb{N}$ is a sequence of random variables $(Z_n)_{n \geq 0}$ such that $Z_0 = k$ and for all $n \geq 1$,
$Z_n = \sum_{i=1}^{Z_{n-1}} X_{n, i}$,
where $X_{n, i}$ are independent random variables with the same distribution as $X$.
We call $T\coloneq\sum_{n=0}^{\infty} Z_n$ the total progeny of the branching process.

\begin{theorem}[\!{\cite[Theorem 3.13]{Branching}}] \label{progeny theorem}
    For a branching process with offspring distribution $X$ and initial population $k$, the distribution of the total progeny $T$ is given by
    \begin{equation*}
        \mathbb{P}[T = n] = \frac{k}{n} \mathbb{P}[X_1 + \ldots + X_n = n-k],
    \end{equation*}
    where $X_1, \ldots, X_n$ are independent random variables with the same distribution as $X$.
\end{theorem}

\section{Proof of Theorem \ref{main theorem}}
\label{section: main proof}
In this section, we prove Theorem \ref{main theorem} by employing an exploration argument. We start by gathering some useful inequalities.

\begin{lemma} \label{edge removal general 1}
     Let $\beta>0$, $n\geq 2$, $\alpha_0\geq 1$ and $C\geq \alpha_0$. For every $x, y \in V$, $|\varphi_x|\geq C$, and $t_x, t_y \in [-\alpha_0, \alpha_0]$,
        \[
        \beta (|\varphi_x| |\varphi_y| + |t_x||t_y|) \leq 
        \frac{2\beta}{C^{n-2}}(|\varphi_x|^{n}+|\varphi_y|^{n} ).
        \]
\end{lemma}    
\begin{proof}
By Young's inequality and the fact that $|t_x|,|t_y|\leq |\varphi_x|$,
\begin{equation}\label{eq:Young inequality}
\beta (|\varphi_x| |\varphi_y| + |t_x||t_y|) \leq 
\frac{\beta}{2} (|\varphi_x|^{2} +  |\varphi_y|^{2}+|t_x|^2+|t_y|^2)\leq \frac{\beta}{2}(3|\varphi_x|^{2}+|\varphi_y|^{2}).
\end{equation}
If $|\varphi_y|<C$, then \eqref{eq:Young inequality} implies that
\[
\beta (|\varphi_x| |\varphi_y| + |t_x||t_y|) \leq 2\beta |\varphi_x|^2\leq \frac{2\beta}{C^{n-2}}|\varphi_x|^n.
\]
If $|\varphi_y|\geq C$, then \eqref{eq:Young inequality} gives
\[
\beta (|\varphi_x| |\varphi_y| + |t_x||t_y|) \leq \frac{\beta}{2C^{n-2}}(3|\varphi_x|^{n}+|\varphi_y|^{n})\leq  \frac{3\beta}{2C^{n-2}}(|\varphi_x|^{n}+|\varphi_y|^{n}).
\]
This completes the proof.
\end{proof}

\begin{lemma}\label{edge removal general 2}
Let $\beta, a>0$, $n>2$, $\alpha_0\geq 1$ and $C\geq  \alpha_0+\left(\frac{4M_f\beta}{a}\right)^{\frac{1}{n-2}}$. If $|\varphi_x| \geq C$ and $|\varphi_y| \leq \frac{f(J_{x,y})}{C^{n-2}}|\varphi_x|^{n-1}$, then for any $t \in [-\alpha_0, \alpha_0]$,
\[
\beta (|\varphi_x| |\varphi_y| + |t||\varphi_y|) \leq \frac{a f(J_{x, y})}{2 M_f} |\varphi_x|^{n}.
\]
\end{lemma}
\begin{proof}
Note that since $|\varphi_y| \leq \frac{f(J_{x,y})}{C^{n-2}}|\varphi_x|^{n-1}$, we have
\begin{align*}  
    \beta |\varphi_y| (|\varphi_x| + |t|) -\frac{a f(J_{x, y})}{2M_{f}} |\varphi_x|^{n}
    &\leq \beta f(J_{x, y}) \frac{|\varphi_x|^{n-1}}{C^{n-2}} (|\varphi_x| + |t|)-\frac{a f(J_{x, y})}{2 M_{f}} |\varphi_x|^{n}\\
    &\leq f(J_{x, y}) |\varphi_x|^{n} \left( \frac{\beta}{C^{n-2}} \left(1 + \frac{\alpha_0}{|\varphi_x|}\right) - \frac{a}{2M_f} 
    \right).
\end{align*}
For every $C\geq \alpha_0 + (4a^{-1}M_f\beta)^\frac{1}{n-2}$, the above expression is negative when $|\varphi_x| \geq C$.      
\end{proof}

We now proceed with the proof of Theorem \ref{main theorem}. Recall that we consider $\Lambda'\subset \Lambda$. Our approach is based on an exploration process that builds the cluster $\mathcal{C}$ of vertices $x$ such that there exists a walk from $\Lambda'$ to $x$ where the spins take large values at each vertex along the walk. To accommodate the potentially large boundary conditions, we allow the minimum spin value needed to be in $\mathcal{C}$ to grow progressively as we move away from $\Lambda'$, ensuring that no vertices in $V \setminus \Lambda$ are included in $\mathcal{C}$. This gradual growth ensures that the conditions of Lemmas \ref{edge removal general 1} and \ref{edge removal general 2} are satisfied, which in turn enables us to isolate the vertices of $\mathcal{C}$ from their neighbours, at the cost of modifying the single-site measure. For those vertices in $\Lambda'$ where the spins take small values, we estimate their contribution to the Radon–Nikodym derivative directly.

\begin{proof}[Proof of Theorem \ref{main theorem}]
Fix a finite subset $\Lambda \subset V$ and let $\Lambda' \subset \Lambda$. We write $E$ for $\{xy \in \overline{E}(\Lambda, J): x,y \in \Lambda\}$ and $A_x$ for $A(x, \Lambda,  \xi, C, J, f, n)$, where $C \geq 1$ is a constant to be determined.
Define $\mathcal{C}$ to be the set of vertices $x \in \Lambda$ such that for some $m \in \{ 0, 1, \ldots\}$ there exists a walk $x_0, x_1, \ldots, x_m$ from $\Lambda'$ to $x$ in $\overline{(\Lambda, J)}$ that satisfies
\begin{equation}
    \label{cluster membership criteria}
    \forall k \in S_m \quad
    |\varphi_{x_k}| \geq C A_{x_0}^{(n-1)^{k}} \prod_{i = 1}^{k} f(J_{x_{i-1}, x_i})^{(n-1)^{k-i}},
\end{equation}
where $S_0 = \{0\}$ and $S_m = \{1, \ldots, m\}$ for $m \geq 1$.
For each $i \in \{ 0, 1, \ldots\}$, let $\mathcal{C}_i\subset \mathcal{C}$ denote the set of vertices for which $i$ is the smallest value of $m$ such that there exists a walk satisfying \eqref{cluster membership criteria}.
Note that if $x \in \mathcal{C}$, then $|\varphi_x| \geq C$, and if $y \in \Lambda'$, then Lemma \ref{A comparison} implies that $y\in \mathcal{C}$ if and only if $y \in \mathcal{C}_0$, that is, if and only if $|\varphi_y| \geq CA_y$.

Our aim is to prove that there exists $K \geq 0$ such that for any $V_1, V_2, \ldots$ pairwise disjoint subsets of $\Lambda \setminus \Lambda'$ we have
\begin{align}\label{eq:branching process comparison}
        &\mathrm{d}\nu_{\Lambda, \beta, \rho, J}^{\xi}  \left[
        \varphi|_{\Lambda'}, \, (\mathcal{C}_i)_{i \geq 1} = (V_i)_{i \geq 1} \right] \leq \\
        \nonumber
        &\exp \left(K|\Lambda'\cup V'| \right) \left( \prod_{x \in \Lambda'} e^{\alpha_1 M_f A_{x}^{n}} \mathrm{d} \rho_{\frac{a}{2}}(\varphi_x) \right)
        \prod_{i = 0}^{\infty} \prod_{y \in V_{i+1}} \max_{x \in V_i} p_{x, y},
    \end{align}
where $\alpha_1=2\beta C^2$, $V_0 = \Lambda'$, $V' = \bigcup_{i = 1}^{\infty} V_i$, and
\begin{equation*}
    p_{x, y} = \begin{cases}
    \int_{|u| \geq C f(J_{x,y})} \mathrm{d} \rho_{\frac{a}{2}}(u)  & \text{if } J_{x,y} \neq 0, \\ 0 & \text{if } J_{x,y} = 0. \end{cases}
\end{equation*}
Here $(p_{x, y})_{x, y \in V}$ can be interpreted as the offspring distribution of a branching process in the sense that $y$ is a child of $x$ with probability $p_{x, y}$. The distribution of the number of children of vertex $x$ in this process depends on $x$, but we can get a uniform control of the distribution by tuning the value of $C$. 

Assume that \eqref{eq:branching process comparison} holds for now. Summing \eqref{eq:branching process comparison} over all possibilities $(V_j)$ for $(\mathcal{C}_j)$ and noting that $\Lambda' \cap V' = \emptyset$ on the event $\{(\mathcal{C}_i)_{i \geq 1} = (V_i)_{i \geq 1}\}$, we obtain
\begin{align}
        \mathrm{d}\nu_{\Lambda, \beta, \rho, J}^{\xi}  \left[
        \varphi|_{\Lambda'}\right] \leq 
        \nonumber
        \left( \prod_{x \in \Lambda'} e^{\alpha_1 M_f A_{x}^{n}} \mathrm{d} \rho_{\frac{a}{2}}(\varphi_x) \right) \sum_{(V_j)}
        e^{K(|\Lambda'|+|V'|)}\prod_{i = 0}^{\infty} \prod_{y \in V_{i+1}} \max_{x \in V_i} p_{x, y}.
\end{align}
We can thus conclude by applying Lemma~\ref{lem:inside proof 2}, which bounds the right-hand side.

Let us now prove \eqref{eq:branching process comparison}. 
    Let $\alpha_0 \geq 1$ be such that $\rho_{a}([-\alpha_0, \alpha_0]) > 0$.
    Given $(\mathcal{C}_i)_{i \geq 1} = (V_i)_{i \geq 1}$ with $\bigcup_{i = 1}^{\infty} V_i = V'$, let $t \in [-\alpha_0, \alpha_0]^{\Lambda' \cup V'}$ and define $\tilde{\varphi} \in \mathbb{R}^{\Lambda}$ to be the configuration with $\tilde{\varphi}_x = t_x$ for $x \in \Lambda' \cup V'$ and $\tilde{\varphi_x} = \varphi_x$ for $x \in \Lambda \setminus (\Lambda' \cup V')$.
    Comparing the value of the integrand at $\varphi$ with its value at $\tilde{\varphi}$ and setting $\alpha_1 = 2 \beta C^{2}$, we will see that we can choose $C$ large enough that
    \begin{align}
    \label{compactified inequality}
    &\beta J_{x,y} \varphi_x \varphi_y \leq \beta J_{x,y} \tilde{\varphi}_x \tilde{\varphi}_y\\
    \nonumber 
    & + |J_{x,y}| f(J_{x, y}) \left(\frac{a|\varphi_x|^{n}}{2M_f}\mathbbm{1}_{x\in \mathcal{C}}+\alpha_1  A_x^n\mathbbm{1}_{x\in \Lambda' \setminus \mathcal{C}}+\frac{a|\varphi_y|^{n}}{2M_f}\mathbbm{1}_{y\in \mathcal{C}}+\alpha_1  A_y^n \mathbbm{1}_{y\in \Lambda' \setminus \mathcal{C}} \right).
    \end{align}
    We now verify \eqref{compactified inequality} by considering an edge $xy \in E$ and splitting into cases based on whether $x$ and $y$ are in $\mathcal{C}$, $\Lambda' \setminus \mathcal{C}$, or $\Lambda \setminus (\Lambda' \cup \mathcal{C})$.
    In each case we will use the bound $\beta J_{x,y}(\varphi_x \varphi_y - \tilde{\varphi}_x \tilde{\varphi}_y) \leq \beta |J_{x,y}| (|\varphi_x| |\varphi_y| + |\tilde{\varphi}_x| |\tilde{\varphi}_y|)$.
    \begin{itemize}
        \item
        If $x, y \in \mathcal{C}$, then $|\varphi_x|, |\varphi_y| \geq C$, so by Lemma \ref{edge removal general 1}, if $C\geq \alpha_0+\left(\frac{4M_f\beta}{a}\right)^{\frac{1}{n-2}}$, then
        \begin{equation*}
        \beta (|\varphi_x| |\varphi_y| + |t_x||t_y|) \leq 
        \frac{a f(J_{x, y})}{2 M_f} ( |\varphi_x|^{n} +  |\varphi_y|^{n}).
        \end{equation*}
        
        \item
        If $x \in \mathcal{C}$ and $y \in \Lambda \setminus (\Lambda' \cup \mathcal{C})$,
        then there exists a walk $x_0, x_1, \ldots, x_m$ from $\Lambda'$ to $x$ in 
        $\overline{(\Lambda, J)}$ satisfying \eqref{cluster membership criteria},
        but there is no such walk from $\Lambda'$ to $y$. Hence the walk $x_0,  \ldots, x_m, y$ does not satisfy \eqref{cluster membership criteria}, so
        \begin{align*}
            |\varphi_y| &\leq  C A_{x_0}^{(n-1)^{m+1}} f(J_{x, y}) \prod_{i = 1}^{m} f(J_{x_{i-1}, x_i})^{(n-1)^{m+1-i}}\\
            &= \frac{f(J_{x, y})}{C^{n-2}} \left(C A_{x_0}^{(n-1)^{m}} \prod_{i = 1}^{m} f(J_{x_{i-1}, x_i})^{(n-1)^{m-i}} \right)^{n-1} 
            \leq \frac{f(J_{x, y})}{C^{n-2}} |\varphi_x|^{n-1}.
        \end{align*}
        Combining the latter with Lemma \ref{edge removal general 2} we get
        \begin{equation*}
            \beta (|\varphi_x| |\varphi_y| + |t_x||\varphi_y|) \leq \frac{a f(J_{x, y})}{2 M_f} |\varphi_x|^{n}.
        \end{equation*}

        \item
        If $x \in \mathcal{C}$ and $y \in \Lambda' \setminus \mathcal{C}$, then applying Lemma~\ref{edge removal general 1} with $C \geq \alpha_0+\left(\frac{4M_f\beta}{a}\right)^{\frac{1}{n-2}}$ and using that $|\varphi_y|\leq CA_y$, we get 
        \begin{align*}
            \beta (|\varphi_x| |\varphi_y| + |t_x| |t_y|)
            &\leq \frac{2\beta}{C^{n-2}}(|\varphi_x|^{n}+(CA_y)^n)
            \leq \frac{a f(J_{x, y})}{2 M_f}|\varphi_x|^{n} + 2\beta C^{2} A_{y}^{n}.
        \end{align*}

        \item 
        If $x, y \in \Lambda' \setminus \mathcal{C}$, then $|\varphi_x| \leq CA_x$ and $|\varphi_y| \leq CA_y  \leq CA_{x}^{n-1}f(J_{x, y})$ by Lemma \ref{A comparison}, so
        \begin{equation*}
            \beta (|\varphi_x| |\varphi_y| + |t_x| |t_y|) \leq 
            2 \beta C^{2}  A_x^n f(J_{x,y}).
        \end{equation*}

        \item
        If $x \in \Lambda' \setminus \mathcal{C}$ and $y \in \Lambda \setminus (\Lambda' \cup \mathcal{C})$,
        then $|\varphi_x| \leq CA_x$ and $|\varphi_y| \leq CA_{x}^{n-1}f(J_{x, y})$, so
        \begin{align*}
            \beta (|\varphi_x| |\varphi_y| + |t_x||\varphi_y|) \leq  2 \beta C^{2}  A_x^n f(J_{x,y}).
        \end{align*}
    \end{itemize}

    The terms $e^{\beta h_{x,\Lambda} \varphi_x}$ coming from interaction with the spins outside $\Lambda$ can be bounded similarly, using that
    if $x \in \mathcal{C}$, then by definition of $\mathcal{C}$ and $A_x$ we have
    \[
    |h_{x, \Lambda}| \leq \left(\sum_{y \in V \setminus \Lambda} |J_{x, y}| f(J_{x,y}) \right) \frac{|\varphi_x|^{n-1}}{C^{n-2}},
    \]
    so Lemma \ref{edge removal general 2} applies.
    If $x \in \Lambda' \setminus \mathcal{C}$ then we bound $|h_{x, \Lambda}|$ using the definition of $A_x$. Overall, we obtain
    \begin{equation}
        \label{compactified inequality 2}
        \beta h_{x,\Lambda} \varphi_x \leq \beta h_{x, \Lambda} \tilde{\varphi}_x +  \left( \frac{a|\varphi_x|^{n}}{2M_f}\mathbbm{1}_{x\in \mathcal{C}}+\alpha_1  A_x^n\mathbbm{1}_{x\in \Lambda' \setminus \mathcal{C}} \right) \sum_{y \in V \setminus \Lambda} |J_{x, y}| f(J_{x,y}).
    \end{equation}

Now \eqref{eq:branching process comparison} follows from Lemma \ref{lem:inside proof} below, which we also state in the case $n=2$.
It only remains to show that we can choose $C$ and $\tilde{C}$ of the required form. Note that $C$ depends on $\beta$ only through the condition $C\geq \alpha_0+\left(\frac{4M_f\beta}{a}\right)^{\frac{1}{n-2}}$ when applying Lemmas \ref{edge removal general 1} and \ref{edge removal general 2}. We can choose $\tilde{C} = \log(\alpha_2 \int_{\mathbb{R}} \mathrm{d} \rho_{\frac{a}{2}}(u)) + \alpha_1 M_f$, where $\alpha_2$ is given by Lemma \ref{lem:inside proof 2} and does not depend on $\beta$. Hence the only dependence of $\tilde{C}$ on $\beta$ comes from $\alpha_1 = 2 \beta C^{2}$. 
\end{proof}

\begin{lemma}\label{lem:inside proof}
Assume that \eqref{C1}, \eqref{C2}, \eqref{compactified inequality} and \eqref{compactified inequality 2} hold for some $n\geq 2$, $a>0,\alpha_1>0$ and $A_x\geq 1$, where $\tilde{\varphi}$ in \eqref{compactified inequality} and \eqref{compactified inequality 2} is defined as above for $\alpha_0 \geq 1$ such that $\rho_{a}([-\alpha_0, \alpha_0]) > 0$. Then there exists a constant $K=K(a,\alpha_0) \geq 0$ such that
\begin{align*}
        &\mathrm{d}\nu_{\Lambda, \beta, \rho, J}^{\xi}  \left[
        \varphi|_{\Lambda'}, \, (\mathcal{C}_i)_{i \geq 1} = (V_i)_{i \geq 1} \right] \leq \\
        \nonumber
        &\exp \left(K|\Lambda'\cup V'| \right) \left( \prod_{x \in \Lambda'} e^{\alpha_1 M_f A_{x}^{n}} \mathrm{d} \rho_{\frac{a}{2}}(\varphi_x) \right)
        \prod_{i = 0}^{\infty} \prod_{y \in V_{i+1}} \max_{x \in V_i} p_{x, y}.
    \end{align*}    
\end{lemma}    
\begin{proof}
For ease of notation, write $P_{\Lambda}(\varphi) = \prod_{x \in \Lambda} e^{-a|\varphi_x|^{n}}$ and
\[
\pi_{E}(\varphi) = \left(\prod_{xy \in E} e^{\beta J_{x, y} \varphi_x \varphi_y} \right) \left( \prod_{x \in \Lambda} e^{\beta h_{x, \Lambda} \varphi_x}\right).
\]
Note that \eqref{long range definition} gives that
\begin{align*}
    &\mathrm{d} \nu_{\Lambda, \beta, \rho, J}^{\xi} \left[
    \varphi|_{\Lambda'}, \, (\mathcal{C}_i)_{i \geq 1} = (V_i)_{i \geq 1} \right]
    = \frac{1}{Z_{\Lambda, \beta, \rho, J}^{\xi}} \int_{\mathbb{R}^{\Lambda \setminus \Lambda'}}
    \mathbbm{1}_{\{(\mathcal{C}_i)_{i \geq 1} = (V_i)_{i \geq 1}\}}
    \pi_{E}(\varphi) P_{\Lambda}(\varphi)
    \prod_{x \in \Lambda} \mathrm{d} \rho_{a}(\varphi_x),
\end{align*}
    where $\rho_a$ satisfies $0 < \rho_{a}(\mathbb{R})<\infty$. 
    We estimate $\pi_{E}(\varphi)$ by applying \eqref{compactified inequality} to each element of the first product and \eqref{compactified inequality 2} to each element of the second product. This yields
    \begin{equation*}
        \pi_{E}(\varphi) \leq  \prod_{x \in \Lambda' \setminus \mathcal{C}}
        \exp \left( \alpha_1 A_{x}^{n} \sum_{y \in V} |J_{x, y}| f(J_{x, y}) \right)
        \prod_{x \in \mathcal{C}} \exp \left( \frac{a|\varphi_x|^{n}}{2 M_f} \sum_{y \in V} |J_{x, y}| f(J_{x, y})\right)
        \pi_{E}(\tilde{\varphi}).
    \end{equation*}
    Using \eqref{C2} the above inequality simplifies to
    \begin{align*}
        \pi_{E}(\varphi)
        &\leq \left( \prod_{x \in \Lambda'} e^{\alpha_1 M_f A_{x}^{n}} \right)
        \left(\prod_{x \in \mathcal{C}} e^{\frac{a}{2}|\varphi_x|^{n}}\right) \pi_{E}(\tilde{\varphi}).
    \end{align*}
    Combining this with the product over vertices of the terms coming from the single-site measure, and using that $|t_x| \leq \alpha_0$ for any $x \in \Lambda' \cup V'$, we get
    \begin{align}\label{eq:pi P inequality}
        \pi_{E}(\varphi) P_{\Lambda}(\varphi)
        &\leq \left( \prod_{x \in \Lambda'} e^{\alpha_1 M_f A_{x}^{n}} \right)
        \left( \prod_{x \in \Lambda'\cup V'} e^{-\frac{a}{2}|\varphi_x|^{n}}  \right)
        \pi_{E}(\tilde{\varphi}) P_{\Lambda \setminus (\Lambda' \cup V')}(\tilde{\varphi})\\
        &\leq \exp \left(a \alpha_{0}^{n} |\Lambda'\cup V'| \right) \left( \prod_{x \in \Lambda'} e^{\alpha_1 M_f A_{x}^{n}} \right) \left( \prod_{x \in \Lambda'\cup V'} e^{-\frac{a}{2}|\varphi_x|^{n}}  \right)
         \pi_{E}(\tilde{\varphi}) P_{\Lambda}(\tilde{\varphi}) \notag.
    \end{align}
    We now integrate with respect to $\varphi_x$ for each $x \in \Lambda \setminus \Lambda'$ over the event $\{(\mathcal{C}_i)_{i \geq 1} = (V_i)_{i \geq 1}\}$.
    Observe that, on this event, if $y \in V_{i + 1}$ for some $i \in \{0, 1, \ldots\}$, then $|\varphi_y| \geq C f(J_{x, y})$ for some $x \in V_i$ with $J_{x, y} \neq 0$. Hence, ignoring the requirement for spins outside $\mathcal{C}$ to be small, we have
    \begin{equation}\label{eq: integral p bound}
        \int_{\mathbb{R}^{V'}} \mathbbm{1}_{\{(\mathcal{C}_i)_{i \geq 1} = (V_i)_{i \geq 1}\}} \prod_{x \in V'} e^{-\frac{a}{2}|\varphi_x|^{n}} \mathrm{d} \rho_{a}(\varphi_x)
        \leq \prod_{i = 0}^{\infty} \prod_{y \in V_{i+1}} \max_{x \in V_i} p_{x, y}.
    \end{equation}
    Note that for any function $F: \mathbb{R}^{\Lambda' \cup V'} \to \mathbb{R}_{\geq 0}$,
    \begin{align}
        \nonumber
        \rho_{a}([-\alpha_0, \alpha_0])^{|\Lambda' \cup V'|}\min_{t \in [-\alpha_0, \alpha_0]^{\Lambda' \cup V'}} F(t) &\leq \int_{[-\alpha_0, \alpha_0]^{\Lambda' \cup V'}} F(t) \prod_{x \in \Lambda' \cup V'}\mathrm{d} \rho_{a}(t_x)\\
        \label{eq:min}
        &\leq \int_{\mathbb{R}^{\Lambda' \cup V'}}F(t) \prod_{x \in \Lambda' \cup V'}\mathrm{d} \rho_{a}(t_x).
    \end{align}
    Thus, since $t \in [-\alpha_0, \alpha_0]^{\Lambda' \cup V'}$ is arbitrary in the definition of $\tilde\varphi$,  integrating \eqref{eq:pi P inequality} and dividing by $Z_{\Lambda, \beta, \rho, J}^{\xi}$, applying  \eqref{eq:min} for $F$ being $\pi_{E}(\tilde{\varphi}) P_{\Lambda}(\tilde{\varphi})$, and using \eqref{eq: integral p bound} yields
    \eqref{eq:branching process comparison}, with $K = \max \{0, a \alpha_{0}^{n} - \log(\rho_{a}([-\alpha_0, \alpha_0]))\}$.
\end{proof}

We now state and prove the second lemma used in the proof of Theorem~\ref{main theorem} above, which we also state in the case $n=2$. We use the same notation as in the proof of Theorem~\ref{main theorem}, and we recall that the definition of $p_{x,y}$ involves a constant $C$.

\begin{lemma}\label{lem:inside proof 2} 
Let $n\geq 2$ and $K\geq0$. Then there exist $C_0,\alpha_2\geq 1$ that do not depend on $\beta$ such that for every $C\geq C_0$ we have
\[
\sum_{(V_j)} e^{K(|\Lambda'|+|V'|)}\prod_{i = 0}^{\infty} \prod_{y \in V_{i+1}} \max_{x \in V_i} p_{x, y}\leq \alpha_2^{|\Lambda'|}.
\]        
\end{lemma}
\begin{proof}
We aim to compare $\prod_{i = 0}^{\infty} \prod_{y \in V_{i+1}} \max_{x \in V_i} p_{x, y}$ with the probability that $W_i = V_i$ for all $i \geq 1$, where $(W_{i})_{i \geq 0}$ is the exploration process defined as follows. First set $W_0 = \Lambda'$ and ensure that $C$ is chosen large enough that $\int_{|u| \geq C} \mathrm{d} \rho_{\frac{a}{2}}(u) \leq 1$. Assuming that $W_i$ has already been constructed, for each vertex $x \in W_i$ and $y \in \Lambda \setminus (W_0 \cup \ldots \cup W_i)$, say that the edge $xy$ is open with probability $p_{x, y}$, independently of all other edges. Otherwise, say $xy$ is closed. Then set $W_{i+1}$ to be the set of vertices $y \in \Lambda \setminus (W_0 \cup \ldots \cup W_i)$ such that $xy$ is open for some $x \in W_i$. Once $W_i$ has been constructed for all $i \in \{0, 1, \ldots \}$, set $W = \bigcup_{i=1}^{\infty} W_i$. With this definition, we have
    \begin{align*}
        &\mathbb{P}[W_{i+1} = V_{i+1} | W_1 = V_1, \ldots, W_i = V_i] =\\
        &\prod_{y \in \Lambda \setminus (V_0 \cup \ldots \cup V_{i+1})}
        \mathbb{P} \left[ \bigcap_{x \in V_i} \{ xy \text{ closed} \}\right]
        \prod_{y \in V_{i+1}} \mathbb{P} \left[ \bigcup_{x \in V_i} \{ xy \text{ open} \}\right]
        \geq b^{|V_i|} \prod_{y \in V_{i+1}}  \max_{x \in V_i} p_{x, y},
    \end{align*}
    where
    \begin{equation*}
        b = \inf_{x \in V} \prod_{y \in V \setminus \{x\}} (1-p_{x, y}).
    \end{equation*}
    Combining over all generations of the exploration process, we obtain
    \begin{align}
        \nonumber
        \mathbb{P}[(W_{i})_{i \geq 1} = (V_i)_{i \geq 1}] &=
        \prod_{i = 0}^{\infty} \mathbb{P}[W_{i+1} = V_{i+1} | W_1 = V_1, \ldots, W_i = V_i]\\
        \label{finished exploration}
        & \geq \prod_{i = 0}^{\infty}  b^{|V_i|} \prod_{y \in V_{i+1}}  \max_{x \in V_i} p_{x, y}
        = b^{|\Lambda'| + |V'|}
        \prod_{i = 0}^{\infty} \prod_{y \in V_{i+1}}  \max_{x \in V_i} p_{x, y}.
    \end{align}
        
    We now claim that $\sup_{x \in V} \sum_{y \in V \setminus \{x\}} p_{x, y}$ tends to $0$ as $C$ tends to infinity, which in turn implies that $b$ tends to $1$.
    For any $x, y \in V$ with $J_{x, y} \neq 0$, 
    we have
    \begin{align*}
        p_{x, y} = \int_{|u| \geq C f(J_{x, y})} e^{-\frac{a}{2}|u|^{n}} \, \mathrm{d} \rho_{a}(u) \leq \exp \left( -\frac{a}{2}
        (Cf(J_{x,y}))^{n} \right) \rho_{a}(\mathbb{R}).
    \end{align*}
    Given $x,y \in V$ with $|J_{x, y}| < \delta_f$, we have by \eqref{C2} that $ p_{x, y} \leq \rho_{a}(\mathbb{R}) |J_{x, y}|^{\frac{a}{2}C^{n}}$, and we can choose $C>0$ to be large enough so that $\rho_{a}(\mathbb{R}) |J_{x, y}|^{\frac{a}{2}C^{n}}\leq \tfrac{|J_{x, y}|}{CM_f}$. 
    Since $\sum_{y \in V} |J_{x, y}| \leq M_f$,
    there are at most $M_f/\delta_f$ vertices $y\in V$ such that $|J_{x,y}|\geq \delta_f$ and for these vertices, we can use the bound $
    p_{x, y} \leq e^{-a
        C^{n}/2} \rho_{a}(\mathbb{R})$.
    The claim follows.
    
    Let $\alpha_3 = b^{-1} e^{K} \in [1, \infty)$. It follows from summing \eqref{finished exploration} over $V_1, V_2, \ldots$ that
    \begin{equation}
        \label{final_bound}
        \sum_{(V_j)} e^{K(|\Lambda'|+|V'|)}\prod_{i = 0}^{\infty} \prod_{y \in V_{i+1}} \max_{x \in V_i} p_{x, y}\leq \alpha_{3}^{|\Lambda'|}
        \sum_{k = 0}^{\infty} \alpha_{3}^k \mathbb{P}[|W| = k].
    \end{equation}
    We want to stochastically dominate $(W_i)_{i \geq 0}$ by a branching process $(Z_i)_{i \geq 0}$.
    By choosing the value of $C$ to be large enough, we can ensure that $b \geq 1/2$. We can then define a random variable $X$ by
    \begin{equation*}
        \mathbb{P}[X = k] = 
        \begin{cases}
        b &\text{if } k = 0,\\
        1 - b - \frac{(1-b)^2}{b} &\text{if } k = 1, \\
        (1-b)^k &\text{if } k \in \{2, 3, \ldots\},
        \end{cases}
    \end{equation*}
    and let $(Z_i)_{i \geq 0}$ be a branching process with initial population $|\Lambda'|$ and offspring distribution $X$.
    If $x \in W_i$ for some $i \in  \{0,1,\ldots\}$, let $X_x$ be the number of vertices $y \in W_{i+1}$ such that the edge $xy$ is open.
    For all $k \geq 0$, we have
    $\mathbb{P}[X_x \geq k] \leq \mathbb{P}[X_x \ge 1]^k \leq (1-b)^k \leq \mathbb{P}[X \geq k]$, which proves the desired stochastic domination. In particular, since $\alpha_3 \geq 1$ we have that
    \begin{align*}
        \nonumber
        \sum_{k = 0}^{\infty} \alpha_{3}^k \mathbb{P}[|W| = k] \leq \sum_{k = 0}^{\infty} \alpha_{3}^k \mathbb{P}[T = k + |\Lambda'|],
    \end{align*}
    where $T$ is the total progeny of $(Z_i)_{i \geq 0}$. Applying Theorem \ref{progeny theorem} gives that
    \begin{align}
        \nonumber
        \sum_{k = 0}^{\infty} \alpha_{3}^k \mathbb{P}[|W| = k]
        &\leq \alpha_3^{|\Lambda'|} + \sum_{k = |\Lambda'|}^{\infty} \alpha_{3}^k \frac{|\Lambda'|}{k + |\Lambda'|} \mathbb{P}[X_1 + \ldots + X_{k + |\Lambda'|} = k]\\
        \label{sum of X 1}
        & \leq \alpha_3^{|\Lambda'|} + \frac{1}{2}\sum_{k = |\Lambda'|}^{\infty} \alpha_3^{k} \mathbb{P}[X_1 + \ldots + X_{2k} \geq k],
    \end{align}
    where $X_1, X_2,\ldots$ are independent with the same distribution as $X$.
    We now bound the latter probability as follows. 
    Setting $\theta = 2 \log(2\alpha_3)$, we have that $\mathbb{E}[e^{\theta X}] \rightarrow 1 $ as $b \rightarrow 1$. By increasing the value of $C$, we can make $b$ as close to 1 as desired, so by choosing $C$ large enough, we can ensure that $\mathbb{E}[e^{\theta X}] < e^{\theta/4}$. We then have by the exponential Markov inequality and independence
    \begin{equation}
        \label{sum of X 2}
        \mathbb{P}[X_1 + \ldots + X_{2k} \geq k] \leq e^{-\theta k} \mathbb{E}[e^{\theta X}]^{2k}
        \leq e^{-\theta k/2} = (2\alpha_3)^{-k}.
    \end{equation}
    Combining \eqref{sum of X 1} and \eqref{sum of X 2} yields
    \begin{equation*}
        \sum_{k=0}^{\infty} \alpha_{3}^{k} \mathbb{P}[|W| = k] \leq \alpha_3^{|\Lambda'|} + 2^{-|\Lambda'|},
    \end{equation*}
    and substituting this in \eqref{final_bound} completes the proof.
\end{proof}    

\begin{remark}
   In the case of nearest-neighbour interactions on a graph of bounded degree, the proof of Lemma~\ref{lem:inside proof 2} can be simplified by using the fact that the number of ways to choose $\mathcal{C} \setminus \Lambda'$ so that $|\mathcal{C} \setminus \Lambda'| = k$ is at most exponential in $k+|\Lambda'|$.
\end{remark}

\section{Regularity for related models}
\label{section: related models}
In this section, we aim to generalise Theorem \ref{main theorem}. We first show that our arguments can also be applied when $\rho$ satisfies \eqref{eq: single-site assumption} with $n=2$ for some $a \geq 4 \beta M_f$, with the other assumptions from Section \ref{section:def of model} unchanged.

Recall from Section \ref{section: bc notation} the definition of $A(x, \Lambda)$ in the case $n=2$.
The assumption that $a \geq 4 \beta M_f$ is stronger than is necessary for our arguments or the arguments of \cite{Lebowitz_Presutti} to apply, but in later applications we only consider the case when $\rho$ satisfies \eqref{eq: single-site n=2},
so we include this assumption for simplicity.
The theorem below is the analogue of Theorem \ref{main theorem} in the case $n=2$.

\begin{theorem}
\label{Theorem: n=2}
    Let $a\geq 4 \beta M_f$ and assume $\rho$ satisfies $\int_{\mathbb{R}} e^{a |u|^{2}} \mathrm{d} \rho(u) < \infty$.
    There exist $C \geq 1, \tilde{C} > 0$ depending only on $\beta, \delta_f, M_f, \rho$ and $a$ such that for any $\Lambda \Subset V$, $\Lambda' \subset \Lambda$, $\psi \in \mathbb{R}^{\Lambda'}$, $\lambda\leq \frac{a}{4 \beta M_f}$, and any boundary conditions $\xi \in \mathbb{R}^{V}$ with $\sum_{y \in V} |J_{x, y} \xi_y| < \infty$ for all $x \in \Lambda$,
    \begin{equation*}
        \mathrm{d} \nu_{\Lambda, \beta, \rho, J}^{\xi}[\varphi|_{\Lambda'} = \psi]
        \leq \prod_{x \in \Lambda'} e^{\tilde{C} A(x, \Lambda, \xi, C)^2} \mathrm{d} \nu_{\Lambda', 0, \rho_{\frac{a}{2}}, 0}^{0} [\psi].
    \end{equation*} 
\end{theorem}

When $\rho$ satisfies \eqref{eq: single-site n=2}, it follows from Theorem \ref{Theorem: n=2} that we have tightness for any boundary conditions that are in $\Xi(\lambda)$ for some $\lambda \geq 1$. For nearest-neighbour interactions on a graph $G$, this includes exponentially growing boundary conditions of the form $|\xi_x| \leq C \lambda^{d_G(o, x)}$ (see \eqref{eq: short range Xi n=2}), so we observe that the threshold for tightness jumps from exponential when $n=2$ to double exponential when $n > 2$.

The proof of Theorem \ref{Theorem: n=2} is essentially the same as that of Theorem \ref{main theorem} but with different definitions of $\mathcal{C}$ and $A$.

\begin{proof}[Proof of Theorem \ref{Theorem: n=2}]
    Fix a finite subset $\Lambda \subset V$ and let $\Lambda' \subset \Lambda$. We write $E$ for $\{xy \in \overline{E}(\Lambda, J): x,y \in \Lambda\}$ and $A_x$ for $A(x, \Lambda, \lambda,  \xi, C, J, f)$, where $C \geq 1$ is a constant to be determined.
    Define $\mathcal{C}$ to be the set of vertices $x \in \Lambda$ such that for some $m \in \{ 0, 1, \ldots\}$ there exists a walk $x_0, x_1, \ldots, x_m$ from $\Lambda'$ to $x$ in $\overline{(\Lambda, J)}$ that satisfies
    \begin{equation}
        \label{cluster membership criteria n=2}
        \forall k \in S_m \quad
        |\varphi_{x_k}| \geq C A_{x_0} \lambda^{k} \prod_{i = 1}^{k} f(J_{x_{i-1}, x_i}),
    \end{equation}
    where $S_0 = \{0\}$ and $S_m = \{1, \ldots, m\}$ for $m \geq 1$.
    For each $i \in \{ 0, 1, \ldots\}$, let $\mathcal{C}_i\subset \mathcal{C}$ denote the set of vertices for which $i$ is the smallest value of $m$ such that there exists a walk satisfying \eqref{cluster membership criteria n=2}.
    Let $\alpha_0 \geq 1$ be such that $\rho_{a}([-\alpha_0, \alpha_0]) > 0$.
    Given $(\mathcal{C}_i)_{i \geq 1} = (V_i)_{i \geq 1}$ with $\bigcup_{i = 1}^{\infty} V_i = V'$, let $t \in [-\alpha_0, \alpha_0]^{\Lambda' \cup V'}$ and define the configuration $\tilde{\varphi} \in \mathbb{R}^{\Lambda}$ as in the proof of Theorem \ref{main theorem}.
    Setting $\alpha_1 = \frac{a  C^{2}}{2 M_f} \geq 2 \beta C^{2} \lambda$, one can show that when $C \geq \alpha_0$
    \begin{align}
    \label{compactified inequality n=2}
    &\beta J_{x,y} \varphi_x \varphi_y \leq\\
    \nonumber
    &\beta J_{x,y} \tilde{\varphi}_x \tilde{\varphi}_y +|J_{x,y}|f(J_{x, y}) \left(\frac{a|\varphi_x|^{2}}{2M_f}\mathbbm{1}_{x\in \mathcal{C}}+\alpha_1 A_x^2 \mathbbm{1}_{x\in \Lambda' \setminus \mathcal{C}}+\frac{a|\varphi_y|^{2}}{2M_f}\mathbbm{1}_{y\in \mathcal{C}}+\alpha_1 A_x^2 \mathbbm{1}_{y\in \Lambda' \setminus \mathcal{C}} \right),
    \end{align}
    and
    \begin{equation}
        \label{compactified inequality 2 (n=2)}
        \beta h_{x,\Lambda} \varphi_x \leq \beta h_{x, \Lambda} \tilde{\varphi}_x +  \left( \frac{a|\varphi_x|^2}{2M_f}\mathbbm{1}_{x\in \mathcal{C}}+\alpha_1 A_x^2 \mathbbm{1}_{x\in \Lambda' \setminus \mathcal{C}} \right) \sum_{y \in V \setminus \Lambda} |J_{x, y}| f(J_{x,y}).
    \end{equation}
    The first inequality \eqref{compactified inequality n=2} can be verified in a similar way to \eqref{compactified inequality}. Indeed, if $|\varphi_x|, |\varphi_y| \geq \alpha_0$, then applying Lemma \ref{edge removal general 1} and using that $a \geq 4 \beta M_f$ gives that
    \begin{align*}
        \beta (|\varphi_x| |\varphi_y| + |t_x||t_y|) 
        \leq \frac{a f(J_{x,y})}{2M_f} (|\varphi_x|^{2} + |\varphi_y|^{2}).
    \end{align*} 
    We also use that if $x\in \mathcal{C}$ and $y\in \Lambda\setminus \mathcal{C}$, then $|\varphi_x| \geq \alpha_0$ and $|\varphi_y| \leq \lambda f(J_{x,y}) |\varphi_x|$, which implies
    \begin{align}
        \label{Edge removal alt n=2 (i)}
        \beta (|\varphi_x| |\varphi_y| + |t_x||\varphi_y|)
        \leq 2 \beta \lambda f(J_{x,y}) |\varphi_x|^{2} 
        \leq \frac{a f(J_{x,y})}{2M_f} |\varphi_x|^{2}.
    \end{align}
    To prove \eqref{compactified inequality 2 (n=2)}, if $x \in \mathcal{C}$ then 
    \begin{align*}
        |h_{x, \Lambda}|\leq \left(\sum_{y \in V \setminus \Lambda} |J_{x, y}| f(J_{x,y}) \right) \lambda |\varphi_x|,
    \end{align*}
    so \eqref{Edge removal alt n=2 (i)} applies.
    If $x \in \Lambda' \setminus \mathcal{C}$ then we bound $|h_{x, \Lambda}|$ using the definition of $A_{x}$ to obtain
    \begin{equation*}
        |h_{x, \Lambda}| \leq C \lambda A_x \left(\sum_{y \in V \setminus \Lambda} |J_{x, y}| f(J_{x,y}) \right).
    \end{equation*}
    Having obtained \eqref{compactified inequality n=2} and \eqref{compactified inequality 2 (n=2)}, we use Lemmas~\ref{lem:inside proof} and \ref{lem:inside proof 2} as in the proof of Theorem \ref{main theorem} to conclude.
\end{proof}

\begin{remark}
\label{remark:changed single-site}
Theorem~\ref{Theorem: n=2} can be generalised by allowing the single-site measure to depend on the vertex. We will use such a generalisation to construct the infinite-volume plus measure as the limit of systems with a shifted single-site measure at the boundary. 

Suppose the single-site measure at vertex $x$ is given by $d\rho_{x, \Lambda}(u) = e^{-a_{x, \Lambda}|u|^2}d\mu_{x, \Lambda}(u)$ and there exist a bounded subset $T \subset \mathbb{R}$ and constants $a_{\min}, a_{\max}, M_1, M_2 >0$ such that for all $\Lambda \Subset V$ and $x \in \Lambda$,
\begin{enumerate}
\renewcommand{\labelenumi}{\textbf{(A\arabic{enumi})}}
\renewcommand{\theenumi}{A\arabic{enumi}}
\item \label{A1}
$a_{\min} \leq a_{x, \Lambda} \leq a_{\max}$,
\item \label{A2}
$\mu_{x, \Lambda}(T) \geq M_1$,
\item \label{A3}
$\mu_{x, \Lambda}(\mathbb{R}) \leq M_2$.
\end{enumerate}
Then there exist $C \geq 1, \tilde{C} > 0$ such that 
for any $\Lambda \Subset V$, $\Lambda' \subset \Lambda$, $\psi \in \mathbb{R}^{\Lambda'}$, $\lambda\leq \frac{a_{\min}}{4 \beta M_f}$, and any boundary conditions $\xi \in \mathbb{R}^{V}$ with $\sum_{y \in V} |J_{x, y} \xi_y| < \infty$ for all $x \in \Lambda$,
\begin{equation*}
    \mathrm{d} \nu_{\Lambda, \beta, \rho, J}^{\xi}[\varphi|_{\Lambda'} = \psi]
    \leq \left( \prod_{x \in \Lambda'} e^{\tilde{C} A(x, \Lambda, \xi, C)^{2}} \right) \mathrm{d} \nu_{\Lambda', 0, \tilde{\rho}, 0}^{0} [\psi],
\end{equation*}
where $\tilde{\rho}$ is given by $d\tilde{\rho}_{x}(u) = e^{-\frac{1}{2}a_{x, \Lambda}|u|^2}d\mu_{x, \Lambda}(u)$.

Assumption \eqref{A2} is used in \eqref{eq:min} while assumptions \eqref{A1} and \eqref{A3} are used to bound 
\begin{equation}
\label{changed pxy def}
p_{x,y} = \int_{|u| \geq C f(J_{x,y})} e^{-\frac{a_{y, \Lambda}}{2}} \mathrm{d} \mu_{y, \Lambda}(u)
\end{equation}
for $J_{x,y} \neq 0$. Theorem \ref{main theorem} also holds for any single-site measures satisfying \eqref{A1}, \eqref{A2}, \eqref{A3}.
\end{remark} 

Changing the Hamiltonian can also be considered. We will still restrict our attention to pairwise interactions and will assume further that interactions occur only between neighbours on a graph with bounded degree, so there is an upper bound on the number of vertices that any given vertex can interact with. 

Let $G=(V, E)$ be a graph with bounded degree. For $\Lambda \Subset V$ and boundary conditions $\xi \in \mathbb{R}^{V}$, define the measure $\nu_{\Lambda, U, \rho}^{\xi}$ by
\begin{equation*}
        \mathrm{d} \nu_{\Lambda, U, \rho}^{\xi}[\varphi] = \frac{1}{Z_{\Lambda, U, \rho}^{\xi}} \prod_{\substack{xy \in E\\x, y \in \Lambda}} U_{xy}(\varphi_x, \varphi_y) \prod_{\substack{xy \in E\\x \in \Lambda, y \in V \setminus \Lambda}} U_{xy}(\varphi_x, \xi_y) \prod_{x \in \Lambda} \mathrm{d} \rho_{x, \Lambda}(\varphi_x),
\end{equation*}
for $\varphi \in \mathbb{R}^{\Lambda}$, where $Z_{\Lambda, U, \rho}^{\xi}$ is the partition function and for each $xy \in E$, $U_{xy}: \mathbb{R}^{2} \rightarrow \mathbb{R}^{+}$ is a function.

\begin{theorem}
\label{changed theorem n=2}
Let $D \geq 1$ and let $G=(V, E)$ be a graph such that $\mathrm{deg}(x) \leq D$ for all $x \in V$.
Let $\rho_{x, \Lambda}$ be single-site measures satisfying \eqref{eq: single-site n=2} and assumptions \eqref{A1},\eqref{A2},\eqref{A3}.
For each $xy \in E$, let $U_{xy}: \mathbb{R}^{2} \rightarrow \mathbb{R}^{+}$ be a function satisfying the following assumptions for 
some constants $C \geq 1$, $\lambda \geq 1$, and function $F: [1, \infty) \rightarrow [1, \infty)$:
\begin{enumerate}[(i)]
    \item If $|\varphi_x| \geq C$, $t_x\in T$ and $|\varphi_y| \leq \lambda |\varphi_x|$, then
    \begin{equation*}
        U_{xy}(\varphi_x, \varphi_y) \leq U_{xy}(t_x, \varphi_y) \exp\left(\frac{a_{x, \Lambda}}{2D}|\varphi_x|^{2}\right).
    \end{equation*}
    \item If $t_x,t_y\in T$ and $|\varphi_x| \leq C A_x, |\varphi_y| \leq C \lambda A_x$ for some $A_x \geq 1$, then
    \begin{equation*}
        \max\left\{ \frac{U_{xy}(t_x, \varphi_y)}{U_{xy}(t_x, t_y)}, \frac{U_{xy}(\varphi_x, \varphi_y)}{U_{xy}(t_x, t_y)}, \frac{U_{xy}(\varphi_x, \varphi_y)}{U_{xy}(t_x, \varphi_y)} \right\} \leq e^{F(A_x)}.
    \end{equation*}
\end{enumerate}
Then there exist constants $C_1, C_2$ such that for any $\Lambda \subset V$ finite, $\Lambda' \subset \Lambda$, $\psi \in \mathbb{R}^{\Lambda'}$, and any boundary conditions $\xi \in \mathbb{R}^{V}$ satisfying $\sum_{y \in V} |J_{x, y} \xi_y| < \infty$ for all $x \in \Lambda$,
\begin{equation*}
    \mathrm{d} \nu_{\Lambda, U, \rho}^{\xi}[\varphi|_{\Lambda'} = \psi]
    \leq \prod_{x \in \Lambda'} \exp \left(C_1 F(A(x, \Lambda, \lambda, \xi, C_2)) -\frac{1}{2} a_{x, \Lambda} |\psi_x|^{2} \right) \mathrm{d} \mu_{x, \Lambda} (\psi_x).
\end{equation*}
\end{theorem}

\begin{proof}[Proof of Theorem \ref{changed theorem n=2}]
    Let $f(1) = 1$ and write $A_x$ for $A(x, \Lambda, \lambda, \xi, C_2, J, f)$,
    where $C_2 \geq 1$ is a constant to be determined. Define $\mathcal{C}$ to be the set of vertices $x \in \Lambda$ such that for some $m \in \{ 0, 1, \ldots\}$ there exists a walk $x_0, x_1, \ldots, x_m$ from $\Lambda'$ to $x$ in $\overline{(\Lambda, J)}$ that satisfies
    \begin{equation}
        \label{changed cluster membership criteria}
        \forall k \in S_m \quad
        |\varphi_{x_k}| \geq C_2 A_{x_0} \lambda^{k},
    \end{equation}
    where $S_0 = \{0\}$ and $S_m = \{1, \ldots, m\}$ for $m \geq 1$.
    For each $i \in \{ 0, 1, \ldots\}$, let $\mathcal{C}_i\subset \mathcal{C}$ denote the set of vertices for which $i$ is the smallest value of $m$ such that there exists a walk satisfying \eqref{changed cluster membership criteria}.
    Given $(\mathcal{C}_i)_{i \geq 1} = (V_i)_{i \geq 1}$ with $\bigcup_{i = 1}^{\infty} V_i = V'$, let $t \in T^{\Lambda' \cup V'}$ and define the configurations $\varphi', \tilde{\varphi} \in \mathbb{R}^{V}$ by
    \begin{equation*}
        \varphi'_x = \begin{cases}
            \varphi_x & \text{if } x \in \Lambda,\\
            \xi_x & \text{otherwise},
        \end{cases}
        \qquad
        \tilde{\varphi}_x = \begin{cases}
            t_x & \text{if } x \in \Lambda' \cup V',\\
            \varphi'_x & \text{otherwise}.
        \end{cases}
    \end{equation*}
    Assumption (i) in the statement of the theorem is analogous to \eqref{Edge removal alt n=2 (i)}, and (ii) allows us to bound $U_{xy}(\varphi'_x, \varphi'_y)$ in terms of $U_{xy}(\tilde{\varphi}_x, \tilde{\varphi}_y)$ when $x$ or $y$ is in $\Lambda' \setminus \mathcal{C}$. Together with the definitions of $\mathcal{C}$ and $A(o, \Lambda)$, they imply that for any $xy \in E$
    \begin{align*}
        &U_{xy}(\varphi'_x, \varphi'_y) \leq\\
        &U_{xy}(\tilde{\varphi}_x, \tilde{\varphi}_y)
        \exp\left(\frac{a_{x, \Lambda}}{2D}|\varphi_x|^{2} \mathbbm{1}_{x \in \mathcal{C}} + F(A_{x}) \mathbbm{1}_{x \in \Lambda' \setminus \mathcal{C}} + \frac{a_{y, \Lambda}}{2D}|\varphi_y|^{2} \mathbbm{1}_{y \in \mathcal{C}} + F(A_{y}) \mathbbm{1}_{y \in \Lambda' \setminus \mathcal{C}}\right).
    \end{align*}
    Combining over all edges and using that each vertex has degree at most $D$, we get
    \begin{align*}
        \prod_{xy \in E} U_{xy}(\varphi'_x, \varphi'_y)
        \leq \left(\prod_{xy \in E} U_{xy}(\tilde{\varphi}_x , \tilde{\varphi}_y) \right) \prod_{x \in \mathcal{C}} e^{\frac{a_{x, \Lambda}}{2}|\varphi_x|^{2}} \prod_{x \in \Lambda'} e^{D F(A_{x})}.
    \end{align*}
    By combining the above with the terms coming from the single-site measure and integrating, we can show as in the proof of Lemma \ref{lem:inside proof} that for some $K\geq0$,
    \begin{align*}
        &\mathrm{d}\nu_{\Lambda, U, \rho}^{\xi}  \left[
        \varphi|_{\Lambda'}, \, (\mathcal{C}_i)_{i \geq 1} = (V_i)_{i \geq 1} \right] \leq \\
        \nonumber
        &\exp \left(K|\Lambda'\cup V'| \right) \left( \prod_{x \in \Lambda'} e^{DF(A_x)} e^{\frac{-a_{x, \Lambda}}{2}|\varphi_x|^2} \mathrm{d}\mu_{x, \Lambda}(\varphi_x) \right)
        \prod_{i = 0}^{\infty} \prod_{y \in V_{i+1}} \max_{x \in V_i} p_{x, y},
    \end{align*}
    with $p_{x,y}$ defined as in \eqref{changed pxy def}.
    Applying Lemma \ref{lem:inside proof 2} concludes the proof.
\end{proof}

\begin{remark}\label{rem:conditional measure}
    One example where Theorem \ref{changed theorem n=2} is useful is the random cluster representation of the $\varphi^{4}$ model, introduced in \cite{well_behaved}, which is a measure on pairs $(\mathsf{a},\omega)$, where $\mathsf{a}$ is the absolute value field and $\omega$ is a percolation configuration. We may wish to consider the distribution of $\mathsf{a}$ in this model conditional on observing a given percolation configuration $\omega$, similarly to \cite[Lemma 6.8]{well_behaved}. In this case, the functions $U_{xy}$ are given by
    \begin{equation*}
        U_{xy}(\mathsf{a}_x, \mathsf{a}_y) =
        \begin{cases}
            e^{-\beta \mathsf{a}_x \mathsf{a}_y} & \text{if } \omega_{xy} = 0,\\
            e^{\beta \mathsf{a}_x \mathsf{a}_y} - e^{-\beta \mathsf{a}_x \mathsf{a}_y} & \text{if } \omega_{xy} = 1.
        \end{cases}
    \end{equation*}
    We now check that assumptions (i) and (ii) in Theorem \ref{changed theorem n=2} are satisfied and that the choice of $C, \lambda$ and $F$ does not depend on $\omega$.
    Assume $\rho_{x, \Lambda}$ are single-site measures supported on $\mathbb{R}^{+}$ that satisfy the assumptions of Theorem \ref{changed theorem n=2} and let $\lambda = \frac{a_{\min}}{4 D \beta}$.
    Let $C\geq 1$ be a constant to be determined and let $T = [t_{\min}, t_{\max}]$, where $0 <t_{\min} < t_{\max} \leq C$ and $T$ satisfies \eqref{A2}. Let $xy \in E$. If $\omega_{xy} = 0$, then the interaction term $U_{xy}(\mathsf{a}_x, \mathsf{a}_y)$ is of the same form as in Theorem \ref{Theorem: n=2}, so (i) follows from \eqref{Edge removal alt n=2 (i)}. Now suppose $\omega_{xy} = 1$ and observe that since the spins $\mathsf{a}_x$ only take positive values in this model, $U_{xy}$ is increasing in both arguments. 
    Note that if $\mathsf{a}_x \geq C \geq t_x \ge t_{\min}$ and $\mathsf{a}_y \leq \lambda \mathsf{a}_x$,
    then $\frac{U_{xy} (\mathsf{a}_x, \mathsf{a}_y)}{U_{xy}(t_x, \mathsf{a}_y)}$ is an increasing function of $\mathsf{a}_y$. Hence,
    \begin{align*}
        \frac{U_{xy} (\mathsf{a}_x, \mathsf{a}_y)}{U_{xy}(t_x, \mathsf{a}_y)}
        \leq  \frac{U_{xy} (\mathsf{a}_x, \lambda \mathsf{a}_{x})}{U_{xy}(t_x, \lambda \mathsf{a}_x)}
        \leq \frac{\exp(\beta \lambda \mathsf{a}_x^{2})}{U_{xy}(t_{\min}, C \lambda)}.
    \end{align*}
    The choice of $\lambda$ implies the right hand side above is at most $\exp\left(\frac{a_{\min}}{2D}\mathsf{a}_x^{2}\right)$ for all $C$ large enough, so (i) holds.
    For (ii), we use that if $\omega_{xy} = 1$, the maximum in (ii) is at most
    \begin{equation*}
        \frac{U_{xy}(CA_{x}, C\lambda A_{x})}{U_{xy}(t_{\min}, t_{\min})},
    \end{equation*}
    and if $\omega_{xy} = 0$ then it is at most $\exp(2\beta C^2 \lambda A_{x}^{2})$ by \eqref{compactified inequality n=2}.
\end{remark}

\section{Corollaries and applications}
\label{section: corollaries}
In this section, we give some examples of interactions $J$ and boundary conditions $\xi$ for which we can apply our regularity results, and we then apply them to construct infinite-volume measures. 
Recall from Section \ref{section: bc notation} that $\Xi$ is the set of boundary conditions $\xi$ for which $\tilde{A}(x, V, \xi) < \infty$ for all $x \in V$, and that we have tightness for any $\xi \in \Xi$. 
For nearest-neighbour interactions, we will give a full characterisation of $\Xi$ and show for certain choices of $\rho$ that $\xi \in \Xi$ is necessary to obtain tightness in the case of non-negative boundary conditions. We also give examples of boundary conditions that are in $\Xi$ for different forms of long-range interactions. Later we give conditions on $\xi$ that ensure the measures $\nu_{\Lambda, \beta, \rho, J}^{\xi}$ converge to an $a$-regular Gibbs measure as $\Lambda \nearrow V$ and construct the extremal regular Gibbs measures $\nu^{+}$ and $\nu^{-}$.

We begin by defining some notation that will be used throughout this section.
Assume $G = (V, E)$ is an infinite connected graph such that every vertex has finite degree, and fix an origin $o \in V$. Let $d_G: V \times V \rightarrow \mathbb{N}_{0}$ be the graph distance in $G$, and for $S \subset V$ let $d_S$ denote the graph distance in the subgraph of $G$ induced by $S$.
For $x \in V$, let $\mathrm{deg}(x)$ be the degree of $x$ in the graph $G$.
For a subset $R \subset V$, we define the interior boundary $\partial R = \{x \in R : d_G(x, V \setminus R) = 1\}$ and exterior boundary $\partial^{\mathrm{ext}} R = \{x \in V \setminus R : d_G(x, R) = 1\}$. 

\subsection{Results for nearest-neighbour interactions}
\label{section:nearest-neighbour}
We first consider the case of (ferromagnetic) nearest-neighbour interactions, which are defined as follows when $G$ has bounded degree.
\begin{definition}
    \label{def: nearest-neighbour}
    If there exists a constant $D$ such that $\mathrm{deg}(x) \leq D$ for all $x \in V$, then we define nearest-neighbour interactions $J_G$ on $G$ by
    \begin{equation*}
        (J_G)_{x, y} = \begin{cases}
            1 & \text{if } xy \in E,\\
            0 & \text{otherwise}.
        \end{cases}
    \end{equation*}
\end{definition}
To simplify the notation, we choose $f$ such that $f(1) = 1$. Then, in the $n>2$ case, 
\begin{equation*}
    A(x, R) = \max \left\{ 1, \max_{y \in \partial R} \left(\frac{|h_{y,R}|}{C|N_{y, V\setminus R}|}\right)^{(n-1)^{-d_R(x,y) -1}} \right\},
\end{equation*}
where $N_{y, V\setminus R} = \partial^{\mathrm{ext}} R \cap \partial^{\mathrm{ext}} \{y\}$ is the set of neighbours of $y$ that are outside $R$.
We also have
\begin{equation*} 
    \tilde{A}(x, R) = \max \left\{ 1, \max_{z \in R \cup \partial^{\mathrm{ext}} R} \left(\frac{|\xi_z|}{C}\right)^{(n-1)^{-d_{R \cup \{z\}}(x, z)}} \right\}.
\end{equation*}
As a consequence of this and connectedness, it follows from Lemma \ref{A comparison} that
\begin{equation}
    \label{eq: short range Xi}
    \Xi = \{\xi \in \mathbb{R}^{V} : \exists \ A_{\xi} \in (0, \infty) \text{ such that } |\xi_z| \leq A_{\xi}^{(n-1)^{d_G(o, z)}} \, \forall \, z \in V\}.
\end{equation}

In the $n=2$ case, we have
\begin{equation}
    \label{eq: formula for A n=2}
    A(x, R) = \max \left\{ 1, \max_{y \in \partial R} \left(\frac{|h_{y,R}|}{C|N_{y, V\setminus R}|\lambda^{d_R(x,y)+1}}\right) \right\},
\end{equation}
and
\begin{equation*} 
    \tilde{A}(x, R) = \max \left\{ 1, \max_{z \in R \cup \partial^{\mathrm{ext}} R} \left(\frac{|\xi_z|}{C\lambda^{d_{R \cup \{z\}}(x, z)}}\right) \right\},
\end{equation*}

so that
\begin{equation}
    \label{eq: short range Xi n=2}
    \Xi(\lambda) = \{\xi \in \mathbb{R}^{V} : \exists \ C_{\xi} \in (0, \infty) \text{ such that } |\xi_z| \leq C_{\xi}\lambda^{d_G(o, z)} \, \forall \, z \in V\}.
\end{equation}

The remainder of the subsection is devoted to justifying that our regularity results are optimal. More precisely, we aim to show that for the $P(\varphi)$ models, any non-negative boundary conditions for which we have tightness are in $\Xi$. To simplify the calculations, we only consider the case when $P(u) = \tilde{a}|u|^{n}$ here. 
Note that in this case $\rho$ satisfies \eqref{eq: single-site assumption} for any $a < \tilde{a}$.

\begin{proposition}
\label{Anti-tightness prop}
Assume $n>2, \tilde{a} > 0$, $\mathrm{d} \rho(u) = e^{-\tilde{a}|u|^{n}} \mathrm{d} u$ and $G$ has bounded degree. If $\xi \in (\mathbb{R}^{+})^{V} \setminus \Xi$, then the family of measures $(\nu_{\Lambda, \beta, \rho, J_G}^{\xi})_{\Lambda \Subset V}$ is not tight.
\end{proposition}

When $V = \mathbb{Z}$ we can obtain the same result as Proposition \ref{Anti-tightness prop} for mixed positive and negative boundary conditions.

\begin{proposition}
    \label{Anti-tightness mixed}
    Assume $n>2, \tilde{a} > 0$, $\mathrm{d} \rho(u) = e^{-\tilde{a}|u|^{n}} \mathrm{d} u$ and
    $G = (\mathbb{Z}, \{xy : |x-y|  = 1\})$. If $\xi \in \mathbb{R}^{\mathbb{Z}} \setminus \Xi$, then the family of measures $(\nu_{\Lambda, \beta, \rho, J_G}^{\xi})_{\Lambda \Subset V}$ is not tight.
\end{proposition}

In the proofs of the above propositions, we will use the following monotonicity property.

\begin{lemma}
\label{J monotonicity Lemma}
Assume $\rho$ is an even measure satisfying \eqref{eq: single-site assumption}.
Suppose $J, J'$ are interactions on $V$ satisfying  \eqref{C1}, \eqref{C2} and $0 \leq J_{x, y} \leq J'_{x,y}$ for all $x, y \in V$. Suppose also that  $\xi, \xi'$ are boundary conditions on $\Lambda$ such that $0 \leq \xi_x \leq \xi'_x$ for all $x \in V$ and $\sum_{y \in V} J'_{x, y} |\xi'_y| < \infty$ for all $x \in \Lambda$. Let $u\geq 0$ and $x \in \Lambda$.
Then 
\begin{equation*}
    \nu_{\Lambda, \beta, \rho, J}^{\xi}[\varphi_x \geq u] \leq \nu_{\Lambda, \beta, \rho, J'}^{\xi'}[\varphi_x \geq u].
\end{equation*}
\end{lemma}
\begin{proof}
    Monotonicity in $\xi$ has already been established in Proposition \ref{b.c monotonicity}, so we just need to prove that $\nu_{\Lambda, \beta, \rho, J}^{\xi}[\varphi_x \geq u] \leq \nu_{\Lambda, \beta, \rho, J'}^{\xi}[\varphi_x \geq u]$. 
    Writing $\sigma_x$ for the sign of $\varphi_x$ and using that $\mathbbm{1}_{\{\sigma_x = 1\}} = \frac{1}{2}(1 + \sigma_x)$, we have
    \begin{align*}
        \nu_{\Lambda, \beta, \rho, J}^{\xi}[\varphi_x \geq u] &= \nu_{\Lambda, \beta, \rho, J}^{\xi}[|\varphi_x| \geq u] \nu_{\Lambda, \beta, \rho, J}^{\xi}[\sigma_x = 1 \mid |\varphi_x| \geq u]\\
        &= \frac{1}{2} \nu_{\Lambda, \beta, \rho, J}^{\xi}[|\varphi_x| \geq u] (1 +\langle \sigma_x \mid |\varphi_x| \geq u \rangle_{\Lambda, \beta, \rho, J}^{\xi}).
    \end{align*}
    Conditional on the absolute value field, $\sigma$ is distributed according to an Ising model with coupling constants determined by the absolute value field. As the boundary conditions are positive, monotonicity of the Ising model in $J$ follows by differentiating and using Griffiths' inequality \cite[Theorem 3.20]{friedli_velenik_2017}, and monotonicity of the absolute value field was proved in \cite[Proposition 4.10]{well_behaved}.
\end{proof}

We are now ready to proceed with the proofs of Propositions \ref{Anti-tightness prop} and \ref{Anti-tightness mixed}.

\begin{proof}[Proof of Proposition \ref{Anti-tightness prop}]
    Using the characterisation \eqref{eq: short range Xi} of $\Xi$, $\xi \in (\mathbb{R}^{+})^{V} \setminus \Xi$ implies that there exists a sequence of vertices $(z_i)_{i \geq 1}$ such that $\xi_{z_i}^{(n-1)^{-m_i}} \rightarrow \infty$ as $i \rightarrow \infty$, where $m_i = d_G(o, z_i)$. Since there are only finitely many vertices at any fixed distance from $o$, by passing to a subsequence, we may assume that $1 \leq m_i < m_j$ for any $i < j$. Given $i \geq 1$, let $\Lambda_i = \{ x \in V : d_G(o, x) < m_i\}$ and let $y_{i,0}, y_{i,1}, \ldots, y_{i,m_i}$ be a walk from $o$ to $z_i$ in $\overline{(\Lambda_i, J)}$.
    Also let $\alpha = \frac{\beta}{\tilde{a}n2^{n-1}}$, and for $j \in \{0, \ldots, m_i\}$, define
    \begin{equation*}
        D_{i, j} = \alpha^{\frac{1-(n-1)^{j-m_i}}{n-2}} \xi_{z_i}^{(n-1)^{j - m_i}}.
    \end{equation*}
    Then $D_{i, m_i} = \xi_{z_i}$ and $D_{i, j+1} = \alpha^{-1} D_{i, j}^{n-1}$. 
    We will show that there exists $\varepsilon > 0$ such that for all $i$ sufficiently large, $
        \nu_{\Lambda_i, \beta, \rho, J_G}^{\xi} [\varphi_o \geq D_{i, 0}] \geq \varepsilon$.
    Since $D_{i, 0} \geq \min \{1, \alpha^{\frac{1}{n-2}}\} \xi_{z_i}^{(n-1)^{-m_i}} \rightarrow \infty$ as $i \rightarrow \infty$, this implies that the sequence is not tight. The strategy for the proof is to condition in turn on the events $\{ \varphi_{y_{i,j}} \geq D_{i, j} \}$. We can then use Lemma \ref{J monotonicity Lemma} to set $J_{x, y} = 0$ everywhere except for the edge between $y_{i,j}$ and $y_{i,j+1}$, meaning that we only have to calculate a one-dimensional integral in each step. We have
    \begin{align*}
        \nu_{\Lambda_i, \beta, \rho, J_G}^{\xi}[\varphi_{o} \geq D_{i,0}]
        &\geq \nu_{\Lambda_i, \beta, \rho, J_G}^{\xi} [\varphi_{y_{i,j}} \geq D_{i,j} \; \forall j \in \{0, \ldots, m_i - 1\}]\\
        &= \prod_{j = 0}^{m_i - 1} \nu_{\Lambda_i, \beta, \rho, J_G}^{\xi} [\varphi_{y_{i,j}} \geq D_{i, j} \mid \varphi_{y_{i,k}} \geq D_{i,k} \; \forall k \in \{j+1, \ldots, m_i - 1\}].
    \end{align*}
    Define $J^{(i,j)}$ by $(J^{(i,j)})_{x, y} = 1$ if $\{x, y\} = \{y_{i,j}, y_{i, j+1}\}$ and $(J^{(i,j)})_{x, y} = 0$ otherwise. Also let $\xi^{(i, j)}$ be defined by $\xi^{(i,j)}_{y_{i, j+1}} = D_{i, j+1}$ and $\xi^{(i,j)}_{x} = 0$ for all $x \in V \setminus \{y_{i, j+1}\}$.
    Using the domain Markov property and Lemma \ref{J monotonicity Lemma}, we have
    \begin{equation*}
        \nu_{\Lambda_i, \beta, \rho, J_G}^{\xi} [\varphi_{y_{i,j}} \geq D_{i, j} \mid \varphi_{y_{i,k}} \geq D_{i,k} \; \forall k \in \{j+1, \ldots, m_i - 1\}] \geq \nu_{\{y_{i, j}\}, \beta, \rho, J^{(i,j)}}^{\xi^{(i,j)}} [\varphi_{y_{i,j}} \geq D_{i, j}].
    \end{equation*}
    We now estimate the probability on the right hand side.
    Let $r = \beta D_{i, j+1} \varphi_{y_{i,j}} - \tilde{a}|\varphi_{y_{i,j}}|^{n}$. Then
    \begin{equation*}
        \frac{\mathrm{d} r}{\mathrm{d} \varphi_{y_{i,j}}} \bigg|_{\varphi_{y_{i,j}} = D_{i, j}} = \beta D_{i, j+1} - \tilde{a}n D_{i, j}^{n-1} = \beta D_{i, j+1} \left(1 - \frac{1}{2^{n-1}} \right),
    \end{equation*}
    where the second equality is from our choice of $\alpha$.
    This is greater than 1 for all $j$ if $i$ is large enough. Then because $r$ is a concave function of $\varphi_{y_{i,j}}$, $\frac{\mathrm{d} r}{\mathrm{d} \varphi_{y_{i,j}}} \geq 1$ whenever $\varphi_{y_{i,j}} \leq D_{i, j}$. Hence,
    \begin{align}
        \label{denominator integral}
        \int_{- \infty}^{D_{i, j}} e^{\beta D_{i, j+1} \varphi_{y_{i,j}} - \tilde{a}|\varphi_{y_{i,j}}|^{n}} \mathrm{d} \varphi_{y_{i,j}} \leq \int_{- \infty}^{\beta D_{i, j} D_{i, j+1} - \tilde{a} (D_{i, j})^{n}} e^{r} \mathrm{d} r
        = \exp(\beta D_{i, j} D_{i, j+1} - \tilde{a} (D_{i, j})^{n}).
    \end{align}
    Note that the maximum value of $r$ occurs when $\varphi_{y_{i,j}} = (\frac{\beta}{\tilde{a}n} D_{i, j+1})^{1/(n-1)} = 2 D_{i, j}$, and $r$ is increasing when $\varphi_{y_{i, j}} < 2 D_{i, j}$. Consequently, when $\varphi_{y_{i,j}} \in [D_{i, j}, 2 D_{i,j}]$, the value of $r$ is at least $\beta D_{i, j} D_{i, j+1} - \tilde{a} (D_{i, j})^{n}$, which implies
    \begin{align}
        \label{numerator integral}
        \int_{D_{i, j}}^{\infty} e^{\beta D_{i, j+1} \varphi_{y_{i,j}} -\tilde{a}|\varphi_{y_{i,j}}|^{n}} \mathrm{d} \varphi_{y_{i,j}} \geq D_{i, j} \exp(\beta D_{i, j} D_{i, j+1} - \tilde{a} (D_{i, j})^{n}).
    \end{align}
    Combining \eqref{denominator integral} and \eqref{numerator integral} we obtain that
    \begin{equation*}
        \nu_{\{y_{i,j}\}, \beta, \rho, J^{(i,j)}}^{\xi^{(i,j)}} [\varphi_{y_{i,j}} \geq D_{i, j}] \geq \frac{D_{i, j}}{D_{i, j} + 1}.
    \end{equation*}
    To conclude, we need to take the product over $j$ and verify that this is bounded below for all $i$ sufficiently large by a positive constant that does not depend on $i$. We have
    \begin{align*}
         \nu_{\Lambda_i, \beta, \rho, J_G}^{\xi}[\varphi_{o} \geq D_{i,0}] \geq \prod_{j = 0}^{m_i - 1} \frac{D_{i, j}}{D_{i, j} + 1}
         = \exp \left( \sum_{j = 0}^{m_i - 1} \log(D_{i, j}) - \log(D_{i, j} + 1) \right).
    \end{align*}
    Taylor expanding $\log(D_{i, j} + 1)$ around $D_{i, j}$, we get $\log(D_{i, j} + 1) \leq \log(D_{i, j}) + \frac{1}{D_{i, j}}$, so
    \begin{align*}
        \nu_{\Lambda_i, \beta, \rho, J_G}^{\xi}[\varphi_{o} \geq D_{i,0}] \geq \exp \left(- \sum_{j=0}^{m_i -1} \frac{1}{D_{i,j}} \right)
        \geq \exp \left( -\sum_{j=0}^{\infty} \frac{1}{\min\{ 1, \alpha^{\frac{1}{n-2}}\} (\xi_{z_i}^{(n-1)^{-m_i}})^{(n-1)^{j}}} \right).
    \end{align*}
    Since $\xi_{z_i}^{(n-1)^{-m_i}} \rightarrow \infty$ as $i \rightarrow \infty$,
    The last sum above converges for all $i$ sufficiently large and decreases to $0$ as $i \rightarrow \infty$, from which the desired result follows.
\end{proof}

\begin{proof}[Proof of Proposition \ref{Anti-tightness mixed}]
    If $\xi \notin \Xi$, then using \eqref{eq: short range Xi} there exists a sequence of vertices $(z_i)_{i \geq 1}$ such that $|\xi_{z_i}|^{(n-1)^{-|z_i|}} \rightarrow \infty$ as $i \rightarrow \infty$. We will proceed with the proof in the case where for infinitely many $i$, $z_i$ and $\xi_{z_i}$ are positive (the other cases are similar). By taking an appropriate subsequence, we can assume that $\xi_{z_i} \geq 0$ and $1 \leq z_i < z_j$ for all $ 1 \leq i < j$.
    
    Define $\Lambda_i = \{x \in \mathbb{Z} : |x| < z_i \}.$ 
    We first consider the case when there exists a subsequence $(z_{i_{k}})_{k \geq 1}$ such that $\xi_{-z_{i_k}} < -\xi_{z_{i_k}}$ for all $k \geq 1$, which implies that $\nu_{\Lambda_{i_k}, \beta, \rho, J_G}^{\xi}[\varphi_0 \leq 0] \geq \frac{1}{2}$. To see why this is true, note that if $\xi_{-z_{i_k}} = -\xi_{z_{i_k}}$, then $\varphi_0$ and $-\varphi_0$ have the same distribution, so the probability that $\varphi_0$ is negative is $1/2$. Now reducing $\xi_{-z_{i_k}}$ increases the probability of the event  $\{\varphi_0\leq 0\}$ by Proposition \ref{b.c monotonicity}. 
    
    Conditionally on $\{\varphi_0 \leq 0\}$, the subgraph $\{-1, -2, \ldots, -(z_{i_k} -1)\}$ has non-positive boundary conditions. Hence $-\varphi$ is distributed according to a measure with non-negative boundary conditions, and we can apply Proposition \ref{Anti-tightness prop} to deduce that there exists $\varepsilon > 0$ such that $\nu_{\Lambda_{i_k}, \beta, \rho, J_G}^{\xi}[\varphi_{-1} \leq -D_k|\varphi_0 \leq 0] \geq \varepsilon$
    for all $k$ sufficiently large, where $D_k \rightarrow \infty$ as $k \rightarrow \infty$. It follows that
    $ \nu_{\Lambda_{i_k}, \beta, \rho, J_G}^{\xi}[|\varphi_{-1}| \geq D_k] \geq \frac{\varepsilon}{2}$,
    so the sequence of measures $(\nu_{\Lambda_{i_k}, \beta, \rho, J_G}^{\xi})_{k \geq 1}$ is not tight.
    
    If no such subsequence exists, then for all $i$ large enough we have that $\xi_{-z_i} \geq - \xi_{z_i}$, so $\nu_{\Lambda_i, \beta, \rho, J_G}^{\xi}[\varphi_0 \geq 0] \geq \frac{1}{2}.$ We can now conclude the proof similarly using that if $\varphi_0 \geq 0$ then we have non-negative boundary conditions on the subgraph $\{1, 2, \ldots, z_{i_k} -1\}$.
\end{proof}

\subsection{Results for more general interactions}
\label{Section: long-range}

In this subsection, we determine which boundary conditions are in $\Xi$ in the case of long-range interactions that satisfy some additional assumptions.

\begin{definition}
    \label{Def:reasonable interactions}
    We say that interactions $(J_{x, y})_{x, y \in V}$ are reasonable if they satisfy the following assumptions in addition to \eqref{C1} and \eqref{C2}:
    \begin{itemize}
        \item
        $V$ is $J$-connected.
        \item 
        There exists $r> 0$ such that for all $x\neq y \in V, \, |J_{x, y}| \leq  d_G(x, y)^{-r}$.
        \item
        $f$ is an even function and is decreasing on $(0, \infty)$.
        \item
        There exists $c \in (0, 1)$ such that $f(2^{r} t) \geq c f(t)$ for all $t \in [0, \infty)$.
    \end{itemize}
\end{definition}

Note that this includes the nearest-neighbour interactions $J_G$ as we are assuming $G$ is connected.
The next proposition gives a sufficient condition to have $\xi \in \Xi$ when $J$ is reasonable. Below $c$ and $r$ are the constants of Definition \ref{Def:reasonable interactions}.
\begin{proposition}
    \label{long range bc example}
    Suppose that $J$ is reasonable and $\xi \in \mathbb{R}^{V}$ is such that there exists $M_{\xi} \in \mathbb{R}$ with $|\xi_x| \leq M_{\xi} f(d_G(o, x)^{-r})$ for all $x \in V \setminus \{o\}$. Then $\xi \in \Xi(\lambda)$ for any $\lambda \geq \frac{1}{c}$.
\end{proposition}

Before proving the proposition, we give two examples where it can be applied. Firstly, the mildest function satisfying the assumption \eqref{C2} is the function $f_1$ given by 
\begin{equation*}
    f_1(t) = \begin{cases}
        \log(|t|^{-1})^{1/2} & \text{if } |t| < e^{-1},\\
        1 & \text{otherwise}.
    \end{cases}
\end{equation*}
Proposition \ref{long range bc example} implies that if $|\xi_x|$ grows at most like $\sqrt{\log(d_G(o, x))}$, then $\xi \in \Xi(\lambda)$ for any reasonable $J$ and $\lambda$ large enough.

For the second example, we consider $\mathbb{Z}^{d}$, or any vertex-transitive graph of dimension $d$. Let $J$ be translation invariant interactions satisfying $|J_{x,y}| \leq C_J d_{G}(x, y)^{-d-\varepsilon}$ for all $x \neq y$, where $C_J, \varepsilon > 0$ are constants. This is the setting in which the regularity results of \cite{Lebowitz_Presutti} and \cite{random_tangled} were proved. 
In this case (after making $C_J =1$ by changing the value of $\beta$), for any $\alpha < \frac{\varepsilon}{d+\varepsilon}$ we can apply Proposition \ref{long range bc example} with $r = d + \varepsilon$ to the function $f_\alpha$ given by
\begin{equation*}
    f_\alpha(t) = \begin{cases}
        |t|^{-\alpha} & \text{if } |t| < 1,\\
        1 & \text{otherwise}.
    \end{cases}
\end{equation*}
This implies that any $\xi$ with $|\xi_x| \leq M_{\xi} d_G(x, y)^{\alpha(d+\varepsilon)}$ is in $\Xi(\lambda)$ for all $\lambda \geq 2^\varepsilon$, so we have tightness for boundary conditions growing at most like $d_{G}(o, x)^{\delta}$ for $\delta < \varepsilon$.

We now give the proof of Proposition \ref{long range bc example}.
\begin{proof}[Proof of Proposition \ref{long range bc example}]
    Since $V$ is $J$-connected, we just need to show that $\tilde{A}(o, V, \lambda)$ is finite, as Lemma \ref{A comparison} then implies $\tilde{A}(x, V, \lambda) < \infty$ for all $x \in V$.
    First observe that if $\lambda \geq \frac{1}{c},$ then the assumptions on $f$ imply that for any walk in $\overline{(V, J)}$ consisting of distinct vertices $x_0, x_1, \ldots, x_m$ with $m \geq 2$,
    \begin{align}
        \label{eq:walk shortening}
        \lambda f(d_G(x_{m-2}, x_{m-1})^{-r}) f(d_G(x_{m-1}, x_{m})^{-r})
        &\geq \lambda f\left(\left(\frac{1}{2} d_{G}(x_{m-2}, x_{m})\right)^{-r}\right)\\
        \nonumber
        &\geq f(d_G(x_{m-2}, x_{m})^{-r}).
    \end{align}
    Repeatedly applying \eqref{eq:walk shortening} and using that $f(J_{x_{i-1}, x_i}) \geq f(d_G(x_{i-1}, x_i)^{-r})$ by Definition \ref{Def:reasonable interactions}, we deduce that
    \begin{align}
        \label{one-step walk}
        \lambda^m \prod_{i=1}^{m} f(J_{x_{i-1}, x_i})
        \geq  \lambda f(d_G(x_{0}, x_{m})^{-r}).
    \end{align}
    The above inequality is in fact valid for any walk $x_0, x_1, \ldots, x_m$ with $m \geq 1$ because repeating a vertex in the walk makes the left hand side of \eqref{one-step walk} larger.
    Now consider $z \in V \setminus \{o\}$ with a walk $x_0, x_1, \ldots, x_m$ from $o$ to $z$ in $\overline{(V, J)}$. 
    Applying \eqref{one-step walk} to reduce $x_0, x_1, \ldots , x_{m}$ to a one-step walk $o, z$ yields
    \begin{align*}
        |\xi_z| \leq M_{\xi} f(d_G(o, z)^{-r})
        \leq M_\xi \lambda^{m} \prod_{i=1}^{m} f(J_{x_{i-1}, x_i}),
    \end{align*}
    so if $\mathcal{A} \geq \frac{M_{\xi}}{C}$, then $\mathcal{A}$ satisfies the requirements for $\tilde{A}(o, V, \lambda)$ for any $z \in V \setminus \{o\}$. The case $z = o$ can also be included by increasing the value of $\mathcal{A}$ further if necessary.
\end{proof}

\subsection{Results for infinite-volume measures}
\label{Section: infinite-volume}
In this section, we show how our results can be applied to measures defined on $\mathbb{R}^{V}$.
Recall that $G = (V, E)$ is an infinite connected graph where every vertex has finite degree and with a fixed origin $o \in V$.
For $y \in V$, we let $B_k(y) = \{x \in V : d_{G}(x, y) \leq k\}$.
Throughout this section, we assume that $J, f, \beta$ and $\rho$ are fixed with $\rho$ satisfying \eqref{eq: single-site n=2} and that $V$ is $J$-connected. We will also drop $J$ from the subscripts.
If $\rho$ satisfies the stronger assumption \eqref{eq: single-site assumption} for some $a >0$, $n>2$, then slightly stronger versions of some statements in this section can be obtained by using the machinery of Theorem \ref{main theorem} instead of Theorem \ref{Theorem: n=2}.

Recall from Definition \ref{Def:Gibbs} the definitions of $a$-regular measures and Gibbs measures.
As a corollary of Theorem \ref{Theorem: n=2}, we obtain regularity when $\nu$ is a limit of finite-volume measures with boundary conditions growing slowly enough that $A(x, \Lambda)$ is bounded by a constant for all $x$ sufficiently far from the boundary of $\Lambda$.

\begin{corollary}
    \label{prop: limit of finite n=2}
    Let $(\Lambda_i)_{i \geq 1}$ be a sequence of finite subsets of $V$ such that $\Lambda_i \nearrow V$ as $i \rightarrow \infty$ and assume that the boundary conditions $\xi$ satisfy for some $\lambda \geq 1$
    \begin{equation}
        \label{eq:tempered boundary conditions n=2}
        \exists A_{\max} \in [1, \infty) \text{ such that } \limsup_{\Lambda \nearrow V} A(x, \Lambda, \lambda, \xi) \leq A_{\mathrm{max}} \, \forall x \in V.
    \end{equation}
    If $\nu_{\Lambda_i, \beta, \rho}^{\xi}$ converges weakly to a probability measure $\nu$ as $i \rightarrow \infty$, then $\nu$ is an $a-$regular Gibbs measure for any $a \geq 2 \beta M_f \lambda$. 
\end{corollary}

\begin{proof}
    Fix $\Lambda' \Subset V$ and $a \geq 4 \beta M_f \lambda$. By Theorem \ref{Theorem: n=2}, we have for any $i$ large enough that $\Lambda' \subset \Lambda_i$,
    \begin{equation*}
        \mathrm{d} \nu_{\Lambda_i, \beta, \rho}^{\xi}[\varphi|_{\Lambda'} = \psi]
        \leq  \left(\prod_{x\in \Lambda'}\exp(\tilde{C} A(x, \Lambda_i, \lambda, \xi)^2)\right) \mathrm{d} \nu_{\Lambda', 0, \rho_{\frac{a}{2}}}^{0} [\psi].
    \end{equation*}
    Taking $i \rightarrow \infty$ and using \eqref{eq:tempered boundary conditions n=2}, we have that $\nu$ is $\frac{a}{2}-$regular with $B = \tilde{C} A^2_{\max}$. The fact that $\nu$ is a Gibbs measure follows from the domain Markov property for the measures $\nu_{\Lambda_i, \beta, \rho}^{\xi}$.
\end{proof}

Let us consider some examples where \eqref{eq:tempered boundary conditions n=2} is satisfied.
Let $x \in V$ and assume that $\Lambda$ is large enough that $d_G(x, z) \geq \frac{1}{2} d_G(o, z)$ for any $z \in V \setminus \Lambda$.
For nearest-neighbour interactions, suppose $\xi \in \Xi(\lambda)$. Then by \eqref{eq: short range Xi n=2}, there exists $C_\xi \in (0, \infty)$ such that for any $z \in V \setminus \Lambda$
\begin{equation*}
    |\xi_z| \leq C_\xi \lambda^{d_G(o, z)} \leq C_\xi \lambda^{2d_G(x, z)},
\end{equation*}
which implies that for any $y \in \partial \Lambda$,
\begin{equation*}
    \frac{|h_{y,\Lambda}|}{C|N_{y, V\setminus \Lambda}|\lambda^{2(d_\Lambda(x,y)+1)}}
    \leq \max_{z \in N_{y, V\setminus \Lambda}} \left\{\frac{|\xi_z|}{C \lambda^{2 d_G(x, z)}}\right\} \leq \frac{C_\xi}{C}.
\end{equation*}
It follows from \eqref{eq: formula for A n=2} that $A(x, \Lambda, \lambda^2, \xi) \leq \max\{1, C_\xi/C\}$, and since this is true for any $\Lambda$ large enough, \eqref{eq:tempered boundary conditions n=2} holds with $\lambda^2$ in place of $\lambda$.

For reasonable interactions, \eqref{eq:tempered boundary conditions n=2} holds with $\lambda \geq \frac{1}{c}$ for any boundary conditions satisfying the assumptions of Proposition \ref{long range bc example}.
Indeed, the choice of $\Lambda$ and the assumptions on $f$ in Definition \ref{Def:reasonable interactions} imply that for any $z \in V \setminus \Lambda$,
\begin{align*}
    f(d_G(x,z)^{-r}) \geq f(2^{r}d_G(o,z)^{-r}) \geq c f(d_G(o,z)^{-r}).
\end{align*}
Hence by the assumptions on $\xi$, we have for any walk $x_0, \ldots, x_m$ from $x$ to $z$
\begin{align}
    \label{eq: reasonable A_max}
    |\xi_z| \leq M_{\xi} f(d_G(o, z)^{-r}) \leq \frac{M_{\xi}}{c}f(d_G(x, z)^{-r}) \leq M_{\xi} \lambda^{m} \prod_{i=1}^{m} f(J_{x_{i-1}, x_i}),
\end{align}
where we have used \eqref{one-step walk} and the fact that $f(J_{x_{i-1}, x_i}) \geq f(d_G(x_{i-1}, x_i)^{-r})$ in the last inequality.
From \eqref{eq: reasonable A_max} we see that $A(x, \Lambda, \lambda, \xi) \leq \frac{M_{\xi}}{C}$.

We now show how we can use regularity to make sense of ``maximal'' boundary conditions, which will allow us to construct the infinite-volume plus measure.
Define $\xi^{+} \in \mathbb{R}^{V}$ by $\xi^{+}_x = \sqrt{\log(|B_{d_G(o, x)}(o)|)}$ and let $\xi^{-} = - \xi^{+}$. Below $r$ and $c$ are as in Definition \ref{Def:reasonable interactions}.

\begin{proposition}
    \label{prop: maximal b.c}
    Assume the interactions $J$ are reasonable and ferromagnetic.
    Suppose also that there exists a constant $c_0 > 0$ such that for any $j,k \geq 1$,
    \begin{equation}\label{eq:f assumption}        
    f(k^{-r}) \geq c_0 \sqrt{\frac{\log(|B_{k+j}(o)|)}{\log(|B_j(o)|)}}.
    \end{equation}
    Then there exist $\frac{2 \beta M_f}{c}-$regular Gibbs measures $\nu^{+}_{\beta, \rho}$, $\nu^{-}_{\beta, \rho}$ such that
    \begin{equation*}
        \lim_{\Lambda \nearrow V}\nu_{\Lambda, \beta, \rho}^{\xi^{+}} =\nu^{+}_{\beta, \rho}, \qquad 
        \lim_{\Lambda \nearrow V}\nu_{\Lambda, \beta, \rho}^{\xi^{-}}=\nu^{-}_{\beta, \rho}.
    \end{equation*}
    Moreover, for any $a>0$, any $a-$regular Gibbs measure $\nu$ satisfies $\nu^{-}_{\beta, \rho} \preceq \nu \preceq \nu^{+}_{\beta, \rho}$.
\end{proposition}

Proposition \ref{prop: maximal b.c} includes the case of nearest-neighbour interactions, as $f$ can be chosen arbitrarily in this case. Also note that if $G$ is vertex-transitive then $|B_{k+j}(o)|\leq |B_k(o)||B_j(o)|$ and the condition on $f$ simplifies to $f(k^{-r}) \geq c_0 \sqrt{\log(|B_k(o)|)}$.
For the $\varphi^4$ model, as a corollary of the Lee--Yang theorem, $\nu^{+}_{\beta, \rho}$ coincides with the measure defined with an external field $h$ by first taking $\Lambda \nearrow V$ and then taking the limit as $h \searrow 0$ (see \cite[Prop. 2.6]{random_tangled}). See also \cite{Lebowitz_Presutti, Bellissard1982} for constructions of the plus measure when $V = \mathbb{Z}^{d}$.

The main ingredient in the proof of Proposition \ref{prop: maximal b.c} is the following lemma, which allows us to obtain monotonicity in $\Lambda$ for the measures $\nu_{\Lambda, \beta, \rho}^{\xi^{+}}$ up to an error term which tends to 0 as $\Lambda \nearrow V$.
\begin{lemma}
    \label{maximal b.c lemma}
    Consider the event $F_{\Lambda, \Lambda'} = \{|\varphi_x| \leq \xi^{+}_x \ \forall x \in \Lambda \setminus \Lambda'\}$.
    If $J$ and $f$ satisfy the assumptions of Proposition \ref{prop: maximal b.c}, then $\nu_{\Lambda, \beta, \rho}^{\xi^{+}}[F_{\Lambda, \Lambda'}^{\mathsf{c}}] \rightarrow 0$ uniformly in $\Lambda \supset \Lambda'$ as $\Lambda' \nearrow V$.
\end{lemma}

\begin{proof}
    We will drop $\beta, \rho$ from the notation and just write $\nu_{\Lambda}^{\xi}$ for the finite-volume measure on $\Lambda$ with boundary conditions $\xi$.
    Let $\lambda = \frac{1}{c}$, where $c$ is as in Definition \ref{Def:reasonable interactions}.
    Applying Theorem \ref{Theorem: n=2} with $a = 4 \beta M_f \lambda$, together with a union bound, yields
    \begin{equation*}
        \nu_{\Lambda}^{\xi^{+}}[F_{\Lambda, \Lambda'}^{\mathsf{c}}] \leq \sum_{x \in \Lambda \setminus \Lambda'} \nu_{\Lambda}^{\xi^{+}}[|\varphi_x| > \xi^{+}_x] \leq
        \frac{1}{\rho_{\frac{a}{2}}(\mathbb{R})} \sum_{x \in \Lambda \setminus \Lambda'} \exp(\tilde{C} A(x, \Lambda, \lambda)^2 ) \rho_{\frac{a}{2}}[|\varphi_x| > \xi^{+}_x].
    \end{equation*}
    Let $a' > 0$ be a constant to be determined. Applying Markov's inequality to the random variable $e^{a' \varphi_x^{2}}$, we have
    \begin{equation}
        \label{eq:Markov}
        \nu_{\Lambda}^{\xi^{+}}[F_{\Lambda, \Lambda'}^{\mathsf{c}}] \leq
        \frac{\rho_{\frac{a}{2}}[e^{a' \varphi_x^2}]}{\rho_{\frac{a}{2}}(\mathbb{R})} \sum_{x \in \Lambda \setminus \Lambda'} \exp(\tilde{C} A(x, \Lambda, \lambda)^2 ) |B_{d_{G}(o, x)}(o)|^{-a'}.
    \end{equation}
    We will show that there exists a constant $C' \geq 1$ such that for any finite $\Lambda \subset V$ containing $o$ and any $x \in \Lambda \setminus \{o\}$, 
    \begin{equation}
        \label{eq:bound on A}
        A(x, \Lambda, \lambda) \leq \sqrt{C' \log(|B_{d_G(o, x)}(o)|)}.
    \end{equation}
    Combining \eqref{eq:Markov} and \eqref{eq:bound on A} gives for any $\Lambda'$ containing $o$
    \begin{align*}
        \nu_{\Lambda}^{\xi^{+}}[F_{\Lambda, \Lambda'}^{\mathsf{c}}] \leq
        \frac{\rho_{\frac{a}{2}}[e^{a' \varphi_x^2}]}{\rho_{\frac{a}{2}}(\mathbb{R})} \sum_{x \in \Lambda \setminus \Lambda'} |B_{d_{G}(o, x)}(o)|^{\tilde{C} C' -a'}\leq \frac{\rho_{\frac{a}{2}}[e^{a' \varphi_x^2}]}{\rho_{\frac{a}{2}}(\mathbb{R})}
        \sum_{i = g(\Lambda')}^{\infty} |B_{i}(o)|^{1+ \tilde{C} C' - a'},
    \end{align*}
     where $g(\Lambda') = \min_{x \in V \setminus \Lambda'} d_G(o, x)$.
    By choosing $a'$ appropriately, the last sum converges and decreases to $0$ as $\Lambda' \nearrow V$. 

    It remains to prove \eqref{eq:bound on A}. 
    Let $C' \geq 1$ be a constant to be determined and set
    \[
    \mathcal{A}_x =  \sqrt{C' \log(|B_{d_G(o, x)}(o)|)}.
    \]
    We aim to show that $C'$ can be chosen so that for any $\Lambda \Subset V$, $x \in \Lambda \setminus \{o\}$ and any walk $x_0, \ldots, x_m$ from $x$ to a vertex $z \in V \setminus \Lambda$,
    \begin{align}
        \label{eq:walk requirement n=2}
        C\mathcal{A}_{x} \lambda^{m} \prod_{i = 1}^{m} f(J_{x_{i-1}, x_i}) \geq |\xi^{+}_z|,
    \end{align}
    which implies that $A(x, \Lambda, \lambda) \leq \mathcal{A}_x$.
    As in the proof of Proposition \ref{long range bc example}, we can reduce the walk $x_0, \ldots, x_m$ to a one-step walk $x_0, x_m$ by using that $f(J_{x_{i-1}, x_{i}}) \geq f(d_G(x_{i-1}, x_{i})^{-r})$ and applying the inequality \eqref{one-step walk}. This gives
    \begin{align}\label{eq:walk reduction}
        C\mathcal{A}_{x} \lambda^{m} \prod_{i = 1}^{m} f(J_{x_{i-1}, x_i})
        \geq \frac{C\sqrt{C'}}{c}  \sqrt{\log(|B_{d_G(o, x)}(o)|)} f(d_G(x, z)^{-r}),
    \end{align}
    and \eqref{eq:f assumption} implies that
    \begin{align}
        \label{eq:f ball assumption}
        \frac{1}{c_0} \sqrt{\log(|B_{d_G(o, x)}(o)|)} f(d_G(x, z)^{-r}) \geq 
        \sqrt{\log(|B_{d_{G}(o, x) + d_G(x, z)}(o)|)} \geq |\xi^{+}_z|.
    \end{align}
    Combining \eqref{eq:walk reduction} and \eqref{eq:f ball assumption}, we see that it is possible to choose $C'$ so that \eqref{eq:walk requirement n=2} is satisfied, completing the proof of \eqref{eq:bound on A} and of the lemma.
\end{proof}

We now proceed with the proof of Proposition \ref{prop: maximal b.c}. 

\begin{proof}[Proof of Proposition \ref{prop: maximal b.c}]
    We only prove the statements for $\nu^{+}_{\beta, \rho}$.
    As we have fixed $\beta, \rho$ we will drop them from the notation and just write $\nu_{\Lambda}^{\xi}$ for the finite-volume measure on $\Lambda$ with boundary conditions $\xi$.
    Let $H$ be an increasing event that depends only on spins inside a finite subset $\Lambda' \subset V$
    and let $(\Lambda_i)_{i \geq 1}$ be a sequence of finite subsets of $V$ with $\Lambda_i \nearrow V$ as $i \rightarrow \infty$.
    For $k \geq i$, write $F_{k,i}$ for the event $F_{\Lambda_{k}, \Lambda_i}$ defined in Lemma \ref{maximal b.c lemma}.
    Using the domain Markov property and the fact that if $\varphi \in F_{k,i}$, then Proposition \ref{b.c monotonicity} implies that $\nu_{\Lambda_i}^{\varphi}[H] \leq \nu_{\Lambda_i}^{\xi^{+}}[H]$, we have for any $k \geq i$ such that $\Lambda' \subset\Lambda_i$,
   $\nu_{\Lambda_{k}}^{\xi^{+}}[H] \leq \nu_{\Lambda_i}^{\xi^{+}}[H] + \nu_{\Lambda_{k}}^{\xi^{+}}[F_{k,i}^{\mathsf{c}}]$.
    Sending first $k$ to infinity, and then $i$ to infinity, and using Lemma \ref{maximal b.c lemma}, we get that $\limsup_{k\to\infty}\nu_{\Lambda_{k}}^{\xi^{+}}[H]\leq \liminf_{i\to\infty}\nu_{\Lambda_{i}}^{\xi^{+}}[H]$.
    Hence $\lim_{\Lambda \nearrow V} \nu_{\Lambda}^{\xi+}[H]$ exists for any increasing event $H$ depending only on finitely many spins. As these events generate the $\sigma-$algebra, we obtain convergence of $\nu_{\Lambda}^{\xi+}$ to a measure $\nu^{+}$ as $\Lambda \nearrow V$.
    
    The assumption \eqref{eq:f assumption} implies that the boundary conditions $\xi^{+}$ satisfy the assumptions of Proposition \ref{long range bc example}.
    Hence, \eqref{eq:tempered boundary conditions n=2} is satisfied with $\lambda = \frac{1}{c}$, and we can apply 
    Corollary \ref{prop: limit of finite n=2} to deduce that $\nu^{+}$ is an $a$-regular Gibbs measure for any $a \geq 2 \beta M_f \lambda$.
    
    Now consider any $a>0$ and suppose $\nu$ is an $a-$regular Gibbs measure. Then for any finite $\Lambda \subset V$, the DLR equation gives that
    \begin{align}
        \label{eq: xi + DLR}
        \nu[H] = \nu[F_{\Lambda}] \int_{\varphi \in \mathbb{R}^{V}} \nu_{\Lambda}^{\varphi}[H] \mathrm{d} \nu(\varphi | F_{\Lambda})
        + \nu[F_{\Lambda}^{\mathsf{c}}] \int_{\varphi \in \mathbb{R}^{V}} \nu_{\Lambda}^{\varphi}[H] \mathrm{d} \nu(\varphi | F_{\Lambda}^{\mathsf{c}}),
    \end{align}
    where $F_{\Lambda} = \{\varphi:\varphi_x \leq \xi^{+}_x \, \forall x \in V\setminus \Lambda\}$.
    We can show that $\nu[F_{\Lambda}^{\mathsf{c}}] \rightarrow 0$ as $\Lambda \nearrow V$ in the same way as in the proof of Lemma \ref{maximal b.c lemma}:
    first using a union bound, regularity, and Markov's inequality, we have for any $a' > 0$
    \begin{equation*}
        \nu[F_{\Lambda}^{\mathsf{c}}] \leq \sum_{x \in V \setminus \Lambda} \nu[|\varphi_x| \geq \xi^{+}_x]
        \leq
        \frac{B \rho_{a}[e^{a' \varphi_x^2}]}{\rho_{a}(\mathbb{R})} \sum_{x \in V \setminus \Lambda} |B_{d_{G}(o, x)}(o)|^{-a'}.
    \end{equation*}
    It then follows that $\nu[F_{\Lambda}^{\mathsf{c}}] \rightarrow 0$ as $\Lambda \nearrow V$, provided that $a'$ is chosen large enough.
    Hence taking $\Lambda \nearrow V$ in \eqref{eq: xi + DLR} and using Proposition \ref{b.c monotonicity}, we have $\nu[H] \leq \nu^{+}[H]$.
\end{proof}

\subsection{Weak domination by a product measure}
\label{Section: weak domination}

Theorem \ref{Theorem: n=2} states that the density of the spin models we consider in this paper can be bounded up to some multiplicative constants by the density of a product measure with all Gaussian moments. This leads naturally to the question of whether the spin model itself is stochastically dominated by a product measure with all Gaussian moments. As we will see, regularity leads to a weaker form of domination by a product measure.

Given $a>0$ and $B_x>0$, we define $\zeta_x=\zeta_{x,a,B_x}$ to be a single-site measure that depends on the vertex $x$ (see Remark \ref{remark:changed single-site}), defined by
    $\mathrm{d} \zeta_x(u) = \mathbbm{1}_{\{u \geq B_x\}} \mathrm{d} \rho_{a}(u - B_x)$.
Below, we let
    \[
    B_x = \left( \frac{2}{a} \left(  \tilde{C} A(x, \Lambda, \xi, C)^2 + \log(\rho_{a}(\mathbb{R})) - \log(\rho_{\frac{a}{2}}(\mathbb{R}))\right) \right)^{\frac{1}{2}},
    \]
    with $C, \tilde{C}$ the constants from Theorem \ref{Theorem: n=2}.

For two measures $\mu$ and $\nu$ on $\mathbb{R}^{\Lambda}$, we say that $\mu$ weakly dominates $\nu$ if for every $\Lambda'\subset \Lambda$ and for all $u \in \mathbb{R}^{\Lambda'}$, 
\[
\mu\Big(\varphi_x\geq u_x \, \forall x\in \Lambda' \Big) \geq \nu\Big( \varphi_x\geq u_x \, \forall x\in \Lambda' \,\Big).
\]
In particular, this definition implies stochastic domination for the restrictions of $\mu$ and $\nu$ to singletons, but not necessarily for the whole field.

The following corollary of Theorem \ref{Theorem: n=2} applies for any interactions $J$ satisfying \eqref{C1} and \eqref{C2}.

\begin{corollary}
    \label{domination corollary}
    Let $a \geq 4 \beta M_f$, $\lambda \leq \frac{a}{4 \beta M_f}$, and assume that $\rho([0, \infty)) > 0$.
    For any finite $\Lambda \subset V$ and $\Lambda' \subset \Lambda$, and any boundary conditions $\xi$ such that $\sum_{y \in V} |J_{x, y} \xi_y| < \infty$ for all $x \in \Lambda$,
    $\nu_{\Lambda', 0, \zeta}^{0}$ weakly dominates $\nu_{(\Lambda|\Lambda'), \beta, \rho}^{\xi}$.
\end{corollary}

\begin{proof}
    We aim to show that for any $u \in \mathbb{R}^{\Lambda'}$,
    \begin{equation*}
        \nu_{\Lambda, \beta, \rho}^{\xi} [\varphi_x \geq u_x, \forall x \in \Lambda'] \leq \nu_{\Lambda', 0, \rho_a}^{0} [\varphi_x + B_x \geq u_x, \forall x \in \Lambda' \mid \varphi_x \geq 0, \forall x \in \Lambda'].
    \end{equation*}
    The probability on the right hand side above is equal to $\nu_{\Lambda, 0, \zeta}^{0} [\varphi_x \geq u_x, \forall x \in \Lambda']$, so this implies the desired weak domination.
    Applying Theorem \ref{Theorem: n=2} with $\Lambda_{u} \coloneq \{x \in \Lambda' : u_x \geq B_x\}$ in place of $\Lambda'$, we obtain
    \begin{align*}
        \nu_{\Lambda, \beta, \rho}^{\xi} [\varphi_x \geq u_x, \forall x \in \Lambda'] &\leq \nu_{\Lambda, \beta, \rho}^{\xi} [\varphi_x \geq u_x, \forall x \in \Lambda_{u}]\\
        &\leq \left(\prod_{x \in \Lambda_u}\exp(\tilde{C} A(x, \Lambda)^2) \right) \nu_{\Lambda_u, 0, \rho_{\frac{a}{2}}}^{0}[\varphi_x \geq u_x,  \forall x \in \Lambda_u]\\
         &\leq \prod_{x \in \Lambda_u} \frac{1}{\rho_{\frac{a}{2}}(\mathbb{R})} \int_{u_x}^{\infty} \exp \left( \tilde{C} A(x, \Lambda)^2 -\frac{a}{2}|\varphi_x|^{2} \right) \mathrm{d} \rho_{a} (\varphi_x).
    \end{align*}
    The choice of $B_x$ gives that $\exp\left(\tilde{C} A(x, \Lambda)^2 -\frac{a}{2}|\varphi_x|^{2} \right) \leq \frac{\rho_{\frac{a}{2}}(\mathbb{R})}{\rho_{a}(\mathbb{R})}$ whenever $\varphi_x \geq B_x$, so since $u_x \geq B_x$ for $x \in \Lambda_u$,
    \begin{align*}
        \nu_{\Lambda, \beta, \rho}^{\xi} [\varphi_x \geq u_x, \forall x \in \Lambda'] &\leq \prod_{x \in \Lambda_{u}} \frac{1}{\rho_{a}(\mathbb{R})} \int_{u_x}^{\infty} \mathrm{d} \rho_{a} (\varphi_x)\\
        &\leq \nu_{\Lambda', 0, \rho_a}^{0}[\varphi_x + B_x \geq u_x, \forall x \in \Lambda' \mid \varphi_x \geq 0, \forall x \in \Lambda'].
    \end{align*} 
\end{proof}

\section{Stochastic domination by product measure}\label{sec: stochastic domination}

As we showed in Corollary~\ref{domination corollary}, regularity allows us to deduce that the spin measures we consider in this paper are in a weak sense dominated by product measures. In this section, we upgrade this result to full stochastic domination by a product measure. We then use this result to give alternative definitions of the plus measure that satisfy certain desirable properties.

In this section, we assume $\rho$ is an even probability measure on $\mathbb R$ satisfying \eqref{eq: single-site n=2} and we consider nearest-neighbour interactions on a graph $G=(V,E)$ of degree bounded by some $D>0$. In order to treat boundary conditions, given $\Lambda\Subset V$ and $\xi \in \mathbb{R}^{V}$, let $\kappa=\kappa(\Lambda,\xi) \in \mathbb{R}^\Lambda$ be given by
\[
\kappa_x=\sum_{y\in \partial^{\mathrm{ext}} \Lambda} \frac{|\xi_y|}{D^{2 d_G(x,y)}}.
\]

\begin{theorem}
\label{thm:product domination}
Let $G=(V,E)$ be a graph of maximal degree $D<\infty$ and consider nearest-neighbour interactions on $G$. Fix $\beta_0>0$ and let $\rho$ be an even
probability measure on $\mathbb R$ satisfying \eqref{eq: single-site n=2}.
There exists a probability measure
$\mu_*$ on $\mathbb R^+$ satisfying
\[
\int e^{\lambda t^2}\,\mathrm{d}\mu_*(t)<\infty
    \qquad\forall \lambda>0,
\]
such that, for every $\beta\in [0,\beta_0]$, for every finite
$\Lambda\subset V$ and every $\xi \in \mathbb{R}^{V}$,
\[
\nu^{\xi}_{\Lambda,\beta, \rho}(|\cdot|)
    \preceq 
    \kappa+\mu_{*}^{\otimes\Lambda}.
\]
\end{theorem}

We will split the proof into several lemmas. In order to state our first lemma, let $\rho_+$ denote the pushforward of $\rho$ under $t\mapsto |t|$.
Given $a>0$, let $\mu_a$ be the probability measure on $\mathbb R^+$ defined for $t\geq 0$ by
\[
\mu_a([t,\infty))
\coloneq\min\left\{1,\int_t^{\infty} e^{a s^2}\,\mathrm{d}\rho_+(s)\right\}.
\]
This indeed defines a probability measure since the right-hand side is non-increasing and left-continuous. 

To simplify the definitions of the measures we will consider in the proof of Theorem~\ref{thm:product domination}, we will work with another single-site measure instead of $\rho_+$ that has ``nicer'' tails. Note that there exists a continuous non-decreasing function $L_0:[0,\infty)\to [0,\infty)$ such that $L_0(t)\to \infty$ as $t\to\infty$ and for every $t\geq 0$ we have
\begin{equation*}
\rho_+([t,\infty))\leq e^{-t^2L_0(t)}.
\end{equation*}
Indeed, by Markov's inequality for every $n\geq 1$ there exists $M_n>0$ such that for every $t\geq 0$ we have that 
\[
\rho_+([t,\infty))\leq M_n e^{-nt^2},
\]
which implies the existence of the function $L_0$. Next, let us define the measure $\pi=\pi(L_0)$ by its tails: for every $t\geq 0$ 
\begin{equation*}
\pi([t,\infty))=e^{-t^2 L_0(t)}.    
\end{equation*}
This indeed defines a probability measure, since $e^{-t^2L_0(t)}$ is non-increasing, continuous and tends to $0$ at $t\to\infty$. Note that $\mu_a$ is stochastically dominated by $\pi_a$, where the measure $\pi_a$ is defined by the formula
\begin{equation*}
\pi_a([t,\infty))=e^{-t^2 L_1(t)},    
\end{equation*}
for a continuous non-decreasing function $L_1:[0,\infty)\to [0,\infty)$ such that \begin{equation}\label{eq: L_1 def}
L_1(t)=L_0(t)/2
\end{equation}
for every $t$ large enough.

\begin{lemma}
\label{lem: laplacian domination}
Fix $\beta>0$. For every $\varepsilon>0$ and for every $\xi\in \mathbb{R}^{V}$ and every $x\in \Lambda$, 
\[\nu_{\Lambda,\beta,\rho}^{\xi}(|\cdot| \mid \varphi_y=\eta_y \, \forall \, y\in \Lambda\setminus \{x\}) \preceq  \varepsilon H_x+\pi_a,
\]
where 
\[
H_x\coloneq\sum_{\substack{y\in\Lambda\\y\sim x}}|\eta_y| + \sum_{\substack{y \in \partial^{\mathrm{ext}} \Lambda \\ y \sim x}} |\xi_y|,
\]
and $a=\beta/\varepsilon$.
\end{lemma}

\begin{proof}
Fix $\varepsilon>0$.
Also fix $x \in \Lambda$ and condition on $\varphi_y=\eta_y$ for $y\neq x$. Let $|\xi \cup \eta| \in \mathbb{R}^{V}$ be the boundary conditions given by $|\xi \cup \eta|_y = |\eta_y|$ if $y \in \Lambda \setminus \{x\}$, and $|\xi \cup \eta|_y = |\xi_y|$ otherwise. By the domain Markov property and Proposition \ref{b.c monotonicity},
\[
\nu_{\Lambda,\beta,\rho}^{\xi}(|\cdot| \mid \varphi_y=\eta_y \, \forall \, y\in \Lambda\setminus \{x\}) \preceq \nu_{\{x\}, \beta, \rho}^{|\xi \cup \eta|}(|\cdot|).
\]
Note that for every $t\geq 0$
\[
\mathrm{d}\nu_{\{x\},\beta,\rho}^{|\xi \cup \eta|}(|\varphi_x|=t)=\frac{\cosh(\beta H_x t)}{Z_{H_x}} \mathrm{d}\rho_+(t),
\]
where
\[
Z_{H_x}\coloneq\int \cosh(\beta H_x t)\,\mathrm{d}\rho_+(t).
\]
Since $\cosh\geq1$, we have $Z_{H_x}\geq1$, and since
$\cosh r\leq e^r$ for $r\geq0$,
\[
\mathrm{d}\nu_{\{x\},\beta,\rho}^{|\xi \cup \eta|}(|\varphi_x|=t)
\leq e^{\beta H_x t}\,\mathrm{d}\rho_+(t).
\]

Note that if $t\geq \varepsilon H_x$, then 
\[
\beta H_x t\leq a t^2.
\]
Let $u\geq0$. Then
\begin{align*}
    \nu_{\{x\},\beta,\rho}^{|\xi \cup \eta|}
    \bigl(|\varphi_x|\geq \varepsilon H_x+u\bigr)
    \leq
    \int_{\varepsilon H_x+u}^{\infty}
    e^{\beta H_x t}\,\mathrm{d}\rho_+(t) 
    \leq
    \int_{\varepsilon H_x+u}^{\infty}
    e^{a t^2}\,\mathrm{d}\rho_+(t) 
    \leq
    \int_{u}^{\infty}
    e^{a t^2}\,\mathrm{d}\rho_+(t).
\end{align*}
Now the latter integral is either at most $1$, in which case it coincides with $\mu_a([u,\infty))\leq \pi_a([u,\infty))$, or it is larger than $1$, in which case $\mu_a([u,\infty))=\pi_a([u,\infty))=1$. In both cases, 
$\nu_{\{x\},\beta,\rho}^{|\xi \cup \eta|} \bigl(|\varphi_x|\geq \varepsilon H_x+u\bigr)\leq \pi_a([u,\infty))$.
This proves that
\[
\nu_{\{x\},\beta,\rho}^{|\xi \cup \eta|}(|\cdot|)
\preceq \varepsilon H_x+\pi_a,
\]
as desired.
\end{proof}

The above lemma seems to suggest that the massive Laplacian of $|\varphi|$, appropriately defined, is stochastically dominated by a product measure. Applying the inverse of the massive Laplacian to this product measure, we end up with a field that has exponential decay of correlations, and one might hope to use that to prove that the latter field is stochastically dominated by another product measure. While our proof of Theorem~\ref{thm:product domination} does not follow this strategy, it is inspired by such ideas.

Our strategy involves working with the heat-bath Glauber dynamics for
$\nu^{\xi}_{\Lambda,\beta,\rho}$, which we couple with an auxiliary
discrete-time Markov chain. We work with the
discrete-time Glauber dynamics, which can be made aperiodic by making the Markov chain lazy. Below, to simplify the notation, we only observe the Markov chain at the times it makes a jump.

Let $(x_n)_{n\geq 1}$ be a sequence of i.i.d.\ vertices of $\Lambda$ chosen uniformly at random. We define $\varphi^{(n)}, V^{(n)} \in \mathbb{R}^{\Lambda}$ for every $n \geq 0$ as follows. For $n=0$, we let $\varphi^{(0)}=V^{(0)}=0$. For $n\geq 1$, if $y\neq x_n$ we set $\varphi^{(n)}_y=\varphi^{(n-1)}_y$, and if $y=x_n$, 
we sample $\varphi^{(n)}_y \sim \nu_{\Lambda,\beta,\rho}^{\xi}( \ \cdot \mid \varphi_z=\varphi^{(n-1)}_z \, \forall \, z\in \Lambda\setminus \{y\})$ and $Y_n\sim \pi_a$ from a monotone coupling such that almost surely, 
\[
|\varphi^{(n)}_y|\leq \varepsilon\sum_{\substack{z\in\Lambda\\z\sim y}}|\varphi^{(n-1)}_z| + \varepsilon\sum_{\substack{z \in \partial^{\mathrm{ext}} \Lambda \\ z \sim y}} |\xi_z|+Y_n,
\]
with $Y_n$ independent from the past and from $(x_n)_{n\geq 1}$. The latter is possible by Lemma~\ref{lem: laplacian domination} and Strassen's theorem~\cite{Strassen1965}. 

Then for every $n \geq 1$ we set
\[
V^{(n)}_y= \begin{cases}\eta_y+Y_n+\varepsilon\sum_{z\sim y}V^{(n-1)}_z & \text{if } y = x_n,\\
V^{(n-1)}_y & \text{if } y \neq x_n,
\end{cases}
\]
where 
\begin{equation*}
\eta_y\coloneq\varepsilon\sum_{\substack{z \in \partial^{\mathrm{ext}} \Lambda \\ z \sim y}}|\xi_z|.
\end{equation*}

Let us first record some basic properties of $V^{(n)}$. 

\begin{lemma}\label{lem: phi V coupling}
For every $n\geq 0$, almost surely
\[
|\varphi^{(n)}|\leq V^{(n)}.
\]
\end{lemma}
\begin{proof}
We prove the statement by induction on $n$. For $n=0$, this is immediate. Assume it is true for $n-1$. If $y\neq x_n$, then 
\[
|\varphi^{(n)}_y|=|\varphi^{(n-1)}_y|\leq V^{(n-1)}_y=V^{(n)}_y
\]
by the induction hypothesis.
If $y=x_n$, by the induction hypothesis and the definitions,
\[
V^{(n)}_y= \eta_y+Y_n+\varepsilon\sum_{z\sim y}V^{(n-1)}_z\geq \eta_y+Y_n+\varepsilon\sum_{z\sim y}|\varphi^{(n-1)}_z|\geq |\varphi^{(n)}_y|.
\]
\end{proof}

Let $c^{(n,m)}=c^{(n,m)}(\varepsilon)\in [0,\infty)^{\Lambda}$ be defined for every $1\leq m\leq n$ as follows. Set $c^{(m,m)}_y=\mathbbm{1}_{y=x_m}$ and for $n>m$, $c^{(n,m)}_y=c^{(n-1,m)}_y$ if $y\neq x_n$ and $c^{(n,m)}_y=\varepsilon \sum_{z\sim y} c^{(n-1,m)}_z$ if $y=x_n$. 

\begin{lemma}\label{lem: V^n compact form}
For every $n\geq 1$ and every $y\in \Lambda$,
\[
V^{(n)}_y=\sum_{m=1}^n c_y^{(n,m)}\eta_{x_m}+\sum_{m=1}^n c_y^{(n,m)} Y_m.
\]
\end{lemma}
\begin{proof}
We prove this by induction on $n$. The statement is true for $n=1$ by the definition of $V^{(1)}$ and $c^{(1,1)}$. Assume it is true for $n-1$. If $y\neq x_n$, then by definition of $V^{(n)}$ and our induction hypothesis
\[
V^{(n)}_y=\sum_{m=1}^{n-1} c_y^{(n-1,m)}\eta_{x_m}+\sum_{m=1}^{n-1} c_y^{(n-1,m)} Y_m.
\]
Since $c^{(n,m)}_y = c^{(n-1,m)}_y$ and $c^{(n,n)}_y = 0$, we get 
\[
V^{(n)}_y=\sum_{m=1}^n c_y^{(n,m)}\eta_{x_m}+\sum_{m=1}^{n} c_y^{(n,m)} Y_m.
\]
If $y=x_n$, then by the definition of $V^{(n)}$ and the induction hypothesis,
\[
V^{(n)}_y=c^{(n,n)}_y\eta_{x_n}+ c^{(n,n)}_y Y_n+ \varepsilon\sum_{m=1}^{n-1}\sum_{z\sim y} c_z^{(n-1,m)} (\eta_{x_m} + Y_m).
\]
The right-hand side coincides with $\sum_{m=1}^n c_y^{(n,m)}\eta_{x_m}+\sum_{m=1}^{n} c_y^{(n,m)} Y_m$, which completes the proof.    
\end{proof}

Let $A$ be the adjacency matrix of the subgraph of $G$ induced by $\Lambda$, and for $0 < \varepsilon<1/D$, define
the matrix 
\[
W=W(\varepsilon)\coloneq\sum_{k\geq0}\varepsilon^k A^k.
\]
Below, we collect some useful bounds that relate $c^{(n,m)}$ to $W$.

\begin{lemma}\label{lem: c_y bound}
Let $0<\varepsilon<1/D$. For every $n\geq m\geq 1$ and every $y\in \Lambda$,   
\[
c_y^{(n,m)}\leq W_{x_m,y}.
\]
\end{lemma}
\begin{proof}
We fix $m \geq 1$ and prove the statement by induction on $n$. First note that
\[
c_y^{(m,m)}=\mathbbm 1_{\{y=x_m\}}\leq W_{x_m,y}.
\]
Now assume that $n > m$ and $c_y^{(n-1,m)} \leq W_{x_m,y}$. If $y\neq x_n$, then
\[
c_y^{(n,m)}=c_y^{(n-1,m)}\leq W_{x_m,y}.
\]
If $y=x_n$, by the induction hypothesis and the fact that $W=I+\varepsilon A W$,
\[
c_y^{(n,m)}
=\varepsilon\sum_{z\sim y}c_z^{(n-1,m)} \leq
\varepsilon\sum_{z\sim y}W_{x_m,z} 
\leq\mathbbm 1_{\{y=x_m\}}+\varepsilon\sum_{z\sim y}W_{x_m,z}
=W_{x_m,y}.
\]    
\end{proof}

\begin{lemma}\label{lem: c_y q bound}
Let $q\in (0,1]$ and $0<\varepsilon<D^{-1/q}$. For every $\xi \in [0,\infty)^{\Lambda}$,  every $n\geq 1$, and every $y\in \Lambda$,
\[
\sum_{m=1}^n (c^{(n,m)}_y)^q\xi_{x_m}\leq (W(\varepsilon^q)\xi)_y\leq \sum_{z\in \Lambda}\sum_{k\geq d_G(y,z)} \varepsilon^{qk} D^k \xi_z.  
\]    
\end{lemma}
\begin{proof}

Let us first prove by induction that $\sum_{m=1}^n (c^{(n,m)}_y)^q\xi_{x_m}\leq (W(\varepsilon^q)\xi)_y$. Indeed, the bound clearly holds for $n=1$. Assume that it holds for $n-1$. If $y\neq x_n$, then 
\[
\sum_{m=1}^n (c^{(n,m)}_y)^q\xi_{x_m}=\sum_{m=1}^{n-1} (c^{(n-1,m)}_y)^q\xi_{x_m}\leq (W(\varepsilon^q)\xi)_y\] 
by the induction hypothesis. If $y=x_n$, then using the inequality $(a+b)^q\leq a^q+b^q$ for $a,b\geq 0$,
\begin{align*}
\sum_{m=1}^n (c_y^{(n,m)})^q\xi_{x_m}&= \xi_{x_n}+
\sum_{m=1}^{n-1}(c_x^{(n,m)})^q\xi_{x_m}= \xi_{x_n}+
\sum_{m=1}^{n-1}\Big(\varepsilon \sum_{\substack{z\in\Lambda\\z\sim y}} c_z^{(n-1,m)}\Big)^q\xi_{x_m}\\
&\leq\xi_{x_n}+\varepsilon^q
\sum_{\substack{z\in\Lambda\\z\sim y}} \sum_{m=1}^{n-1} (c_z^{(n-1,m)})^q\leq \xi_{x_n}+\varepsilon^q
\sum_{\substack{z\in\Lambda\\z\sim y}} (W(\varepsilon^q)\xi)_z= (W(\varepsilon^q)\xi)_y.
\end{align*}
The bound 
\[
(W(\varepsilon^q)\xi)_y\leq \sum_{z\in \Lambda}\sum_{k\geq d_G(y,z)} \varepsilon^{qk} D^k \xi_z  
\]
follows from the fact that we need at least $d_G(y,z)$ steps to go from $z$ to $y$ and the fact that there are at most $D^k$ walks of length $k$ starting from $z$.    
\end{proof}

We now proceed to the proof of Theorem~\ref{thm:product domination}. Our strategy is to work with the coupled Markov chains $(\varphi^{(n)},V^{(n)})$ and show that for each $n$, $V^{(n)}$ is stochastically dominated by a common product measure. To prove the latter, we will use the decomposition of Lemma~\ref{lem: V^n compact form}. We first show that each term $c^{(n,m)}Y_m$ is stochastically dominated by an inhomogeneous product measure with all Gaussian moments. While this might seem counterintuitive at first glance due to the fact that the field $c^{(n,m)}Y_m$ is highly correlated, the fast spatial decay of the coefficients $c^{(n,m)}$ coming from Lemma~\ref{lem: c_y bound} allows us to compare the field to an independent one. Lemma~\ref{lem: c_y q bound} then provides us with the necessary integrability conditions, so that when we sum over $m$ we obtain another product measure with all Gaussian moments. 

\begin{proof}[Proof of Theorem~\ref{thm:product domination}]
Let $\varepsilon>0$ be a constant to be determined and set $a = \beta/\varepsilon$. We work with $(\varphi^{(n)},V^{(n)})$. Note that the invariant distribution of $\varphi^{(n)}$ is $\nu^{\xi}_{\Lambda,\beta,\rho}$, and by classical results, $\varphi^{(n)}$ converges to its invariant distribution as $n\to\infty$, see e.g.\ \cite{RT96}. Since $|\varphi^{(n)}|\leq V^{(n)}$ almost surely by Lemma \ref{lem: phi V coupling}, it suffices to prove that $V^{(n)}$ is stochastically dominated by a common product measure.

Recall Lemma~\ref{lem: V^n compact form}, and note that we can use Lemma~\ref{lem: c_y q bound} to get
\[
\sum_{m=1}^n c_y^{(n,m)}\eta_{x_m}\leq \kappa_y,
\]
provided $\varepsilon$ is chosen small enough.
It remains to handle
\[
R^{(n)}_y\coloneq\sum_{m=1}^n c_y^{(n,m)}Y_m.
\] 
We will prove that conditionally on $(x_n)_{n\geq 1}$, $R^{(n)}$ is stochastically dominated by the same product measure, which is enough to conclude. Note that conditioning on $(x_n)_{n\geq 1}$ fixes the value of $c^{(n,m)}$. 

We will use the following decomposition. For every $x\in \Lambda$ and $n \geq m \geq 1$, define 
\[
V_x^{(n,m)}\coloneq c_x^{(n,m)}Y_m,
\]
so that 
$R^{(n)}=\sum_{m=1}^n V^{(n,m)}$.
Let $(Y_x^{(n,m)})_{x\in\Lambda,\ 1\leq m\leq n}$ be independent random variables with law $\pi_a$, independent of everything else and define
\[
G_x^{(n,m)}\coloneq \ell^{(n,m)}\sqrt{c_x^{(n,m)}}\,Y_x^{(n,m)},
\]
where
\[
\ell^{(n,m)}\coloneq
\sum_{y\in\Lambda}\sqrt{c_y^{(n,m)}}.
\]
We claim that for every $1\leq m\leq n$, conditionally on $(x_n)_{n\geq 1}$,
\begin{equation}\label{eq: V G domination}
V^{(n,m)}\preceq G^{(n,m)}.
\end{equation}

To show the above stochastic domination, let $\mathcal E\subset(\mathbb R^+)^{\Lambda}$ be an increasing event. We will show that 
\[
\mathbb{P}(V^{(n,m)}\in \mathcal{E} \mid (x_n)_{n\geq 1}) \leq \mathbb{P}(G^{(n,m)}\in \mathcal{E} \mid (x_n)_{n\geq 1}).
\]
If $\ell^{(n,m)}=0$, then both $V^{(n,m)}$ and $G^{(n,m)}$ are equal to $0$, and there is nothing to prove, so let us assume that $\ell^{(n,m)}\neq 0$ and that $\mathbb{P}(V^{(n,m)}\in \mathcal{E} \mid (x_n)_{n\geq 1}) > 0$.
First notice that 
\[
\{V^{(n,m)}\in \mathcal{E}\}= \{(c^{(n,m)}Y_m)_{y\in \Lambda}\in \mathcal{E}\} = \{Y_m \in I_{n,m}\}
\]
for $I_{n,m}$ an interval of the form $[r, \infty)$ or $(r, \infty)$ for some $r \geq 0$.
We may assume that $I_{n,m}=[r,\infty)$ since $Y_m$ is a continuous random variable. Then $(c_{x}^{(n,m)}r)_{x\in \Lambda}\in\mathcal E$,
thus
\begin{align*}
    \mathbb P(G^{(n,m)}\in\mathcal E \mid (x_n)_{n\geq 1})
    &\geq\prod_{x\in \Lambda} \mathbb P\left(\ell^{(n,m)}\sqrt{c^{(n,m)}_{x}}Y^{(n,m)}_{x}\geq c^{(n,m)}_x r
    \Bigm| (x_n)_{n\geq 1}\right) \\
    &=\prod_{x\in \Lambda}\mathbb P\left(Y^{(n,m)}_{x}\geq \frac{\sqrt{c^{(n,m)}_{x}}}{\ell^{(n,m)}}r \biggm| (x_n)_{n\geq 1}\right).
\end{align*}
To bound the latter, let $F(t)\coloneq\pi_a([t,\infty))$
and note that for every $t\geq 0$ and every $0\leq \lambda \leq 1$ we have that 
\begin{equation}
\label{eq:splitting property}
F(\lambda t)=e^{-(\lambda t)^2 L_1(\lambda t)}\geq e^{-\lambda t^2 L_1(t)}=F(t)^{\lambda}.
\end{equation}  
Set $\lambda_x\coloneq\sqrt{c^{(n,m)}_{x}}/\ell^{(n,m)}$.
Then $\lambda_x\geq0$ and
\[
\sum_{x\in\Lambda}\lambda_x=1.
\]
Using \eqref{eq:splitting property},
\begin{equation*}
    \prod_{x\in\Lambda}
    \mathbb P\left(Y^{(n,m)}_x\geq \lambda_x r
    \mid (x_n)_{n\geq 1} \right)=\prod_{x\in\Lambda}
    F(\lambda_x r)       
   \geq \prod_{x\in\Lambda}
   F(r)^{\lambda_x}= F(r)=\mathbb{P}(Y_m \in I_{n,m} \mid (x_n)_{n\geq 1}).
\end{equation*}
This implies
\[
\mathbb P(G^{(n,m)}\in\mathcal E \mid (x_n)_{n\geq 1}) \geq
\mathbb P(V^{(n,m)}\in\mathcal E \mid (x_n)_{n\geq 1}).
\]
Hence, \eqref{eq: V G domination} holds,
as desired.

Considering fixed $n$ and conditioning on $(x_k)_{k\geq 1}$, the fields $(V^{(n,m)})_{1 \leq m \leq n}$ are independent of each other and
the fields $(G^{(n,m)})_{1 \leq m \leq n}$ are by construction independent of each other, so from \eqref{eq: V G domination} we get that
\[
R^{(n)}=\sum_{m=1}^n V^{(n,m)}
\preceq \sum_{m=1}^n G^{(n, m)} \eqcolon G^{(n)}.
\]

We now construct a probability measure $\mu_*$ with all Gaussian moments such that for any $y \in \Lambda$ and $n \geq 1$, the law of $G^{(n)}_y$ conditional on $(x_n)_{n \geq 1}$ is stochastically dominated by $\mu_*$. Since the law of $G^{(n)}$ conditional on $(x_n)_{n \geq 1}$ is a product measure, this will enable us to conclude that, conditionally on $(x_n)_{n \geq 1}$, $G^{(n)}$ is dominated by $(\mu_{*})^{\otimes\Lambda}$.

Set $S_x\coloneq\sum_{m=1}^n \ell^{(n,m)} (c^{(n,m)}_x)^{1/4}$. Then a union bound gives that for every $t\geq 0$,
\begin{align*}
\mathbb{P}(G^{(n)}_x\geq t \mid (x_n)_{n\geq 1})&\leq \sum_{m=1}^n \mathbb{P}\left(\ell^{(n,m)}\sqrt{c^{(n,m)}_x}Y^{(n,m)}_{x}\geq \frac{\ell^{(n,m)} (c^{(n,m)}_x)^{1/4}}{S_x}t \biggm| (x_n)_{n\geq 1}\right)\\ &=\sum_{m=1}^n \mathbb{P}\left(Y^{(n,m)}_x\geq \frac{1}{(c^{(n,m)}_x)^{1/4} S_x}t \biggm| (x_n)_{n\geq 1}\right).
\end{align*}
We claim that by choosing $\varepsilon>0$ to be small enough, we can make 
\begin{equation}\label{eq: c root bound}
\sum_{m=1}^n (c^{(n,m)}_x)^{1/4} S_x\leq 2, \quad \sum_{m=1}^n (c^{(n,m)}_x)^{1/2} S^2_x\leq 2.
\end{equation}
Indeed, by Lemma~\ref{lem: c_y bound},
\[
\ell^{(n,m)}\leq \sum_{y\in \Lambda} \sqrt{W_{x_m,y}},
\]
which, using the bound 
\[
W_{u,v}\leq \sum_{k\geq d_G(u,v)} \varepsilon^k D^k,
\] 
implies that $\limsup_{\varepsilon\to 0}\ell^{(n,m)}\leq 1$. By Lemma~\ref{lem: c_y q bound}, both $\sum_{m=1}^n (c^{(n,m)}_x)^{1/4}$ and $\sum_{m=1}^n (c^{(n,m)}_x)^{1/2}$ tend to $1$ as $\varepsilon\to 0$. Combining these two facts, we obtain \eqref{eq: c root bound}.

For every $t\geq 0$, by \eqref{eq: c root bound}
\begin{align*}
 \mathbb{P}(G^{(n)}_x\geq t \mid (x_n)_{n\geq 1})&\leq \sum_{m=1}^n \exp\left(-\tfrac{t^2}{(c^{(n,m)}_x)^{1/2} S^2_x}L_1\big(\tfrac{t}{(c^{(n,m)}_x)^{1/4} S_x}\big)\right) \\ &\leq \sum_{m=1}^n\exp\left(-\tfrac{t^2}{(c^{(n,m)}_x)^{1/2} S^2_x}L_1(t/2)\right) \\& \leq \sum_{m=1}^n \exp\left(-\tfrac{t^2}{2(c^{(n,m)}_x)^{1/2} S^2_x} L_1(t/2)\right) \exp\left(-\tfrac{t^2}{4}L_1(t/2)\right).
\end{align*}
This implies that for every $t\geq 0$ large enough so that $e^{-\tfrac{t^2}{4} L_1(t/2)}\leq 1/4$,
\begin{align*}
 \mathbb{P}(G^{(n)}_x\geq t \mid (x_n)_{n\geq 1})& \leq \sum_{m=1}^n \exp\left(-\tfrac{1}{2(c^{(n,m)}_x)^{1/2} S^2_x}\right) \exp\left(-\tfrac{t^2}{4}L_1(t/2)\right)
 \\& \leq \sum_{m=1}^n (c^{(n,m)}_x)^{1/2} S^2_x \, \exp\left(-\tfrac{t^2}{4}L_1(t/2)\right)\\& \leq 2 \exp\left(-\tfrac{t^2}{4}L_1(t/2)\right)
 \\& \leq \exp\left(-\tfrac{t^2}{8}L_1(t/2)\right),
\end{align*}
where we used the fact that $e^{-1/s}\leq s/2$ for every $s>0$ and \eqref{eq: c root bound}.
This gives that conditionally on $(x_n)_{n\geq 1}$, $G^{(n)}$ is stochastically dominated by the product of $\mu_*$, where $\mu_*$ is a probability measure on $\mathbb R^+$ defined by its tails: for every $t\geq 0$, $\mu_*([t,\infty))=e^{-t^2 L_2(t)}$, where $L_2(t)=0$ if $e^{-\tfrac{t^2}{4} L_1(t/2)} \geq 1/4$, 
\begin{equation}\label{eq: L_2 def}
L_2(t)=\tfrac{L_1(t/2)}{8}     
\end{equation}
if $e^{-\tfrac{t^2}{4} L_1(t/2)} \leq 1/8$, and $L_2(t)$ is an arbitrary non-decreasing continuous interpolation such that $L_2(t)\leq L_1(t/2)/8$ if $1/8 < e^{-\tfrac{t^2}{4} L_1(t/2)}< 1/4$. 

This also implies that, conditionally on $(x_n)_{n\geq 1}$, $R^{(n)}$ is stochastically dominated by the product of $\mu_*$.
Then the theorem follows from averaging over $(x_n)_{n\geq 1}$. This completes the proof.
\end{proof}

\subsection{Alternative construction of the plus measure}\label{sec: alternative construction}

One potential disadvantage of the construction in Proposition \ref{prop: maximal b.c} is that it relies on growing boundary conditions, so the finite-volume measures are not regular up to the boundary. In this section, we use the stochastic domination by a product measure of Theorem~\ref{thm:product domination} to give alternative definitions of the plus measure at finite volume that are regular up to the boundary and satisfy a monotonicity property in the volume with gaps, which is reminiscent of the corresponding property of the plus measure for the Ising model.

We give two constructions, one with random boundary conditions, the other by making the single-site measure depend on the vertex. As in the statement of Theorem \ref{thm:product domination}, we assume $\rho$ is an even probability measure on $\mathbb R$ satisfying \eqref{eq: single-site n=2} and we work with nearest-neighbour interactions on a graph $G=(V,E)$ of degree bounded by some constant $D>0$.

Let $\zeta = \zeta(\beta)$ be the probability measure on $[0, \infty)$ given by Proposition~\ref{prop: Lambda monotonicity} below.
Let us introduce the measure $\tilde{\nu}^0_{\Lambda,\beta,\rho}$ defined on any bounded measurable function $g: \mathbb{R}^{\Lambda} \rightarrow \mathbb{R}$ as
    \begin{equation*}
        \tilde{\nu}^0_{\Lambda,\beta,\rho}[g] = \int_{\xi \in \mathbb{R}^{\partial \Lambda}} \langle g^\xi\rangle_{\Lambda \setminus \partial \Lambda, \beta, \rho}^{\xi} \mathrm{d} \nu_{\partial \Lambda, 0, \zeta}^{0}(\xi),
    \end{equation*}
where $g^\xi(\varphi) = g(\varphi \cup \xi)$ with $\varphi \cup \xi$ the configuration equal to $\xi$ on $\partial \Lambda$ and equal to $\varphi$ on $\Lambda \setminus \partial \Lambda$.
In words, $\tilde{\nu}^0_{\Lambda,\beta,\rho}$ is a measure with random boundary conditions on $\partial \Lambda$ sampled from the product measure $\nu_{\partial \Lambda, 0, \zeta}^{0}$. 

\begin{proposition}\label{prop:random boundary conditions}
Let $G=(V,E)$ be a connected graph of degree bounded by $D>0$. Consider nearest-neighbour interactions on $G$ and let $\rho$ be an even probability measure. For every $\beta\geq 0$,
\[
\tilde{\nu}^0_{\Lambda,\beta,\rho}\rightarrow \nu^{+}_{\beta, \rho} \text{ as } \Lambda \nearrow V.
\]
\end{proposition}

For our second construction, we define a single-site measure $\tilde{\rho}_{x, \Lambda}$ by
\begin{equation*}
\mathrm{d}\tilde{\rho}_{x, \Lambda}(u) = \begin{cases}
        \mathrm{d} \zeta(u) & \text{if } x\in \partial \Lambda,\\
        \mathrm{d} \rho(u) & \text{otherwise.}
    \end{cases}
\end{equation*}
    
\begin{proposition}
\label{prop: alternative plus measure}
Let $G=(V,E)$ be a connected graph of degree bounded by $D>0$. Consider nearest-neighbour interactions on $G$ and let $\rho$ be an even probability measure. For every $\beta\geq 0$,
\begin{equation*}
        \nu_{\Lambda, \beta, \tilde{\rho}}^{0} \rightarrow \nu^{+}_{\beta, \rho} \text{ as } \Lambda \nearrow V.
    \end{equation*}
\end{proposition}

To prove Proposition \ref{prop: alternative plus measure}, we will need to use monotonicity in $\beta$ of the measures $\nu_{(\Lambda | \partial \Lambda), \beta, \tilde{\rho}}^{0}$, which is provided by the following lemma.
\begin{lemma}
    \label{Lemma: beta monotonicity}
    For any $\Lambda \Subset V$ and $\beta' \geq \beta \geq 0$,
    \begin{equation*}
        \nu_{(\Lambda | \partial \Lambda), \beta, \tilde{\rho}}^{0} \preceq \nu_{(\Lambda | \partial \Lambda), \beta', \tilde{\rho}}^{0}.
    \end{equation*}
\end{lemma}

\begin{proof}
    Let $A$ be an increasing event depending only on vertices in $\partial \Lambda$. Without loss of generality, assume that $\nu_{(\Lambda | \partial \Lambda), \beta, \tilde{\rho}}^{0}[A] > 0$. Differentiating with respect to $\beta$ we obtain
    \begin{equation*}
        \frac{\mathrm{d} \nu_{(\Lambda | \partial \Lambda), \beta, \tilde{\rho}}^{0}[A]}{\mathrm{d} \beta} = \sum_{xy \in E} \langle \mathbbm{1}_{A} \varphi_x \varphi_y \rangle_{\Lambda, \beta, \tilde{\rho}}^{0} - \langle \mathbbm{1}_{A} \rangle_{\Lambda, \beta, \tilde{\rho}}^{0} \langle \varphi_x \varphi_y \rangle_{\Lambda, \beta, \tilde{\rho}}^{0},
    \end{equation*}
    so it suffices to show that $\langle \varphi_x \varphi_y \mid A \rangle_{\Lambda, \beta, \tilde{\rho}}^{0} - \langle \varphi_x \varphi_y \rangle_{\Lambda, \beta, \tilde{\rho}}^{0} \geq 0$
    for all $xy \in E$.
    
    If $x, y \in \partial \Lambda$ then $\varphi_x, \varphi_y \geq 0$, so $\varphi_x\varphi_y$ is an increasing function and the desired inequality follows from the FKG inequality.
    Now suppose that $x, y \in \Lambda \setminus \partial \Lambda$. We use the domain Markov property and the fact that $\tilde{\rho} = \rho$ on $\Lambda \setminus \partial \Lambda$ to get
    \begin{equation*}
        \langle \varphi_x \varphi_y \mid A \rangle_{\Lambda, \beta, \tilde{\rho}}^{0}
        = \int_{\eta \in (\mathbb{R}^{+})^{\partial \Lambda}} \langle \varphi_x \varphi_y \rangle_{\Lambda \setminus \partial \Lambda, \beta, \rho}^{\eta} \mathrm{d} \nu_{(\Lambda | \partial \Lambda), \beta, \tilde{\rho}}^{0} [\eta \mid A].
    \end{equation*}
    Since $\nu_{(\Lambda | \partial \Lambda), \beta, \tilde{\rho}}^{0} [\ \cdot \mid A] \succeq \nu_{(\Lambda | \partial \Lambda), \beta, \tilde{\rho}}^{0} [\ \cdot \ ]$ by the FKG inequality, and $\langle \varphi_x \varphi_y \rangle_{\Lambda \setminus \partial \Lambda, \beta, \rho}^{\eta}$ is an increasing function of $\eta$ (which follows by differentiating and using Griffiths' inequality), the right hand side above is at least $\langle \varphi_x \varphi_y \rangle_{\Lambda, \beta, \tilde{\rho}}^{0}$.
    
    It remains to consider the case when the edge $xy$ has one endpoint in $\partial \Lambda$ and the other endpoint in $\Lambda \setminus \partial \Lambda$. If $x \in \Lambda \setminus \partial \Lambda$ and $y \in \partial \Lambda$, then
    \begin{equation*}
        \langle \varphi_x \varphi_y \mid A \rangle_{\Lambda, \beta, \tilde{\rho}}^{0}
        = \int_{\eta \in (\mathbb{R}^{+})^{\partial \Lambda}} \eta_y \langle \varphi_x  \rangle_{\Lambda \setminus \partial \Lambda, \beta, \rho}^{\eta} \mathrm{d} \nu_{(\Lambda | \partial \Lambda), \beta, \tilde{\rho}}^{0} [\eta \mid A].
    \end{equation*}
    Proposition \ref{b.c monotonicity} together with the fact that the boundary conditions $\eta$ are positive and $\rho$ is an even measure implies that $\eta_y \langle \varphi_x  \rangle_{\Lambda \setminus \partial \Lambda, \beta, \rho}^{\eta}$ is an increasing function of $\eta$, so the right hand side is again at least $\langle \varphi_x \varphi_y \rangle_{\Lambda, \beta, \tilde{\rho}}^{0}$.
\end{proof}

As a corollary, we obtain the following stochastic domination.
\begin{corollary}
    \label{cor: alternative measures}
    For any $\Lambda \Subset V$ and $\beta \geq 0$, \[\tilde{\nu}^0_{\Lambda,\beta,\rho} \preceq \nu^0_{\Lambda,\beta,\tilde{\rho}} \quad \text{and} \quad \tilde{\nu}^0_{\Lambda,\beta,\rho}(|\cdot|) \preceq \nu^0_{\Lambda,\beta,\tilde{\rho}} (|\cdot|).\]
\end{corollary}
\begin{proof}
    By Lemma \ref{Lemma: beta monotonicity}, we have $\tilde{\nu}_{(\Lambda | \partial \Lambda), 0, \rho}^{0} = \nu_{(\Lambda | \partial \Lambda), 0, \tilde{\rho}}^{0} \preceq \nu_{(\Lambda | \partial \Lambda), \beta, \tilde{\rho}}^{0}$, and stochastic domination in the whole of $\Lambda$ follows from the domain Markov property and monotonicity in boundary conditions.
\end{proof}

We now show that for appropriately defined $\zeta$, the measures $\tilde{\nu}^0_{\Lambda,\beta,\rho}$ and $\nu^0_{\Lambda,\beta,\tilde{\rho}}$ satisfy monotonicity in volume with gaps.

\begin{proposition}
\label{prop: Lambda monotonicity}
Let $G=(V,E)$ be a graph of bounded degree. For every $\beta_0>0$ there exists a probability measure $\zeta$ on $[0,\infty)$ such that 
\[
\zeta(e^{\lambda \varphi^2})<\infty \qquad \forall \lambda>0
\]
and $r>0$ such that the following holds. Consider $\Lambda'\subset \Lambda \Subset V$ such that $d_G(\partial \Lambda, \Lambda')>r$. Then for every $\beta\in [0,\beta_0]$, $\tilde{\nu}^0_{\Lambda',\beta,\rho}$ stochastically dominates $\nu^0_{(\Lambda | \Lambda'),\beta, \tilde{\rho}}$ and $\tilde{\nu}^0_{\Lambda',\beta,\rho}(|\cdot|)$ stochastically dominates $\nu^0_{(\Lambda | \Lambda'),\beta, \tilde{\rho}}(|\cdot|)$. In particular, we have $\tilde{\nu}^0_{\Lambda',\beta,\rho} \succeq \tilde{\nu}^0_{(\Lambda | \Lambda'),\beta,\rho}$ and $\nu^0_{\Lambda',\beta,\tilde{\rho}}\succeq \nu^0_{(\Lambda | \Lambda'),\beta,\tilde{\rho}}$ as well as the corresponding stochastic dominations for the absolute value field. 
\end{proposition}

\begin{proof}
Let $\zeta$ be a probability measure on $[0,\infty)$ defined by its tails: for every $t\geq 0$, $\zeta([t,\infty))=e^{-t^2 L(t)}$ for some non-decreasing continuous function $L:[0,\infty)\to[0,\infty)$ to be defined. It suffices to find $r>0$ such that for every $\Lambda'\subset \Lambda \Subset V$ with $d_G(\partial \Lambda, \Lambda')>r$, 
$\nu_{\partial \Lambda', 0, \zeta}^{0} =\tilde{\nu}^0_{(\Lambda' | \partial \Lambda'),\beta,\rho}$ stochastically dominates $\nu^0_{(\Lambda | \partial \Lambda'),\beta,\tilde{\rho}}(|\cdot|)$. Then the desired stochastic dominations in the whole of $\Lambda'$ follow from the domain Markov property and monotonicity in boundary conditions. By Corollary \ref{cor: alternative measures}, we also get $\tilde{\nu}^0_{\Lambda',\beta,\rho} \succeq \tilde{\nu}^0_{(\Lambda | \Lambda'),\beta,\rho}$ and $\nu^0_{\Lambda',\beta,\tilde{\rho}} \succeq \nu^0_{(\Lambda | \Lambda'),\beta,\tilde{\rho}}$ together with the corresponding statements for the absolute value field.

To this end, we argue as in the proof of Theorem~\ref{thm:product domination}, and we only highlight the necessary modifications. Let $\varepsilon > 0$ and set $a = \beta/\varepsilon$. Let $\zeta_a$ be defined from $\zeta$ by 
\[
\zeta_a([t,\infty))
\coloneq \min\left\{1,\int_t^{\infty} e^{a s^2}\,\mathrm{d}\zeta(s)\right\}.
\] 
Let $(x_n)_{n\geq 1}$ be a sequence of i.i.d.\ vertices of $\Lambda$ chosen uniformly at random, and for $y \in \Lambda$ define $c^{(n,m)}_y, \ell^{(n,m)},$ and $S_y$ as in the proof of Theorem \ref{thm:product domination}. We define $\varphi^{(n)}, V^{(n)} \in \mathbb{R}^{\Lambda}$ for every $n \geq 0$ as follows. For $n=0$, we let $\varphi^{(0)}=V^{(0)}=0$. For $n\geq 1$, if $y\neq x_n$ we set $\varphi^{(n)}_y=\varphi^{(n-1)}_y$, and if $y=x_n$, 
we sample $\varphi^{(n)}_y \sim \nu_{\Lambda,\beta,\tilde{\rho}}^{0}( \ \cdot \mid \varphi_z=\varphi^{(n-1)}_z \, \forall \, z\in \Lambda \setminus \{y\})$, $Y_n\sim \pi_a$ and $Z_n\sim \zeta_a$ from a monotone coupling such that almost surely, 
\[
|\varphi^{(n)}_y|\leq \varepsilon\sum_{\substack{z\in\Lambda\\z\sim y}}|\varphi^{(n-1)}_z|+\mathbbm{1}_{y \in \Lambda \setminus \partial \Lambda}Y_n+\mathbbm{1}_{y \in \partial \Lambda} Z_n,
\]
with $Y_n$ and $Z_n$ independent from each other, from the past and from $(x_n)_{n\geq 1}$. Then we set
\[
V^{(n)}_y= \begin{cases}\mathbbm{1}_{y \in \Lambda \setminus \partial \Lambda}Y_n+\mathbbm{1}_{y \in \partial \Lambda} Z_n+\varepsilon\sum_{z\sim y}V^{(n-1)}_z & \text{if } y = x_n,\\
V^{(n-1)}_y & \text{if } y \neq x_n.
\end{cases}
\]
We have $|\varphi^{(n)}| \leq V^{(n)}$ almost surely, and decomposing as in Lemma \ref{lem: V^n compact form} we get
\[
V_y^{(n)}=\sum_{m=1}^n
c_y^{(n,m)}\mathbbm{1}_{x_m\in \Lambda \setminus \partial \Lambda}Y_m+
\sum_{m=1}^n
c_y^{(n,m)}\mathbbm{1}_{x_m\in \partial\Lambda}Z_m.
\]

To handle the first sum in the above expression, note that $\sum_{m=1}^n
c_y^{(n,m)}\mathbbm{1}_{x_m\in \Lambda \setminus \partial \Lambda}Y_m\leq \sum_{m=1}^n
c_y^{(n,m)}Y_m$, and that in the proof of Theorem~\ref{thm:product domination} we showed that the latter is stochastically dominated by the product of some probability measures $\mu_*$ defined by 
$\mu_*([t,\infty))=e^{-t^2 L_2(t)}$ for some non-decreasing continuous function $L_2:[0,\infty)\to[0,\infty)$. Now let $B>0$ be such that 
\begin{equation}\label{eq: L def}
\exp(-t^2L_2(t/2)/8)\leq 1/2 \quad \forall t\geq B.
\end{equation}
Define $L(t)=L_2(t/2)/8$ for every $t\geq B$ and
$L(t)=0$ for every $t<B$. This fixes the definition of the measure $\zeta$.

By arguing as in the proof of Theorem~\ref{thm:product domination}, the second term $\sum_{m=1}^n
c_y^{(n,m)}\mathbbm{1}_{x_m\in \partial\Lambda}Z_m$ is stochastically dominated by
\[
\sum_{m=1}^n \ell^{(n,m)}\sqrt{c_y^{(n,m)}} \mathbbm{1}_{x_m\in \partial\Lambda}\,Z_y^{(n,m)},
\]
where $(Z_x^{(n,m)})_{x\in\Lambda,\ 1\leq m\leq n}$ are independent random variables with law $\zeta_a$.
To estimate the latter field, 
we again follow the proof of Theorem \ref{thm:product domination} and deduce that for all $\varepsilon$ small enough, there exists a probability measure $\zeta_*$ such that conditionally on $(x_n)_{n \geq 1}$, $(\sum_{m=1}^n
c_x^{(n,m)}\mathbbm{1}_{x_m\in \partial\Lambda}Z_m)_{x \in \Lambda}$ is stochastically dominated by independent random variables $(s_x E_x)_{x\in \Lambda}$ with $E_x\sim \zeta_*$ and $s_x = \sum_{m=1}^{n} (c^{(n, m)}_x)^{1/2} S_x^2 \mathbbm{1}_{x_m \in \partial \Lambda}$.
Thanks to Lemma~\ref{lem: c_y q bound} applied to $\xi_x=\mathbbm{1}_{x\in \partial \Lambda}$, we get that $\sup_{d_G(x,\partial \Lambda)\geq r} s_x$ tends to $0$ as $r\to\infty$, uniformly over $\Lambda$ and $(x_n)_{n \geq 1}$. Furthermore, $\zeta_*$ can be defined by its tails $\zeta_*([t,\infty))=e^{-t^2 L_3(t)}$ for some non-decreasing continuous function $L_3:[0,\infty)\to[0,\infty)$. By \eqref{eq: L_1 def} and \eqref{eq: L_2 def}, this function $L_3$ can be defined in terms of $L$ so that it satisfies $L_3(t)=L(t/2)/16$ for every $t\geq C$ for some constant $C>0$ large enough (that depends on the choice of $L(t)$ but this will not matter). 

With the above at hand, let $T_x$ be a random variable with law $\mu_*$ which is independent of $E_x$. Note that for every $t\geq B$ and every $x\in \Lambda$ such that $s_x\leq 1/16$ and $B/(2s_x)\geq C$, by a union bound, \eqref{eq: L def} and the expressions of $L_2$ and $L_3$ in terms of $L$,
\begin{align*}
\mathbb{P}(T_x+s_xE_x\geq t)&
\leq \mathbb{P}(T_x\geq t/2)+\mathbb{P}(s_xE_x\geq t/2)\\
&\leq \exp\left(-\tfrac{t^2}{4} L_2(t/2)\right) +\exp\left(-\tfrac{t^2}{4 s^2_x} L_3\left(\tfrac{t}{2s_x}\right)\right)\\
&\leq \exp\left(-2 t^2 L(t)\right)+\exp\left(-2t^2 L\left(\tfrac{t}{4s_x}\right)\right)\\
&\leq 2\exp\left(-2t^2 L(t)\right)\\
&\leq \exp\left(-t^2 L(t)\right)\\
&=\zeta([t,\infty)).
\end{align*}
Since the latter inequality holds trivially for $t< B$ because
$\zeta([t,\infty))=1$ in this case, the desired stochastic domination follows.
\end{proof}

\begin{remark}
For every $C>0$, the same argument shows that one can construct a measure $\zeta=\zeta(C)$ supported on $[C,\infty)$ with all Gaussian moments such that the monotonicity of Proposition~\ref{prop: Lambda monotonicity} with gaps holds for some $r$ that depends on $C$ (and the other parameters of the proposition).
\end{remark}

\begin{remark}
In the case of $P(\varphi)$ models, where $\rho$ has density with respect to Lebesgue measure equal to $e^{-P(\varphi)}$ for some even polynomial $P$ of degree $n\geq 4$ and of leading coefficient $a_n>0$, all the functions $L_i$ and $L$ appearing in the proof of Proposition~\ref{prop: Lambda monotonicity} can be chosen to be polynomials of degree $n-2$. By tuning the constants involved, $\zeta$ can be chosen to be of the form $\pi(C+|\cdot|)$ for some constant $C>0$, where $\pi$ has density $e^{-\tfrac{a_n}{2}\varphi^n}\mathrm{d}\varphi$.
\end{remark}

We now prove convergence of the measures $\nu_{\Lambda, \beta, \tilde{\rho}}^{0}$ and $\tilde{\nu}^0_{\Lambda,\beta,\rho}$ to the plus measure.
\begin{proof}[Proof of Proposition \ref{prop: alternative plus measure}]
    First note that the single-site measures $\tilde{\rho}_{x, \Lambda}$ satisfy the assumptions of Remark \ref{remark:changed single-site} with $a_{\min} = \frac{a}{2}$, 
    and that the result of Theorem \ref{Theorem: n=2} still applies in this case.
    Therefore, there exists $\tilde{C}_2 > 0$ such that for all $\Lambda \Subset V$, $\Lambda' \subset \Lambda \setminus \partial \Lambda$ and $\psi \in \mathbb{R}^{\Lambda'}$,
    \begin{equation}
        \label{eq:regularity for rho tilde n=2}
        \mathrm{d} \nu_{\Lambda, \beta, \tilde{\rho}}^{0}[\varphi|_{\Lambda'} = \psi]
        \leq e^{\tilde{C}_2 |\Lambda'|} \mathrm{d} \nu_{\Lambda', 0, \rho_{\frac{a}{2}}}^{0} [\psi].
    \end{equation}
    By the monotonicity in volume of Proposition \ref{prop: Lambda monotonicity}, $\nu_{\Lambda, \beta, \tilde{\rho}}^{0}$ converges to some measure $\tilde{\nu}$ as $\Lambda \nearrow V$. The measure $\tilde{\nu}$ is $\frac{a}{2}$-regular by \eqref{eq:regularity for rho tilde n=2} and is a Gibbs measure by the domain Markov property; hence $\tilde{\nu} \preceq \nu^{+}_{\beta, \rho}$ by Proposition~\ref{prop: maximal b.c}. It remains to prove that $\nu^{+}_{\beta, \rho} \preceq \tilde{\nu}$, which implies that $\tilde{\nu}=\nu^{+}_{\beta, \rho}$.

    Let $\Lambda_i = B_i(o)$.
    Using Theorem \ref{thm:product domination} we have for any $i \geq 0$, 
    \begin{equation}
        \label{eq: zeta domination 1}
        \nu_{(\Lambda_{3i}| \partial \Lambda_i), \beta, \rho}^{\xi^{+}} \preceq 
        \kappa(\Lambda_{3i}, \xi^{+})|_{\partial \Lambda_i} + \mu_{*}^{\otimes \partial \Lambda_i}.
    \end{equation}
    We claim that for all $i$ large enough
    \begin{equation}
        \label{eq: zeta domination 2}
        \forall x \in \partial \Lambda_i \quad \mu_* + \kappa(\Lambda_{3i}, \xi^+)_x \preceq \zeta.
    \end{equation}
    Indeed, for $x \in \partial \Lambda_i$ we have
    \[
    \kappa(\Lambda_{3i}, \xi^{+})_x=\sum_{y\in \partial^{\mathrm{ext}} \Lambda_{3i}} \frac{|\xi^{+}_y|}{D^{2 d_G(x,y)}} \leq |\partial^{\mathrm{ext}} \Lambda_{3i}| \frac{\sqrt{\log(|\Lambda_{3i+1}|)}}{D^{4i}} \leq \frac{\sqrt{(3i+1) \log(D)}}{D^i}.
    \]
    The right hand side above tends to $0$ as $i \to \infty$. Choosing $i$ large enough so that $\kappa(\Lambda_{3i}, \xi^{+})_x \leq B$ and $(t-\kappa(\Lambda_{3i}, \xi^{+})_x)^2 \geq t^2/8$ for all $t \geq B$, it follows from the definitions of $\mu_*$ and $\zeta$ that 
    $\mu_*([t- \kappa(\Lambda_{3i}, \xi^+)_x, \infty)) \leq \zeta([t, \infty))$,
    which proves \eqref{eq: zeta domination 2}.
    
    Combining \eqref{eq: zeta domination 2} with \eqref{eq: zeta domination 1} and applying Lemma \ref{Lemma: beta monotonicity}, we obtain for all $i$ large enough
    \begin{equation*}
        \nu_{(\Lambda_{3i}| \partial \Lambda_i), \beta, \rho}^{\xi^{+}} \preceq \nu_{\partial \Lambda_i, 0, \zeta}^{0} \preceq \nu_{(\Lambda_i|\partial \Lambda_i), \beta, \tilde{\rho}}^{0}.
    \end{equation*}
    Now let $\Lambda' \Subset V$ and consider an increasing event $H$ that depends only on spins inside $\Lambda'$. The domain Markov property and Proposition \ref{b.c monotonicity} together with the above stochastic domination imply that for all $i$ large enough,
    \begin{align*}
        \nu_{\Lambda_{3i}, \beta, \rho}^{\xi^{+}} [H] &= \int_{\mathbb{R}^{\partial \Lambda_i}} \nu_{\Lambda_i \setminus \partial \Lambda_i, \beta, \rho}^{\eta}[H] \mathrm{d} \nu_{(\Lambda_{3i}| \partial \Lambda_i), \beta, \rho}^{\xi^{+}}[\eta]\\
        &\leq \int_{\mathbb{R}^{\partial \Lambda_i}} \nu_{\Lambda_i \setminus \partial \Lambda_i, \beta, \rho}^{\eta}[H] \mathrm{d} \nu_{(\Lambda_{i}|\partial \Lambda_i), \beta, \tilde{\rho}}^{0}[\eta] = \nu_{\Lambda_i, \beta, \tilde{\rho}}^{0}[H].
    \end{align*}
    Taking $i \rightarrow \infty$, we have $\nu^{+}_{\beta, \rho}[H] \leq \tilde{\nu}[H]$.
\end{proof}

\begin{proof}[Proof of Proposition~\ref{prop:random boundary conditions}]
First observe that the monotonicity in volume of Proposition~\ref{prop: Lambda monotonicity} implies convergence of the measures $\tilde{\nu}^0_{\Lambda,\beta,\rho}$ to an infinite-volume limit $\tilde{\nu}$ as $\Lambda \nearrow V$. By Corollary \ref{cor: alternative measures}, we have that $\tilde{\nu}^0_{\Lambda,\beta,\rho} \preceq \nu^0_{\Lambda,\beta,\tilde{\rho}}$ for any $\Lambda \Subset V$, so $\tilde{\nu} \preceq \nu^{+}_{\beta, \rho}$ because $\nu^0_{\Lambda,\beta,\tilde{\rho}} \rightarrow \nu^{+}_{\beta, \rho}$ by Proposition \ref{prop: alternative plus measure}. 
The proof that $\nu^{+}_{\beta, \rho} \preceq \tilde{\nu}$ is similar to that of Proposition~\ref{prop: alternative plus measure}, and in fact, even simpler, as one does not need to use Lemma~\ref{Lemma: beta monotonicity}.
\end{proof}

\bibliographystyle{plain}
\bibliography{references}

\end{document}